\RequirePackage{fix-cm}
\documentclass[smallcondensed]{svjour3}                     % onecolumn (standard format)
\smartqed  % flush right qed marks, e.g. at end of proof
\usepackage{graphicx}
\usepackage{multirow}%
\usepackage{amsmath,amssymb,amsfonts}%
\usepackage{mathrsfs}%
\usepackage[title]{appendix}%
\usepackage{xcolor}%
\usepackage{placeins}
\usepackage{textcomp}%%
\usepackage{manyfoot}%
\usepackage{booktabs}%
\usepackage{algorithm}%
\usepackage{algorithmicx}%
\usepackage{algpseudocode}%
\usepackage{listings}%
\usepackage{subcaption}
\usepackage{comment}
\usepackage{mathtools}
\usepackage{svg}
\usepackage{url}
\usepackage[colorlinks=true, linkcolor=blue, citecolor=blue, urlcolor=blue]{hyperref}

\newcommand{\Par}[1]{\left( #1 \right)}
\newcommand{\dgi}[0]{g_j-g_i-h_i(x_j-x_i)}
\newcommand{\dfi}[0]{f_j-f_i-g_i(x_j-x_i)-\frac{h_i}{2}(x_j-x_i)^2}

\newcommand{\dxi}[0]{x_j-x_i}
\newcommand{\dhhi}[0]{h_j-h_i}
\newcommand{\dg}[0]{T^g_{12}}
\newcommand{\df}[0]{T^f_{12}}
\newcommand{\dx}[0]{\Delta x_{12}}
\newcommand{\dhh}[0]{\Delta h_{12}}
\newcommand{\quadr}[0]{S=\{(x_i,f_i,g_i,h_i)\}_{i\in[N]}}
\newcommand{\quadrnof}[0]{S=\{(x_i,g_i,h_i)\}_{i\in[N]}}
\newcommand{\yun}[0]{y_1}
\newcommand{\ydeux}[0]{y_2}
\newcommand{\zun}[0]{z_1}
\newcommand{\zdeux}[0]{z_2}

\newcommand{\ra}[1]{\renewcommand{\arraystretch}{#1}}

\newcommand{\FLip}[1]{\mathcal{F}_{#1}} 
\newcommand{\GLip}[1]{\mathcal{G}_{#1}} 
\newcommand{\HLip}[1]{\mathcal{H}_{#1}} 
\newcommand{\R}{\mathbb{R}}
\newcommand{\bR}{\bar{\mathbb{R}}}
\newcommand{\bC}{\bar{\mathcal{C}}}
\newcommand{\HH}{\mathcal{H}}

\newcommand{\F}{\mathcal{F}}

\newcommand{\C}{\mathcal{C}}

\newcommand{\Sc}{\mathcal{S}_{M,+}}
\newcommand{\QSc}{\mathcal{T}_{M,+}}
\newcommand{\ba}{{\beta(\alpha)}}
\newcommand{\figureH}[1]{
    \ifthenelse{\equal{\displayfigures}{false}}
      {% False case
       \textcolor{blue}{\textit{hidden figure}}
      }
      {% True case
      #1 
      }
}
\newcommand{\displayfigures}{true}
\def\mycolor{blue}

\newtheorem{assump}{Assumption}

\begin{document}

\title{Performance Estimation of second-order optimization methods on classes of univariate functions%\thanks{Grants or other notes
}

\titlerunning{Performance Estimation of second-order optimization methods}        % if too long for running head

\author{Anne Rubbens$^{1\dagger}$ \and
        Nizar Bousselmi$^{1\dagger}$\and
        Julien M. Hendrickx$^{1}$ \and
        François Glineur$^{1,2}$
}

\authorrunning{A. Rubbens, N. Bousselmi, J. M. Hendrickx, F. Glineur} % if too long for running head

\institute{
$^{1}$ UCLouvain, ICTEAM, Louvain-la-Neuve, Belgium \\
$^{2}$ UCLouvain, CORE, Louvain-la-Neuve, Belgium  \\
\email{firstname.surname@uclouvain.be}  \\
$^\dagger$ These authors contributed equally to this work. The order of authorship was determined by a random draw. 
}

\date{Received: date / Accepted: date}

\maketitle

\begin{abstract}
We develop a principled approach to obtain exact computer-aided worst-case guarantees on the performance of second-order optimization methods on classes of univariate functions. We first present a generic technique to derive interpolation conditions for a wide range of univariate functions, and use it to obtain such conditions for generalized self-concordant functions (including self-concordant and quasi-self-concordant functions) and functions with Lipschitz Hessian (both convex and non-convex). We then exploit these conditions within the Performance Estimation framework to tightly analyze the convergence of second-order methods on univariate functions, including (Cubic Regularized) Newton's method and several of its variants. Thereby, we improve on existing convergence rates, exhibit univariate lower bounds (that thus hold in the multivariate case), and
%compare different variants of Newton's method on a fair basis, i.e., with respect to the same setting.
analyze the performance of these methods with respect to the same criteria.
\end{abstract}

\section{Introduction}\label{sect:introduction}
Consider the unconstrained minimization problem
\begin{equation}
    \min_{x\in \mathbb{R}^d} f(x),\label{eq:minprob}
\end{equation} 
where $f$ belongs to a given function class $\F$, e.g, the class of twice-differentiable functions $f:\R^d\to \R$ whose Hessian is $M$-Lipschitz continuous, i.e,
\begin{align}\label{eq:hess1}
    \|\nabla^2 f(x)-\nabla^2 f(y)\|\leq M\|x-y\|, \quad \forall x,y \in \R^d.
\end{align}
To solve \eqref{eq:minprob}, consider $N$ iterations of a black-box (oracle-based) second-order optimization method $\mathcal{M}$, e.g., Newton's method, i.e, 
\begin{align*}
    x_{k+1}=x_k- \nabla^2 f(x_k)^{-1} \nabla f(x_k).
\end{align*}
Our goal is to answer the following question.
\begin{center}
    \noindent \textbf{What is the worst-case convergence rate after $N$ iterations of method $\mathcal{M}$ on the class $\F$?
    } 
\end{center}

A worst-case convergence rate of a method on a function class is an upper bound on how quickly the method converges to a (local or global) minimum for any function in the class. This bound is said to be \emph{tight} or \emph{exact} when it is attained by a function of the class considered. Obtaining an exact rate allows understanding the possible sources of inefficiency and to tune a given method, by choosing its parameters as to minimize the rate. In addition, obtaining the exact convergence rate of different methods provides a basis for accurately comparing these methods. In this work, we extend the Performance Estimation framework \cite{drori2014performance,taylor2017smooth} that allows to automatically compute tight convergence rates of first-order methods, to the tight analysis of second-order methods applied to \emph{univariate} functions.

\subsection{Second-order optimization}
Second-order methods such as Newton's method or Cubic Regularized Newton method \cite{nesterov2006cubic}, exploit evaluations of the Hessian of the objective function. They have attracted a lot of attention due to their local quadratic convergence and importance for interior-point methods \cite{conn2000trust,nesterov1994interior} based on self-concordant functions. Several globally convergent variants of Newton's method were then also proposed on different classes of functions \cite{cartis2011adaptive,doikov2024super,hanzely2022damped,hildebrand2021optimal,ivanova2024optimal,mishchenko2023regularized,nesterov2006cubic,polyak2009regularized}. 

The potential tightness of existing convergence rates is an open question in general.
In practice, the methods often behave significantly better than predicted by the theory, either because worst-case instances are very unlikely to appear in practical problems, or because the current convergence rates are too pessimistic and can be improved.
In addition, a fair comparison between the different variants of Newton's method is difficult to perform, in the sense that the performance measures and hypotheses vary for each analysis. A tight and systematic analysis of these variants on a few well-chosen performance measures would thus allow improving on the existing bounds, provide insight into their level of pessimism, and allow for a more fair comparison. More detailed state-of-the-art about the methods we analyze in this work is provided in Section \ref{sec:results}.

\subsection{Obtaining convergence rates: the Performance Estimation Framework}
For a wide range of methods, the convergence rate after a fixed number of iterations can be computed via the Performance Estimation Problem (PEP) approach \cite{drori2014performance,taylor2017smooth}, which builds on the following observation: in principle, our main question can be reformulated as the following type of problem, where we select the distance to the minimizer as the performance measure:
\begin{equation}\label{PEP_untractable}
    \begin{aligned}
   &  && \underset{x_0, x_\star \in \R^d, f \in \F}{\max \ } ||x_N-x_\star||\\
     &\text{Method's definition:}&&\text{ s.t. } x_{k+1}=x_k-\nabla^2f(x_k)^{-1} \nabla f(x_k), \quad k=0,...,N-1 \\
     &\text{Initial condition:}&&\quad \quad||x_0-x_\star||\leq R \\
     &\text{Local optimality:}&&\quad \quad \nabla f(x_\star)=0, \nabla^2 f(x_\star) \succeq 0. 
\end{aligned}
\end{equation}
That is, we are looking for the function $f$, initial point $x_0$, and local optimizer $x_\star$ leading to the worst performance of the method $\mathcal{M}$, and hence ensure that $\mathcal{M}$ behaves better or equivalently on all functions of the class $\F$. Problem \eqref{PEP_untractable} is an infinite-dimensional problem, but \cite{drori2014performance,taylor2017smooth} have shown how it admits a finite reformulation. The idea is to replace optimization over $f\in \F$ by optimization over the set $S=\{(x_k,g_k,h_k,f_k)\}_{k=0,\ldots,N,\star}$ containing all iterates $x_k$ (and the minimizer $x_\star$) along with the corresponding gradients $g_k$, Hessians $h_k$, and function values $f_k$, under the constraint that $S$ is consistent with an actual function in $\F$. Indeed, \eqref{PEP_untractable} is equivalent to
\begin{equation}\label{PEP_untractable_bis}
    \begin{aligned}
   &  && \underset{S=\{(x_k,f_k,g_k,h_k)\}_{k=0,\ldots,N,\star}}{\max \ } ||x_N-x_\star||\\
     &\text{Method's definition:}&&\text{ s.t. } x_{k+1}=x_k-h_k^{-1} g_k, \quad k=0,...,N-1 \\
     &\text{Initial condition:}&&\quad \quad||x_0-x_\star||\leq R \\
     &\text{Local optimality:}&&\quad \quad g_\star=0, h_\star \succeq 0\\
     &\text{Consistency with $\F$:}&&\quad \quad \exists f\in \F: f(x_k)=f_k, \ \nabla f(x_k)=g_k, \nabla^2 f(x_k)=h_k, \\
     &&&\quad \quad \quad \quad \quad \quad k=0,\ldots,N,\star.
\end{aligned}
\end{equation}%
Making the last constraint in \eqref{PEP_untractable_bis} explicit requires having access to \emph{interpolation conditions} for $\F$, that is necessary and sufficient conditions to impose on a data set $S=\{(x_k,f_k,g_k,h_k)\}_{k=0,\ldots,N,\star}$ to ensure its consistency with a function in $\F$.  Solving a PEP as in \eqref{PEP_untractable_bis} with interpolation conditions ensures tightness of the resulting bound. If instead we impose necessary but not sufficient conditions on the set $S$, we obtain a relaxation of \eqref{PEP_untractable}, potentially leading to bounds that are not tight (see \cite{rubbens2024constraintbased,taylor2017smooth} for examples). Figure \ref{fig:qsc} illustrates the difference between relying on exact interpolation conditions and necessary conditions when solving a PEP. 
\begin{figure}[ht!]
    \centering
    \figureH{
    \includegraphics[width=1\linewidth]{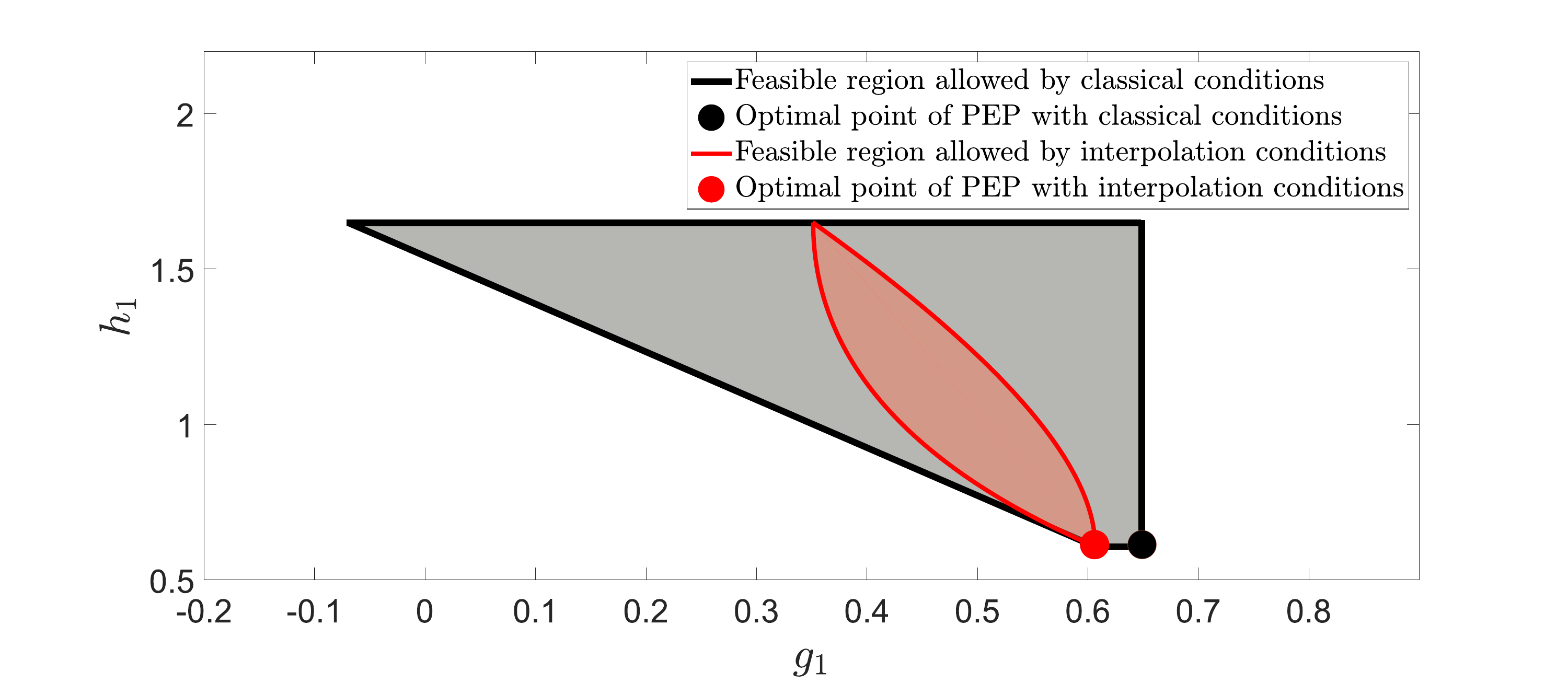}
    %\includesvg[width=0.95\linewidth]{figures/QSC_feasibleV2.svg}
    }
    \setlength{\belowcaptionskip}{-10pt}
    \caption{Given $(x_0,g_0,h_0)=(0,1,1)$, $M=1$ and $x_1=\frac{1}{2}$, the plot shows the admissible region for $(g_1,h_1)$ such that $\{(x_i,g_i,h_i)\}_{i=0,1}$ satisfies (i) classical conditions defining quasi-self-concordant functions, i.e., \eqref{eq:qsc_existing}, (black curves), or (ii) interpolation conditions for this function class, i.e., Corollary \ref{th:IC_qsc} (red curves). It also shows the points $(g_1,h_1)$ corresponding to the worst-case of a variant of Newton's method, whose iterations are given by \eqref{eq:GNM1}, according to both conditions. Using interpolation conditions instead of necessary conditions significantly restricts the domain $(g_1,h_1)$ of the PEP, and the optimal point obtained by solving a PEP with necessary conditions (black dot) is not consistent with any actual function of the class, hence the result is not tight.}
    \label{fig:qsc}
\end{figure}
\FloatBarrier

Obtaining tight bounds thus requires (i) obtaining interpolation constraints for the function class $\F$ of interest, and (ii) solving the PEP \eqref{PEP_untractable_bis}, either analytically or numerically. We briefly present the state of the art regarding these two steps.
\subsubsection*{Interpolation conditions}
Interpolation conditions are known so far for a wide range of functions involving zeroth or first-order oracles, e.g., Lipschitz continuous functions \cite{gruyer2009minimal,whitney1934differentiable}, (smooth) weakly or strongly convex functions \cite{azagra2017whitney,azagra2019smooth,rotaru2024exact,taylor2017smooth,yan2012extension}, weakly convex functions with bounded subgradients \cite{rubbens2024constraintbased}, smooth Lipschitz functions \cite{thesis}, indicator functions \cite{luner2024performance,thesis,taylor2018exact}, difference-of-convex functions \cite{abbaszadehpeivasti2024rate}, and several other classes of functions (see, e.g.,~\cite{guille2022gradient,goujaud2022optimal,goujaud2024pepit}). In addition, interpolation conditions exist for (strongly) monotone, cocoercive, Lipschitz, or linear operators \cite{bousselmi2023interpolation,ryu2020operator}. It turns out that imposing any of these interpolation conditions on a data set $S$ in \eqref{PEP_untractable_bis} can be expressed in a convex way, with the use of a Gram-matrix based lifting \cite{taylor2017smooth}. Obtaining such conditions from scratch is in general difficult, see, e.g., \cite{rubbens2025constructive} for details.

To this day, not much has been achieved on the topic of interpolation conditions for function classes involving higher-order properties, e.g., functions whose second or third derivative is Lipschitz continuous. In particular, to the best of our knowledge, no interpolation condition exists yet for any second-order function class. Some progress has been made in obtaining interpolation conditions for convex functions, when these are meant to interpolate a dataset up to their second derivative \cite{yan2012extension}, but these are sufficient conditions only, and require the dataset to cover a convex region. On the contrary, we seek to obtain necessary and sufficient conditions to be imposed on a finite number of points, hence a non-convex domain.

\subsubsection*{Solving a PEP efficiently}
Obtaining tight bounds via the PEP framework requires being able to solve the associated PEP \eqref{PEP_untractable_bis} efficiently. For instance, for a wide range of methods and function classes, one can show that \eqref{PEP_untractable_bis} can be reformulated as a convex semidefinite program. This is the case for first-order methods and interpolation conditions that can be formulated as (in-)equalities convex in function values and scalar products of iterates and (sub-)gradients, on which the PEP framework allowed obtaining tight bounds \cite{colla2022automatic,de2017worst,dragomir2021optimal,drori2020complexity,gannot2022frequency,kamri2022worst,lee2022convergence,taylor2019stochastic,taylor2018exact}.

Very recently, the framework was extended to include non-convex PEPs \cite{gupta2023nonlinear,das2024branch}, which do not admit such convex reformulation, by relying on a non-convex quadratically constrained quadratic program formulation of \eqref{PEP_untractable_bis}. The use of these non-convex formulations is to this day restricted to the analysis of first-order methods.

To the best of our knowledge, the only attempt at computer-aided analysis of second-order methods is the study of a single iteration of Newton's method on the class of self-concordant functions in \cite{taylorinexact}, but these results are based on interpolation conditions of a larger function class, hence not a priori tight, and are inherently restricted to the analysis of a single iteration as they rely on the use of a metric related to the starting iterate.

\begin{remark}
    Solving non-convex formulations of PEP has been done on a few occasions previously to (i)~find a 2-dimensional worst case using KKT conditions and computer algebra system \cite[Section 4.1]{ryu2020operator}, (ii)~design optimal method using global non-convex solver \cite{gupta2023nonlinear}, (iii)~analyze adaptive first-order methods \cite{barre2020complexity}.
\end{remark}

\subsection{Organization and contributions}
In this work, we first propose a generic approach to derive exact interpolation conditions for classes of univariate functions, namely, we
\begin{itemize}
    \item[(i)] Propose a technique to lift interpolation conditions from a given univariate function class $\F$ to an associated univariate function class $\int \F$, defined as the set of functions whose derivative belongs to $\F$ (Section \ref{sec:intrp} and Theorem \ref{lem:technique}). 
    \item[(ii)] Present a range of function classes on which the technique applies, that is functions satisfying $|f^{(k+1)}(x)|\leq A f^{(k)}(x)^\alpha$ for some $A,\alpha \geq 0$, where $f^{(k)}$ denotes the $k^{\text{th}}$ derivative of $f$ (Section \ref{sec:fun_classes}). This definition includes generalized smoothness \cite{li2023} and generalized self-concordance \cite{sun2019generalized}, and in particular adapts to (strongly convex) functions with $M$-Lipschitz continuous Hessian and (quasi)-self-concordant functions \cite{sun2019generalized}.
    We propose a unified interpretation of these classes by showing they satisfy a Lipschitz condition on a specific quantity depending on $\alpha$ (Proposition \ref{prop:def2}).
    \item[(iii)] Exploit the technique to derive explicit interpolation conditions for these function classes, summarized in Table \ref{tab:tab2} (Section \ref{sec:new_sec3}).
    \begin{table}[ht!]
    \scriptsize
    \centering
    \ra{1.5}
    \setlength{\abovecaptionskip}{0.5pt}
    \caption{Summary of interpolation conditions of classes of univariate functions.}
        \begin{tabular}{@{}llll@{}}
            \toprule
            Class & $\{(x_i,h_i)\}$ &  $\{(x_i,g_i,h_i)\}$ & $\{(x_i,f_i,g_i,h_i)\}$ \\
            \midrule 
            Hessian Lipschitz & Th.\@ \ref{prop:interp_basic} &  Th.\@ \ref{thm:IC_generalized_self_concordant} or Cor.\@ \ref{prop:Hinterpnofi} & Th.\@ \ref{th:ICf_hessian_Lipschitz} (new) \\
            Hessian Lipschitz and strongly convex & Th.\@ \ref{prop:interp_basic} (new) & Th.\@ \ref{thm:IC_generalized_self_concordant} or Cor. \ref{th:ICg_convex_hessian_Lipschitz}  (new) & \\
            Self-concordant & Th.\@ \ref{prop:interp_basic} (new) & Th.\@ \ref{thm:IC_generalized_self_concordant} or Cor.\@ \ref{th:IC_sc} (new) &  \\
            Quasi-self-concordant & Th.\@ \ref{prop:interp_basic} (new) & Th.\@ \ref{thm:quasi_self} or Cor.\@ \ref{th:IC_qsc} (new) &   \\
            Generalized self-concordant & Th.\@ \ref{prop:interp_basic} (new) & Th.\@ \ref{thm:IC_generalized_self_concordant} (new) &   \\
            \bottomrule
        \end{tabular} 
    \label{tab:tab2}
\end{table}
\end{itemize}

We then use these conditions to analytically analyze several second-order methods (Section \ref{sec:results}), thereby
\begin{itemize}
    \item[(i)] Obtaining a tight bound for one iteration of the Cubic Regularized Newton method on Hessian Lipschitz univariate functions~\cite{nesterov2006cubic} (Theorem \ref{lem:new_descent_lemma}) and Gradient Regularized Newton method on quasi-self-concordant univariate functions \cite{doikov2023minimizing} (Lemma \ref{lem:descent_qsc}). These automatically serve as multivariate lower bounds, that improve on existing ones.
    \item[(ii)] Proving tightness of existing results in the \emph{multivariate case}: Newton's method and Gradient method on Hessian Lipschitz functions \cite{nesterov2018lectures}  (Theorems \ref{th:local_quadratic_tight} and \ref{th:Nesterov_GM}) and Newton's method on quasi-self-concordant and strongly convex functions (Lemma \ref{lem:lemma_qsc_improved}).
    \item[(iii)] Proposing an alternative proof of an existing result on the convergence of Newton's method on self-concordant functions (Section \ref{sect:sc_nm});
\end{itemize}

Finally, we leverage on the advances in solving non-convex PEPs to solve formulations associated with several second-order methods (Section \ref{sec:results}), thereby
\begin{itemize}
    \item[(i)] Numerically improving convergence rates in the univariate case: Cubic Regularized Newton method on Hessian Lipschitz functions (Figure \ref{fig:CNM_descent_lemma}, for more than one iteration), Gradient Regularized Newton method on (strongly) convex Hessian Lipschitz functions (Figure \ref{fig:GRN_local_strongly}), Newton and damped Newton methods on self-concordant functions (Theorem \ref{th:dnm_sc_our}), Gradient Regularized Newton method on quasi-self-concordant functions (Figure \ref{fig:LRN2}, for more than one iteration). 
    \item[(ii)] Tuning methods in the univariate case: Cubic Regularized Newton method with a stepsize on Hessian Lipschitz functions (Figure \ref{fig:CNM_alpha}), and Fixed damped Newton method on Hessian Lipschitz functions (Figure \ref{fig:DNM_alpha}).
    \item[(iii)] Comparing different variants of Newton's method in the univariate case (Figure \ref{fig:all_methods}).
\end{itemize}
Table \ref{tab:tab1} summarizes the state-of-the-art bounds on the performance of the second-order methods we analyzed. 
{\small
\begin{table}[ht!]
    \centering
    \ra{1.5}
    \setlength{\abovecaptionskip}{0.5pt}
    \caption{Summary of existing convergence results of the literature together with our analytical or numerical contributions (in the univariate case).}
        \begin{tabular}{@{}llllll@{}}
            \toprule
            Method$^1$ & Class$^2$ & Perf.\@ meas.\@ & Initial condition & References & Contributions$^3$ (univariate case)\\
            \midrule 
            CNM & HL & $\min\limits_{k}|f'(x_{k})|$ & $f(x_0)-f_\star$ & \cite[Th.\@ 1]{nesterov2006cubic} & Improved guarantee (Th.\@ \ref{lem:new_descent_lemma} and Fig.\@ \ref{fig:CNM_descent_lemma}) \\
            CNM & HL & $|f'(x_{N})|$ & $f(x_0)-f_\star$ &  & First (numerical) guarantee (Fig.\@ \ref{fig:CNM_descent_lemma}) \\
            CNM & HL $\cap$ C & $f(x_N)-f_\star$ & $f(x_0)-f_\star$ & \cite[Th.\@ 6]{nesterov2006cubic} & Improved guarantee (Rem.\@ \ref{rem:CNM_gradient_dominated}) \\
            CNM & HL $\cap$ sC & $f(x_N)-f_\star$ & $f(x_0)-f_\star$ & \cite[Th.\@ 7]{nesterov2006cubic} & Improved guarantee (Rem.\@ \ref{rem:CNM_gradient_dominated}) \\
            GNM2 & HL $\cap$ sC & $|f'(x_{k+1})|$ & $|f'(x_{k})|$ & \cite[Th.\@ 2.7]{mishchenko2023regularized} & Improved (Fig.\@ \ref{fig:GRN_local_strongly}) \\
            NM & HL & $|x_{k+1} - x_\star|$ & $|x_k - x_\star|$ & \cite[Th.\@ 1.2.5]{nesterov2018lectures} & Tightness proved (Th.\@ \ref{th:local_quadratic_tight}) \\
            GM & HL & $|x_{k+1} - x_\star|$ & $|x_k - x_\star|$ & \cite[Th.\@ 1.2.4]{nesterov2018lectures} & Tightness proved (Th.\@ \ref{th:local_linear_tight})\\
            DNM & HL & $|x_{k+1} - x_\star|$ & $|x_k - x_\star|$ &  & First (numerical) guarantee (Figs.\@ \ref{fig:DNM_r} and \ref{fig:DNM_alpha}) \\
            NM & SC & $\lambda_f(x_{k+1})$ & $\lambda_f(x_k)$ & \cite[Eq.\@ (11)]{hildebrand2021optimal} & Worst-case function (Sect. \ref{sect:sc_nm}) \\
            DNM & SC & $\lambda_f(x_{N})$ & $\lambda_f(x_0)$ & \cite[Eq.\@ (11)]{hildebrand2021optimal} & Tightness proved (Th.\@ \ref{th:dnm_sc_our}) \\
            GNM1 & QSC & $|f'(x_{N})|/f''(x_N)$ & $|f'(x_0)|/f''(x_0)$ & \cite[Eq.\@ (49)]{doikov2023minimizing} & Improved guarantee (Lem.\@ \ref{lem:descent_qsc} and Fig.\@ \ref{fig:LRN2}) \\
            NM & QSC $\cap$ sC & $|f'(x_{k+1})|$ & $|f'(x_k)|$ & \cite[Eq.\@ (53)]{doikov2023minimizing} & Tightness proved (Lem.\@ \ref{lem:lemma_qsc_improved}) \\
            all & HL $\cap$ sC & $|f'(x_{k+1})|$ & $|f'(x_{k})|$ &  & First (numerical) comparison (Fig.\@ \ref{fig:all_methods}) \\
            \bottomrule
        \end{tabular} 
        {\scriptsize \begin{flushleft}
        $^1$ NM: Newton's method, CNM: Cubic Regularized Newton method, GNM: Gradient Regularized Newton method, GM: Gradient method, DNM: fixed damped Newton method. \\
        $^2$ HL: Hessian Lipschitz, C: Convex, sC: Strongly convex, SC: Self-Concordant, QSC: Quasi-Self-Concordant functions.\\
        $^3$ Our contributions are proved when referred to as Lem., Rem., Sect., or Th.\@ and they are only conjectured based on numerical experiments when referred to as Fig.\@ or Conj.
        \end{flushleft}
        }
    \label{tab:tab1}
\end{table}
}
\FloatBarrier
\subsection{Restriction to the univariate setting}
Our method and analysis are restricted to classes of univariate functions ($d=1$) due to two main difficulties, preventing a straightforward generalization to higher dimensions. First, it is in general not straightforward to obtain interpolation conditions for given function classes. Nevertheless, we propose a technique to automatically obtain such conditions in the univariate case, which heavily relies on properties specific to univariate functions. Second-order interpolation conditions for classes of multivariate functions remain an open question. Second, there is no known efficient way to solve PEPs involving second-order quantities in the multivariate case: these are a priori highly non-convex and of large dimensions. On the contrary, we manage to solve PEPs in various univariate second-order settings due to their smaller size.

This work thus only takes a first step into the problem of tightly analyzing second-order methods. However, we believe it is important for several reasons. First, univariate interpolation conditions can provide intuition about candidate multivariate conditions. Second, exact univariate bounds provide lower bounds on the multivariate worst-case performance of the method. In fact, in many settings (first and second-order methods), univariate functions are observed to provide the general multivariate worst-case \cite{cartis2010complexity,doikov2021new,taylor2017smooth,toint2024examplesslowconvergenceadaptive,rotaru2024exact} (it is even more often the case when considering a single iteration). Finally, in case the obtained univariate lower bound matches a known analytical multivariate upper bound, the associated worst-case instance proves tightness of this upper bound.

\subsection{Notation}
Given $N \in \mathbb{N}$, we let $[N]=\{0, \cdots,N\}$. Throughout, we consider the space $\mathbb{R}$, and let $\bR=\R\cup \{+\infty\}$.

We denote by $\C^{m}$ the class of univariate functions $f:\R\to \R$ that are at least $m$ times everywhere differentiable, with derivative up to order $m$ everywhere continuous. We denote by $\bar{\C}^0$ the class of functions $f:\R \to \bar{\R}$ that are everywhere continuous, i.e., $\forall c\in \R$, $\lim_{x\to c}f(x)=f(c)$, in the following sense:
\begin{align*}
    \begin{cases}
        \forall \epsilon >0, \ \exists \delta>0, \forall x\in \R: \ |x-c|<\delta \Rightarrow |f(x)-f(c)|<\epsilon &\text{ for finite  $f(c)$},\\
        \forall N >0, \ \exists \delta>0, \forall x\in \R: \ |x-c|<\delta \Rightarrow f(x)>N &\text{ for infinite  $f(c)$}.
    \end{cases}
\end{align*}
We denote by $\text{dom } f=\{x\in \R:\ f(x) \neq +\infty \}$ the \emph{effective domain of} $f\in \bar{\C}^0$. 

Let $f\in \bar{\C}^0$, and suppose $f$ is at least $m$ times differentiable on $\text{dom }  f$. We denote by $f^{(k)}$ the $k^{\text{th}}$ derivative of $f$ (where $f^{(0)}=f$), defined $\forall c\in \R$ and for $k\geq 1$ as
\begin{align*}
    f^{(k)}(c)=\lim_{x\to c} \frac{f^{(k-1)}(x)-f^{(k-1)}(c)}{x-c},
\end{align*}
i.e. as the classical order $k$ derivative on the effective domain, and $+\infty$ otherwise. We say $f\in \bC^m$ if (i) $f^{(k)}$ exists, and (ii) $f^{(k)}\in \bC^0$, $\forall k\leq m$.

Finally, we say a function $f:\R\to \bar{\R}$ is $M$-Lipschitz when \begin{align}\label{eq:M_Lip}
    |f(x)-f(y)|\leq M|x-y|, \ \forall x,y \in \R,
\end{align} $M$-smooth when it is differentiable with $M$-Lipschitz derivative, convex when $f(\lambda x+(1-\lambda)y)\leq \lambda f(x)+(1-\lambda)f(y) \quad \forall x,y \in \R, \ \lambda \in (0,1)$, and strongly convex with parameter $\mu > 0$ when $f(x)-\mu \frac{x^2}{2}$ is convex.

\section{Generic technique to lift interpolation conditions from known classes to higher-order classes}\label{sec:intrp}
We seek for \emph{interpolation conditions} for classes of univariate functions $\F \subseteq \bar{\C}^m$, i.e., necessary and sufficient conditions ensuring $\F$-\emph{interpolability} of a discrete data set.
\begin{definition}[Interpolation with function values]\label{def:interp_fun}
Given a univariate function class $\F \subseteq \bar{\C}^m$, a set $S=\{(x_i, f_i^0, f_i^1, ..., f_i^m)\}_{i\in [N]}\in (\R\times \ldots \times \R)^N$ is $\F$-interpolable with function values if and only if
    \begin{equation}
        \exists f \in \F :
           f^{(k)}(x_i)  = f^k_i, \ \forall k\in[m], \quad \forall i\in[N].
    \end{equation}
\end{definition}
In some circumstances, given a dataset $S$, one seeks to ensure the existence of a function in $\F$ consistent with $S$ except possibly for its function values $f_i^0$. Definition \ref{def:interp_fun} can be straightforwardly extended to this case, called interpolation without function values.
\begin{definition}[Interpolation without function values]\label{def:interp_nofun}
Given a univariate function class $\F \subseteq \bar{\C}^m$, a set $S=\{(x_i, f_i^0, f_i^1, ..., f_i^m)\}_{i\in[N]}\in (\R\times \ldots \times \R)^N$ is $\F$-interpolable without function values if and only if
    \begin{equation}
        \exists f \in \F  : f^{(k)}(x_i)  = f^k_i, \quad \forall k=1,\cdots,m , \ \forall i\in[N].
    \end{equation}
\end{definition}
The quantities $f_i^0$ are thus ignored in interpolation without function values and are not necessarily equal to function values $f^{(0)}(x_i)$.

\begin{remark}
In the univariate case, interpolation conditions without function values are often more concise than interpolation conditions with function values. Hence, we resort to this formulation whenever possible, for instance when analyzing methods that do not involve function values in their settings (see, e.g., \cite[Section 3.4]{thesis}). 
\end{remark}
Whenever considering the cases $m= 1, 2$, we denote $f_i^0:=f_i$, $f_i^1:=g_i$, $f_i^2:=h_i$, $f^{(1)}(x):=f'(x)$ and $f^{(2)}(x):=f''(x)$.

\subsection{Overview of the technique}\label{sec:overview}
Suppose interpolation conditions are known for a class $\mathcal{F} \subseteq \bar{\mathcal{C}}^m$. We propose a technique, summarized in Theorem~\ref{lem:technique}, to derive interpolation conditions for the associated class
\begin{align}
    \int \mathcal{F} := \{ f : \mathbb{R} \to \bar{\mathbb{R}} \mid f' \in \mathcal{F} \} \subseteq \bar{\mathcal{C}}^{m+1},
    \label{def:Fm-1}
\end{align}
consisting of functions whose derivative belongs to $\mathcal{F}$. This approach thus lifts interpolation conditions from $\mathcal{F}$ to the higher-order class $\int \mathcal{F}$.

Before presenting the method in details, we illustrate its two main steps on a simple example, depicted in Figure~\ref{fig:illustration_method}. Consider the class $\mathcal{F}_M$ of $M$-Lipschitz functions, whose interpolation conditions correspond to a discretization of \eqref{eq:M_Lip}, see, e.g., \cite{kirszbraun1934zusammenziehende,valentine1945lipschitz}. Using these conditions, we aim to derive interpolation conditions for the class $\int \mathcal{F}_M$ of $M$-smooth functions, i.e., conditions that ensure the consistency of a dataset $S = \{(x_i,f_i,g_i)\}_{i\in[N]}$ with a $M$-smooth function defined everywhere.

First, we address interpolation conditions without function values for $\int \F_M$. By the fundamental theorem of calculus, every continuous function in $\mathcal{F}_M$ is the derivative of some function in $\int \mathcal{F}_M$, so these conditions coincide with the interpolation conditions with function values for $\mathcal{F}_M$ (Lemma~\ref{lemma:interpnof}). Specifically, $S$ is $\int \mathcal{F}_M$-interpolable without function values if and only if
\begin{align*}
    |g_i - g_j| \leq M |x_i - x_j| \quad \forall i,j.
\end{align*}
\begin{figure}[ht!]
    \centering
     \begin{subfigure}[b]{0.85\textwidth}
         \centering
         %\includesvg[width=\textwidth]{figures/gmin_gmaxV1}
         \includegraphics[width=\textwidth]{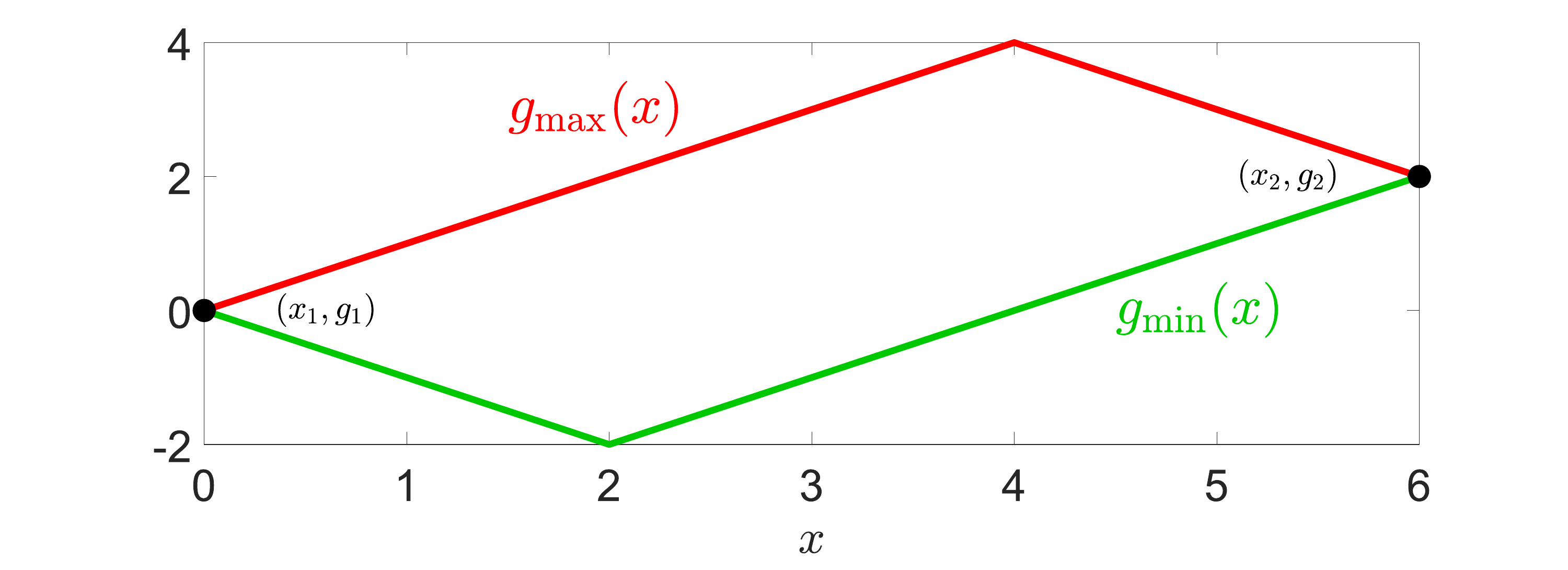}
     \end{subfigure}
     \begin{subfigure}[b]{0.85\textwidth}
         \centering
         %\includesvg[width=\textwidth]{figures/fmin_fmaxV1}
         \includegraphics[width=\textwidth]{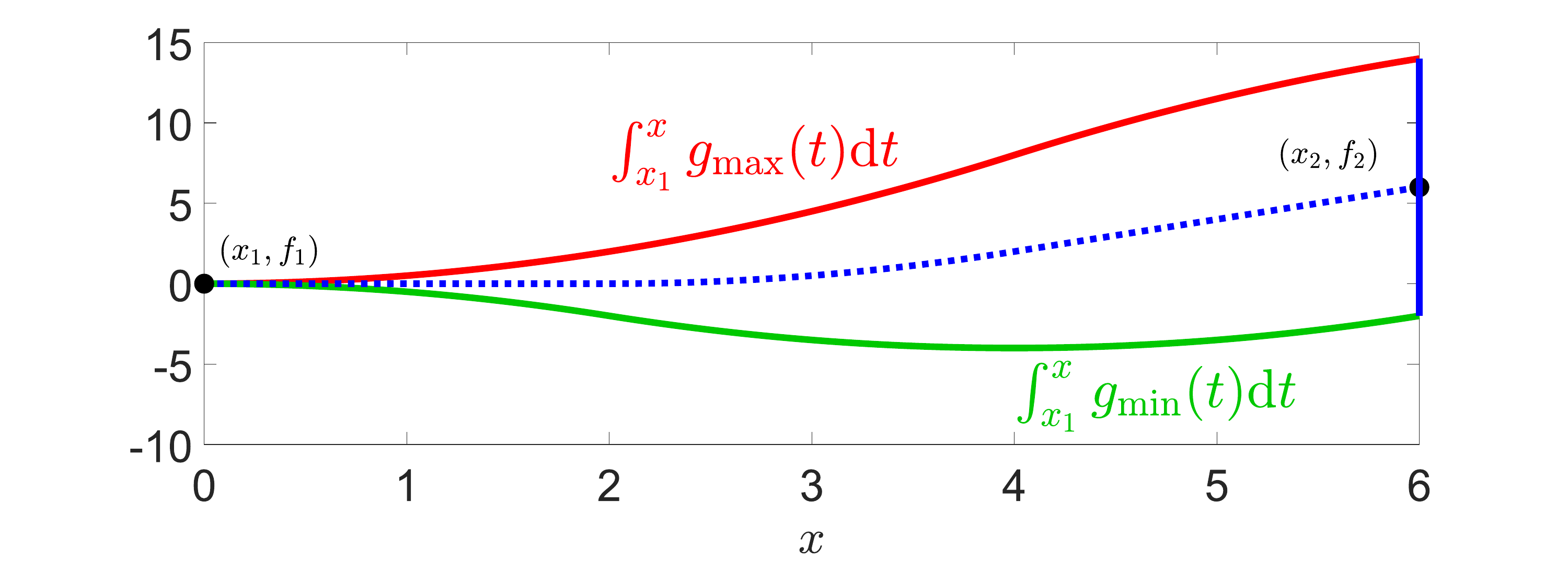}
     \end{subfigure}
    \caption{Illustration of the method to obtain interpolation conditions ensuring $S=\{(x_i,f_i,g_i)\}_{i=1,2}$ to be $\FLip{M}$-interpolable: extremal $M$-Lipschitz interpolants $g_{\min}$ (above, green curve) and $ g_{\max}$ (above, red curve) of $\{(x_i,g_i)\}_{i=1,2}$ (see Proposition \ref{prop:properties} for their derivation). These extremal interpolants are integrated (below, red and green curves), and define an interval including all possible smooth functions interpolating $S$, except possibly $f_2$. The interval $[f_1+\int_{x_1}^{x_2}g_{\min}(x) \mathrm{d}x, f_1+\int_{x_1}^{x_2}g_{\max}(x) \mathrm{d}x]$ (below, blue line) is exactly the interval of admissible values for $f_2$, in the sense that whenever $f_2$ belongs to this interval, there exists a convex combination of the integrals of $g_{\min}$ and $g_{\max}$ (below, blue dotted curve) which is in $\int \FLip{M}$ and interpolates $S$.}
    \label{fig:illustration_method}
\end{figure}

Second, we exploit these interpolation conditions without function values to derive interpolation conditions \emph{with} function values for \(\int \mathcal{F}_M\). Suppose  \(S\) is \(\int \mathcal{F}_M\)-interpolable without function values. Since piecewise smooth continuous functions are globally smooth, the problem reduces to interpolation on a single interval \([x_i, x_j]\) (see Lemma~\ref{lem:interval}). 

On this interval, consider all functions in \(\mathcal{F}_M\) that interpolate \(\{(x_i, g_i), (x_j, g_j)\}\), and in particular the \emph{extremal} interpolants \(g_{\min}\) and \(g_{\max}\) (Definition~\ref{def:gextr}), see Figure~\ref{fig:illustration_method} for an illustration. Given \(f_i\), integration of these extremal interpolants over \([x_i, x_j]\) yields the interval
\[
\left[\, f_i + \int_{x_i}^{x_j} g_{\min}(x) \, \mathrm{d}x, \quad f_i + \int_{x_i}^{x_j} g_{\max}(x) \, \mathrm{d}x \,\right]
\]
of admissible values for \(f_j\) (cf.\ Lemma~\ref{prop:interm1}). Indeed, if \(f_j\) lies within this interval, there exists a convex combination \(g = \lambda g_{\min} + (1-\lambda) g_{\max}\) such that the function $f_i + \int_{x_i}^x g(t) \, \mathrm{d}t$ (i) is \(M\)-smooth by convexity of the smoothness property, i.e., any convex combination of smooth functions is itself smooth, and (ii) interpolates \(S\). Conversely, if \(f_j\) lies outside this interval, no function in \(\mathcal{F}_M\) interpolates \(S\).

Hence, imposing the interpolation condition
\begin{align}
\int_{x_i}^{x_j} g_{\min}(x) \, \mathrm{d}x \leq f_j - f_i \leq \int_{x_i}^{x_j} g_{\max}(x) \, \mathrm{d}x,
\end{align}
in addition to the conditions without function values, yields interpolation conditions with function values for \(\int \mathcal{F}_M\), see Theorem \ref{thm:IC_generalized_self_concordant}.

In this toy example, to recover interpolation conditions for \(\int \mathcal{F}_M\) from those of \(\mathcal{F}_M\), we relied on several key properties of both \(\mathcal{F}_M\) and \(\int \mathcal{F}_M\):
\begin{enumerate}
\item Access to interpolation condition without function values for $\int \F$;

\item Existence of extremal interpolants;

\item Convexity of $\int \F$;

\item Piecewise \(\int \mathcal{F}\)-interpolability implies global \(\int \mathcal{F}\)-interpolability.
\end{enumerate}
For general classes of continuous univariate functions, Point 1 is obtained naturally using the continuous and univariate properties of functions. Points 2, 3, and 4 will be guaranteed by Assumptions \ref{assum:extrgrad}, \ref{assum:conv}, and \ref{assum1}.

Before stating Theorem~\ref{lem:technique}, we detail the steps of the method and underlying assumptions. In Section~\ref{sec:fun_classes}, we demonstrate that Theorem~\ref{lem:technique} applies to a broad range of function classes, including (convex) functions with Lipschitz Hessians and (quasi-) self-concordant functions. For illustrative applications of the technique to these and other classes, see Section~\ref{sec:interp_basic}.
\subsection{Interpolation without function values}
In the univariate case, interpolation without function values of $\int \F$ amounts exactly to interpolation with function values of $\F$, for continuous function classes.
\begin{lemma} \label{lemma:interpnof}
    Consider univariate function classes $\F \subseteq \bar{\C}^m$ and $\int \F \subseteq \bar{\C}^{m+1}$ (defined in \eqref{def:Fm-1}). A set $S=\{(x_i,f_i^0,f_i^1,...,f_i^m)\}_{i\in[N]}$ is $\int \F$-interpolable (without function values) if and only if $\tilde S:=\{(x_i,f_i^1,...,f_i^m)\}_{i\in[N]}$ is $\F$-interpolable (with function values).
    \begin{proof}
        By Definition \ref{def:interp_nofun}, $S$ is $\int \F$-interpolable without function values if 
        \begin{alignat*}{4}
             & \exists \ f\in \textstyle \int \F & &:   f^{(k)}(x_i)=f_i^k, && \forall k=1,\ldots,m, \ \forall i \in [N] \\
             \overset{\text{\eqref{def:Fm-1}}}{\Leftrightarrow}\ & \exists \ f: f'=g \in \F&& , \ f^{(k)}(x_i)=f_i^k, \ &&\forall k =1,\ldots,m, \ \forall i \in [N]\\
            \Leftrightarrow \ \ &\exists \ g\in \F && : g^{(k)}(x_i)=f_i^{k+1}, &&\ \forall k=0,\ldots,m-1, \ \forall i \in [N]\\
             \mathclap{\hspace{5.6cm}\overset{\text{Def.\@ \ref{def:interp_fun}}}{\Leftrightarrow}  \tilde S \text{ is $\F$-interpolable with function values,}}
       \end{alignat*}
        where the second equality follows from the fundamental theorem of calculus, see, e.g., \cite[Chapter 6]{rudin1976principles} applied on the effective domain of $g$.
    \qed \end{proof}
\end{lemma}

\subsection{Interpolation conditions with function values (2 points)}
We now show how to go from interpolation without function values to interpolation with function values. Informally, given a set $S=\{(x_i,f_i^0,f_i^1,\ldots,f_i^{m+1})\}_{i\in [N]}$ which is $\int \F$-interpolable without function values, we ensure that one of the functions in $\F$ interpolating $S$ (except possibly for $f_i^0, \ i\in [N]$) is the derivative of a function $f$ satisfying $f(x_i)=f_i^0, \ i\in [N]$.

We take the simplest situation possible and restrict our attention to the interpolation of a single pair of data points $x_1<x_2$.

Consider a set $S=\{(x_i,f_i^0,f_i^1,\ldots,f_i^{m+1})\}_{i=1,2}$ which is $\int \F$-interpolable without function values and the associated set $\tilde S=\{(x_i,f_i^1,\ldots,f_i^{m+1})\}_{i=1,2}$. By Lemma \ref{lemma:interpnof}, $\tilde S$ is $\F$-interpolable, hence there exists at least one function $g \in \F$ interpolating $\tilde S$. For each of such $g$, we can construct an associated function $f \in \int \F$, interpolating $S$ with the exception of $f(x_2)$, which could differ from $f_2^0$.
\begin{lemma} \label{lem:fromgtof}
Consider a univariate function class $\F \subseteq\bar{\C}^m$, and let \break$S=\{(x_i,f_i^0,f_i^1,\ldots,f_i^{m+1})\}_{i=1,2}$ be $\int \F$-interpolable without function values, $\tilde S=\{(x_i,f_i^1,\ldots,f_i^{m+1})\}_{i=1,2}$, and $g\in \F$ be a function interpolating $\tilde S$.  Then, 
    \begin{align}\label{def:associated_fun}
    f(x):=f_1^0-\int_{-\infty}^{x_1} g(z) \mathrm{d}z +\int_{-\infty}^x g(z) \mathrm{d}z,
    \end{align}
    belongs to $\int \F$ and interpolates $S$ except $f(x_2)$, which could differ from $f_2^0$.
\begin{proof}    
    It holds that $f'(x)=g(x)$, hence by Definition \ref{def:Fm-1}, $f\in \int \F$. In addition, 
    \begin{align*}
        & f^{(k)}(x_i)=g^{(k-1)}(x_i)= f_i^k,\ i=1,2, \ k=1,2,\ldots,m+1, \text{ and } f(x_1)=f_1^0,
    \end{align*}
    hence $f$ interpolates $S$ except possibly for $f(x_2)$.
\qed \end{proof}
\end{lemma}

Consider, among these functions $g\in \F$ interpolating $S$, the lowest and highest functions, called the \emph{extremal interpolants} of $S$, see Figure \ref{fig:illustration_method} for an illustration. As such, extremal interpolants are not well-defined and might not exist or be non-unique. We thus introduce the notion of \emph{lower and upper interpolating envelopes}, that is pointwise minimum and maximum of all functions in $\F$ interpolating $S$.
\begin{definition}[Extremal interpolating envelopes]\label{def:gextr}
Consider a univariate function class $\F \subseteq \bar{\C}^m$ and an $\F$-interpolable set $S=\{(x_i,g_i^0,g_i^1,...,g_i^m)\}_{i=1,2}$, where $x_1<x_2$. The \emph{extremal interpolating envelopes} $g_{\min}$ and $g_{\max}$ of $S$ are defined as, $\forall x\in \R$,
\begin{align} \label{eq:defgmin}
g_{\min\text{ (resp. } \max)}(x) =\underset{g:\mathbb{R}\rightarrow \mathbb{R}}{\inf} \text{ (resp. $\sup$)}\  \quad & g(x) \\
\text{s.t.} \quad & g \in  \F \nonumber \text{ and }g^{(k)}(x_i)=g_i^k, \ i=1,2, \ k\in [m]. \nonumber 
\end{align}
\end{definition}
Extremal interpolating envelopes always exist, but do not always belong to $\F$. To ensure it is the case, we require $\F$ to be \emph{extremally interpolable}:
\begin{assump}[Extremally interpolable function class]\label{assum:extrgrad}
    We say a function class $\F \subseteq \bar{\C}^m$ of univariate functions is \emph{extremally interpolable} if any $\F$-interpolable set $ S=\{(x_i,g_i^0,g_i^1,...,g_i^m)\}_{i=1,2}$ satisfies $g_{\min}, \ g_{\max} \in \F$ where $g_{\min}$ and $g_{\max}$ are defined in \eqref{eq:defgmin}.
\end{assump}
Whenever $\F$ is extremally interpolable, we refer to extremal interpolating envelopes of $S$ as extremal interpolants of $S$.

Suppose now that such extremal interpolants exist, and that the associated functions $f_{\min},f_{\max} \in \int \F$ as defined in \eqref{def:associated_fun} evaluated at $x_2$ are respectively smaller and bigger than $f_2^0$. Then, there exists a function $f \in \int \F$ interpolating $S$ (including $f(x_2)=f_2^0$), provided $\int \F$ is \emph{extremally completable}.
\begin{assump}[Extremally completable function class]\label{assum:conv}
Let $\F \subseteq\bar{\C}^m$ be an extremally interpolable (Assumption \ref{assum:extrgrad}) class of univariate functions, and consider $\int \F \in \bar{\F}^{m+1}$. Let $S=\{(x_i,f_i^0,f_i^1,\ldots,f_i^{m+1})\}_{i=1,2}$ be $\int \F$-interpolable without function values, $\tilde S=\{(x_i,f_i^1,\ldots,f_i^{m+1})\}_{i=1,2}$, and $g_{\min}, g_{\max}\in \F$ be the extremal interpolants of $\tilde S$, as defined in \eqref{eq:defgmin}. Let, in addition,
\begin{align*}
    f_1^0+\int_{x_1}^{x_2} g_{\min}(z) \mathrm{d}z\leq f_2^0\leq f_1^0+\int_{x_1}^{x_2} g_{\max}(z) \mathrm{d}z.
\end{align*}
We then say $\int \F$ is \emph{extremally completable} if
\begin{align*}
    \exists f \in \int \F :\ f^{(k)}(x_i)= f^{k}_i, \quad \forall i\in [N], \quad \forall k\in [m+1].
\end{align*}  
\end{assump}
Instances of function classes $\int \F$ that are extremally completable include any function class satisfying the more natural notion of \emph{convexity}, including quasi-self-concordant functions and functions with Lipschitz Hessian, see Proposition \ref{prop:properties}.
\begin{assump}[Convex function class]\label{assum:conv1}
        We say a class $\F$ of univariate functions is \emph{convex} if, given any $f_a,f_b \in \F$ and any $\lambda \in [0,1]$,  
        \begin{align*}
            \lambda f_a+(1-\lambda) f_b \in \F.
        \end{align*}  
\end{assump}
Convex function classes are extremally interpolable.
\begin{lemma}
    Let $\F \subseteq\bar{\C}^m$ be an extremally interpolable (Assumption \ref{assum:extrgrad}) class of univariate functions. If $\int \F$ is convex (Assumption \ref{assum:conv1}), then $\int \F$ is extremally completable (Assumption \ref{assum:conv}).
\end{lemma}
\begin{proof}
    With the same notation as in Assumption \ref{assum:conv}, let     \begin{align*}
    f_{\min/\max}(x):=f_1^0-\int_{-\infty}^{x_1} g_{\min/\max}(z) \mathrm{d}z +\int_{-\infty}^x g_{\min/\max}(z) \mathrm{d}z.
    \end{align*}
 By Lemma \ref{lem:fromgtof}, $f_{\min}$ and $f_{\max}$ belong to $\int \F$. In addition, the two functions interpolate $S$ with the exception of $f_2^0=\lambda f_{\text{lo}}(x_2)+(1-\lambda)f_{\text{hi}}(x_2)$ for some $\lambda\in [0,1]$. Hence $f:= \lambda  f_{\text{lo}}+ (1-\lambda) f_{\text{hi}}$ interpolates $S$ (including $f(x_2)=f_2$), and by convexity of $\int \F$, we have $f\in \int \F$.
\qed \end{proof}
Convexity on its own proves to be too restrictive, since not satisfied by classes of interest, e.g., self-concordant functions. This justifies the introduction of the weaker notion of extremal completability, satisfied by self-concordant functions.

For classes $\F$ that are extremally interpolable, and such that $\int \F$ is extremally completable, we obtain interpolation conditions with function values for $\int \F$, on a single interval, by enforcing $f_2^0$ to belong to the interval $[f_1^0+\int_{x_1}^{x_2} g_{\min}(z) \mathrm{d}z, f_1^0+\int_{x_1}^{x_2} g_{\max}(z) \mathrm{d}z]$ defined by the integrals of the extremal interpolants of $S$.
\begin{lemma}
    \label{prop:interm1}
    Let $\F \subseteq \bar{\C}^m$ be an extremally interpolable (Assumption \ref{assum:extrgrad}) class of univariate functions and let $\int \F \subseteq \bar{\C}^m$ (defined in \eqref{def:Fm-1}) be extremally completable (Assumption \ref{assum:conv}).
    
    A set $S=\{(x_i,f_i^0,f_i^1,...,f_i^{m+1})\}_{i=1,2}$, where $x_1<x_2$ is $\int \F$-interpolable with function values if and only if $S$ is $\int \F$-interpolable without function values, and $f_1^0, \ f_2^0$ satisfy
            \begin{align}\label{eq:integr}
                \int_{x_1}^{x_2} g_{\min}(x) \mathrm{d}x\leq f_2^0-f_1^0\leq \int_{x_1}^{x_2} g_{\max}(x) \mathrm{d}x,
            \end{align}
            where $g_{\min}$ and $g_{\max}$ are defined as in \eqref{eq:defgmin}.
\begin{proof}
    Sufficiency follows the definition of extremal completability. To prove necessity, suppose $S=\{(x_i,f_i^0,f_i^1,...,f_i^{m+1})\}_{i=1,2}$ is $\int \F$-interpolable (with function values). Then, by Lemma \ref{lemma:interpnof}, $S$ is $\F$-interpolable (without function values). In addition, by Definition \ref{def:gextr} of extremal interpolants, any function $g:\R \to \R$ such that (i) $\exists f\in \int \F: \ g(x)=f'(x), \ \forall x\in \R$ and (ii) $g$ interpolates $S$,  satisfies $$g_{\min}(x)\leq g(x)\leq g_{\max}(x), \ \forall x\in [x_1,x_2].$$
    Hence, \eqref{eq:integr} is necessarily satisfied.
    
\qed \end{proof}
\end{lemma}

To obtain full interpolation conditions for $\int \F$, it remains to consider the case of an arbitrary number of points to interpolate.
\subsection{Interpolation conditions with function values ($N$ points)}
In the univariate case, interpolation of an arbitrary number of points is equivalent to that of a single pair, provided the class $\int \F$ is \emph{order $m+1$ connectable}, in the sense that the juxtaposition of different functions in $\int \F$, where the functions coincide up to order $m+1$ on the boundaries of some intervals, is itself a function in $\int \F$.
\begin{assump}[Order $m$ connectable function class]\label{assum1}
A class $\F \subseteq \bar{\C}^m$ of univariate functions is \emph{order $m$ connectable} if, given any $x_0\leq ...\leq x_{K+1} \in \mathbb{R}\cup \{-\infty,\infty\}$, where $x_0=-\infty$ and $x_{K+1}=\infty$, and any $K$ functions $f_j \in \F$ such that 
\begin{align*}
    f_j^{(l)}(x_{j+1})=f_{j+1}^{(l)}(x_{j+1}), \quad \forall l\in[m], \quad j\in [K-1],
\end{align*}
then the piecewise function \begin{align*}
    f:\R \to \R:\ f(x):=f_j(x) \quad \forall x\in [x_j,x_{j+1}], \, \forall j\in [K]
\text{ belongs to $\F$.}
\end{align*} 
\end{assump}
\begin{remark}
Linear functions are order $1$ connectable, since the juxtaposition of several linear functions whose slope is identical is linear, but not order $0$ connectable, since the juxtaposition of linear functions with different slopes is not linear. Hence, an appropriate choice of order is essential when determining whether a function class is connectable.
\end{remark}

Order $m+1$ connectivity of $\int \F$ allows extending Lemma \ref{prop:interm1} to any arbitrary number of points.
\begin{lemma}\label{lem:interval}
   Let $\F \subseteq \bar{\C}^m$ be an order $m$ connectable class of univariate functions (Assumption \ref{assum1}). A set $S=\{(x_i,f_i^0,f_i^1,...,f_i^m)\}_{i\in[N]}$, with ordered points $x_0\leq x_1\leq\ldots\leq x_N$, is $\F$-interpolable with function values if and only if,
    \begin{equation}
        \forall i\in[N-1] : \left\{ (x_i,f_i^0,f_i^1, ..., f_i^m),(x_{i+1},f_{i+1}^0,f_{i+1}^1, ..., f_{i+1}^m)\right\} \text{ is } \F \text{-interpolable}. \label{eq:goal}
    \end{equation}
\end{lemma}
The proof of Lemma \ref{lem:interval} is deferred to Appendix \ref{app:interval}.
\begin{remark}
    This differs from the multivariate case, where interpolation of a single pair can significantly differ from interpolation of an arbitrary number of points (see, e.g., \cite[Proposition 4]{ryu2020operator}). For example, some inequalities are interpolation conditions whenever applied to a single pair but are only necessary otherwise.
\end{remark}

\subsection{Main Theorem}
We are now ready to present our main Theorem, allowing to lift interpolation conditions from $\F$ to $\int \F$.
\begin{theorem}
    \label{lem:technique}
    Let $\F \subseteq \bar{\C}^m$ be an extremally interpolable (Assumption \ref{assum:extrgrad}) class of univariate functions, and let $\int \F \subseteq \bar{\C}^{m+1}$ (defined in \eqref{def:Fm-1}) be extremally completable  (Assumption \ref{assum:conv}) and order $m+1$ connectable (Assumption \ref{assum1}).
    
    A set $S=\{(x_i,f_i^0,f_i^1,...,f_i^{m})\}_{i\in [N]}$, where $x_0\leq x_1 \leq ... \leq x_N$ is $\int \F$-interpolable if and only if $S$ is $\int \F$-interpolable without function values, and $\forall i \in [N]$,
            \begin{align} \label{eq:cond_finale}
               \int_{x_i}^{x_{i+1}} g_{\min}(x) \mathrm{d}x\leq f_{i+1}-f_i\leq\int_{x_i}^{x_{i+1}} g_{\max}(x) \mathrm{d}x,
            \end{align}
            where $g_{\min}$ and $g_{\max}$ are defined as in \eqref{eq:defgmin}.
    \begin{proof}
    It suffices to combine Lemmas \ref{prop:interm1} and \ref{lem:interval}.
    
    \qed \end{proof} 
\end{theorem}
\begin{remark}
    Theorem \ref{lem:technique} requires imposing conditions on ordered pairs only. However, imposing conditions on all pairs yields, by Lemma \ref{lem:interval}, equivalent interpolation conditions, that are often clearer and easier to handle.
\end{remark}
\begin{remark}
    While the sequel focuses on function classes with second-order properties, Theorem \ref{lem:technique} can provide interpolation conditions for function classes with higher-order or first-order properties.
\end{remark}

\section{Generalized Lipschitz functions}\label{sec:fun_classes}
This section presents a unified characterization of classes of interest in second-order optimization, including, e.g., (quasi)-self-concordant functions and (convex) functions with Lipschitz Hessian, and on which the technique presented in Section \ref{sec:intrp} applies. We show that these functions can all be defined as functions whose second derivative satisfies some kind of generalized Lipschitz condition, and belongs to given classes $\F$. We call such classes $\F$ \emph{basic function classes}, and will be interested in obtaining interpolation conditions for $\int^{(2)}\F$.

We first present several characterizations of $\F$, before proving they satisfy all assumptions required for validity of Theorem \ref{lem:technique}. 

\subsection{Basic function classes}
Basic classes $\F$ are defined as satisfying a notion of generalized Lipschitzness, in the spirit of generalized smoothness introduced in \cite{li2023}  and generalized self-concordance introduced in \cite{sun2019generalized}. Their interest lies in the fact that the associated class $\int^{(2)} \F$ include all classes of generalized self-concordant functions \cite[Equation (1)]{sun2019generalized} of the form
\begin{equation}\label{eq:generalized_sc}
    |f'''(x)| \leq A f''(x)^\alpha, \ f''(x)\geq 0,  
\end{equation}
for some $A,\alpha\geq 0$, as well as the class of functions with Lipschitz Hessian.
Throughout, given $\alpha \geq 0$, we define 
\begin{align}\label{eq:beta}
    \ba=\begin{cases}
        \frac{1}{1-\alpha} & \text{ if }\alpha \neq 1\\
        1& \text{ if }\alpha = 1.
    \end{cases}
\end{align}
\begin{definition}[Generalized Lipschitz function] \label{def:function_class}
    Let $M,\alpha\geq 0$, and $\ba$ be defined as in \eqref{eq:beta}. We say that $f:\R\to\bR$ is $(M,\alpha,+)$-generalized Lipschitz if $f$ (i) is non-negative everywhere, (ii) belongs to $\bC^0$ if $\alpha>1$, and $C^0$ if $\alpha\leq 1$, (iii)~is piecewise $\bar{\C}^1$ if $\alpha>1$, and piecewise $\C^1$ if $\alpha \leq 1$, and (iv) satisfies, whenever differentiable,
    \begin{align}
        |f'(x)|& \leq |\ba|Mf(x)^\alpha. \label{eq:function_class}
    \end{align}
    We denote by $\FLip{M,\alpha,+} \subseteq \bar{\C}^0$ (and $\C^0$ when $\alpha\leq 1$) the class of $(M,\alpha,+)$-generalized Lipschitz functions. In addition, we say $f$ is $(M,0)$-generalized Lipschitz if $(M,0,+)$-generalized Lipschitz without the non-negativity constraint. 
\end{definition}
Observe that $\mathcal{F}_{M,0}\triangleq\mathcal{F}_{M}$. We use the notation $\FLip{M,\alpha,(+)}$ to refer to both $\FLip{M,\alpha,+}$ and $\F_{M,0}$. We propose an alternative definition of $\FLip{M,\alpha,(+)}$ that does not involve derivatives.
\begin{proposition}[2-points definition]\label{prop:def2}
    Let $M,\alpha\geq 0$, $f\in\bar{\C}^0$, and $f$ piecewise $\bar{\C}^1$. In addition, if $\alpha\leq 1$, let $f\in \C^0$, and $f$ piecewise $\C^1$. Then, $f\in \FLip{M,\alpha,+}$, if and only if $f=0$, or $\forall x,y \in \R$:
    \begin{align}
            &|\tilde f(x)-\tilde f(y)|\leq M |x-y|, \label{eq:def1} \\
            & \begin{cases}
                f(x) \geq 0, & \text{ if } \alpha < 1, \\
                f(x) >0,     &\text{ if } \alpha \geq 1,
            \end{cases}
    \end{align}
     where $
        \tilde f(x)= 
        \begin{cases}
            f(x)^{1-\alpha}, & \text{if  }\alpha\neq 1,\hspace{1cm} \\
            \log(f(x)), &  \text{if } \alpha=1.
        \end{cases}  
        $   
        
    \noindent In addition, $f\in \FLip{M,0}$ if and only if it satisfies \eqref{eq:def1}.
\begin{proof}
     When $f'(x)$ exists and $f(x)\neq 0$, we have $\tilde{f}'(x) = \frac{f'(x)}{|\ba| f(x)^\alpha} \ \forall \alpha \geq 0$. In addition, on the interior of any interval over which $f$ is null, so are $\tilde f$ and $\tilde f'$.

    \noindent \textit{(Necessity of \eqref{eq:def1})}
    Let $f\in \FLip{M,\alpha,+}$. By assumptions on $f$ and definition of $\tilde f$, $\tilde f$ is continuous and piecewise $\C^1$. Indeed, if $\alpha \leq 1$, $f\in \C^0$ and is piecewise $\C^1$, and $\tilde f$ preserves these properties, with points of discontinuity the intersection of all points where $f$ is non-differentiable, and the points where $f$ switches from non-zero to zero values. On the other hand, if $\alpha>1$, then $\tilde f=0$ when $f=\infty$, hence, $\tilde f\in \C^0$ and is piecewise $\C^1$ despite $f\in \bar{\C}^0$, and piecewise $\bar{\C}^1$.
    
    Let $x_1<x_2<\ldots<x_K$ be the points at which $\tilde f$ is non-differentiable. Consider any interval $\mathcal{I}_i$ on which $\tilde{f}$ is differentiable, and non-zero. By the mean value theorem, it holds that $\forall x< y\in \mathcal{I}_i$, $\exists c \in (x,y)$ such that
    \begin{align*}
        \frac{|\tilde f(x)-\tilde f(y)|}{|x-y|}=|\tilde f'(c)|=\frac{|f'(c)|}{|\ba||f(c)^\alpha|} \overset{\eqref{eq:function_class}}{\leq} \frac{|\ba| M |f(c)^\alpha|}{|\ba||f(c)^\alpha|} = M.
    \end{align*}
In addition, on any interval over which $\tilde f=0$, $\tilde f$ is $M$-Lipschitz. Therefore, $\tilde f$ is piecewise $M$-Lipschitz. Consider now $x<x_i<\ldots<x_j<y$ where $1\leq i\leq j\leq K$; i.e., $x,y$ may belong to distinct intervals. Then,
    \begin{align*}
        \frac{|\tilde f(x)-\tilde f(y)|}{|x-y|}&\leq \frac{|\tilde f(x)-\tilde f(x_i)|+|\tilde f(x_i)-\tilde f(x_{i+1})|+\ldots+|\tilde f(x_j)-\tilde f(y)|}{|x-y|} \\&\leq  M\frac{x_i-x+x_{i+1}-x_{i}+\ldots+y-x_j}{y-x}=M.
    \end{align*}

    Finally, nonnegativity is due to Definition \ref{def:function_class} and positivity in the case $\alpha \geq 1$ arises from the limit case of \eqref{eq:def1}: if $f\in \FLip{M,\alpha,+}$, and $f(x)=0$ at some $x\in \R$, then $f=0$.
    
    \noindent \textit{(Sufficiency of \eqref{eq:def1})}
    Let $f\in \bC^0$, piecewise $\bC^1$ and satisfying \eqref{eq:def1}. At all $x\in\R$ where $f$ is differentiable and non-zero, it holds:
        \begin{align*}
        &|\tilde f'(x)|=\frac{|f'(x)|}{|\ba||f(x)|^\alpha}=\lim_{h\to 0}\frac{|\tilde f(x+h)-\tilde f(x)|}{h} \overset{\eqref{eq:def1}}{\leq} M \Rightarrow |f'(x)|\leq|\ba| M f(x)^\alpha.
    \end{align*}
    In addition, at all $x$ where $f(x)=0$, by nonnegativity of $f$, either $f'(x)$ does not exist, or $f'(x)=0$. Hence, $f\in\F_{M,\alpha,+}$.

    Finally, the case $\F_{M,0}$ follows the same argument, except there is no need to handle the non-negativity constraint.    
\qed \end{proof}
\end{proposition}
Proposition \ref{prop:def2} thus reduces generalized self-concordant functions to functions $f$ whose associated quantity $\tilde f''$ is simply Lipschitz continuous, which allows obtaining interpolation conditions almost straightforwardly for all basic function classes and their higher-order associated classes, once interpolation conditions are known for $\F_{M,0}$.

\subsection{Examples of function classes $\FLip{M,\alpha,(+)}$ and $\int^{(k)}\FLip{M,\alpha,(+)}$}
We propose a non-exhaustive list of function classes that fall under Definition \ref{def:function_class}. 

\subsubsection*{Zeroth and first-order classes}
Classical classes of functions whose description involves zeroth and first-order derivatives include the class of $M$-Lipschitz functions $\FLip{M,0} \subseteq \C^0$, i.e., functions satisfying $|f(x)-f(y)|\leq M |x-y|, \ \forall x,y \in \R$ and the class of $M$-smooth functions $\GLip{M,0}=\int \FLip{M,0} \subseteq \C^1$, i.e., functions satisfying $|f'(x)-f'(y)|\leq M |x-y|, \ \forall x,y \in \R$.

\subsubsection*{Second-order classes}
The class of Hessian Lipschitz functions can be expressed as $\HLip{M,0} :=\int^{(2)} \FLip{M,0} \subseteq \C^2$, and the one of convex generalized self-concordant functions \eqref{eq:generalized_sc} corresponds to $\HLip{M,\alpha,+} :=\int^{(2)} \FLip{M,\alpha,+} \subseteq \bar{\C}^2 $ if $\alpha >1$, $\HLip{M,\alpha,+} :=\int^{(2)} \FLip{M,\alpha,+} \subseteq \C^2 $ if $\alpha \leq 1$. This class includes, e.g., convex functions with Lipschitz Hessian, $\HLip{M,0,+}$, self-concordant functions, $\HLip{M,3/2,+}$, and quasi-self-concordant functions, $\HLip{M,1,+}$. Table \ref{tab:tab5} summarizes these classes.
{\small
\begin{table}[ht!]
    \centering
    \ra{1.5}
    \setlength{\abovecaptionskip}{0.5pt}
    \caption{Classes of generalized self-concordant functions $\HLip{M,\alpha,+}$. SC=Self-concordant.}
        \begin{tabular}{@{}lllll@{}}
            \toprule
            $\alpha$  & $\alpha\neq {1,2,\frac{3}{2}}$ & $\alpha=1$ & $\alpha=2$ & $\alpha=\frac{3}{2}$ \\
            \midrule 
                        1-pt definition& $|f'''(x)|\leq M|\ba|f''(x)^{\alpha}$&$|f'''(x)|\leq Mf''(x)$&$|f'''(x)|\leq Mf''(x)^{\alpha}$&$|f'''(x)|\leq 2Mf''(x)^{\alpha}$\\
            2-pts definition & $\tilde f''(x)=f''(x)^{1/\ba}$ & $\tilde f''(x)=\log(f''(x))$&$\tilde f''(x)=f''(x)^{-1}$&$\tilde f''(x)=f''(x)^{-1/2}$   \\ 
($\tilde f''(x)$ $M$-Lipschitz) &  & &&   \\ 
            Function classes & Generalized SC & Quasi-SC& / & SC\\
            References & \cite{hanzely2022damped,nesterov2018lectures,sun2019generalized}&\cite{bach2010self,doikov2023minimizing}&/&\cite{nesterov1994interior}\\      
            \bottomrule
        \end{tabular} 
    \label{tab:tab5}
\end{table}
}

\subsection{Properties of classes $\FLip{M,\alpha,(+)}$}
We show that the function classes defined in Definition \ref{def:function_class} satisfy some of the assumptions in Theorem \ref{lem:technique}:
\begin{proposition}\label{prop:properties}
    Let $\alpha,M\geq 0$. Then, $\FLip{M,\alpha,(+)}$, $\int \FLip{M,\alpha,(+)}$ and $\int^{(2)} \FLip{M,\alpha,(+)}$ are respectively order $0$, $1$ or $2$ connectable (Assumption \ref{assum1}). In addition, when $\alpha \leq 1$, $\int \FLip{M,\alpha,(+)}$ and $\int^{(2)} \FLip{M,\alpha,(+)}$ are also convex (Assumption \ref{assum:conv1}).
\end{proposition}
\begin{proof}
\textit{(Connectability)} Definition \ref{def:function_class}, involving derivatives, is pointwise and holds everywhere except for a finite set of points. Hence, by definition, $\FLip{M,\alpha,(+)}$ is order $0$-connectable, and the function $f$, juxtaposition of several functions in $\FLip{M,\alpha,(+)}$ might be non-differentiable only at the junctions between intervals. The same argument holds for $\int\FLip{M,\alpha,(+)}$ and $\int ^{(2)}\FLip{M,\alpha,(+)}$ (and $\int^{(k)}\FLip{M,\alpha,(+)}$), except that it involves higher-order derivatives.

\textit{(Convexity)} When $\alpha \leq 1$, $\int \FLip{M,\alpha,(+)}$ is convex  since, given $f_1, f_2 \in \int \FLip{M,\alpha,(+)}$, $\lambda \in [0,1]$, and $f = \lambda f_1 + (1-\lambda)f_2$ it holds:
\begin{align}\label{eq:proof_complet_a0}
\begin{split}
   |f''(x)| & = 
   |\lambda f_1''(x)+(1-\lambda) f_2''(x)|\\
   &\leq \lambda |f_1''(x)|+(1-\lambda) |f_2''(x)|\\ 
   & \overset{\eqref{eq:function_class}}{\leq} M_1 (\lambda f_1'(x)^\alpha+(1-\lambda) f_2'(x)^{\alpha}) \\
   & \leq  M_1(\lambda f_1'(x) +(1-\lambda) f_2'(x))^{\alpha}
\end{split}
\end{align}
by concavity of $g(t)=t^{\alpha}, \ t\geq 0, \ \alpha \leq 1$. Moreover, $\forall x\in \R$, if $f_1(x), f_2(x) \geq 0$ then $f(x) \geq 0$. The same arguments hold for $\int ^{(2)} \FLip{M,\alpha,(+)}$ (with $f'''$ and $f''$ instead of $f''$ and $f'$ in \eqref{eq:proof_complet_a0}). \qed \end{proof}

The other properties involved in Theorem \ref{lem:technique}, i.e., extremal interpolability and, when $\alpha>1$, extremal completability of $\FLip{M,\alpha,(+)}$ and $\int \FLip{M,\alpha,(+)}$ (Assumptions \ref{assum:extrgrad} and \ref{assum:conv}), require having access to interpolation conditions for both classes. Hence, we answer this question in Section \ref{sec:new_sec3}, after iteratively building such conditions. 

\section{Interpolation conditions for $\FLip{M,\alpha,(+)}$, $\int \FLip{M,\alpha,(+)}$, and $\int^{(2)}\FLip{M,0,(+)}$}\label{sec:new_sec3}
We now exploit the technique proposed in Section \ref{sec:intrp} to obtain interpolation conditions for the basic function classes of Section \ref{sec:fun_classes}. First, we obtain interpolation conditions for $\FLip{M,\alpha,(+)}$, before proving these classes are extremally interpolable (see Assumption \ref{assum:extrgrad}) and extremally completable (see Assumption \ref{assum:conv}), and relying on Theorem \ref{lem:technique}, to obtain interpolations conditions for $\int \FLip{M,\alpha,(+)}$. We then proceed iteratively to obtain interpolation conditions for $\int^{(2)} \FLip{M,0,(+)}$. It appears that the non-convex solver of Gurobi \cite{gurobi} can deal with all the interpolation conditions considered (see Section \ref{sec:results}). 
\subsection{Interpolation conditions for $\FLip{M,\alpha,(+)}$}\label{sec:interp_basic}
Interpolation conditions for classes $\FLip{M,\alpha,(+)}$ happen to be discretized versions of Definition \eqref{eq:def1}.
\begin{theorem}\label{prop:interp_basic}
    Let $M,\alpha\geq 0$. A set $S=\{(x_i,f_i)\}_{i\in [N]}$ is  $\FLip{M,\alpha,+}$-interpolable if and only if $f_i=0, \ \forall i\in [N]$, or $\forall i,j\in[N]$:
        \begin{align} 
            & |\tilde f_i-\tilde f_j|\leq M |x_i-x_j|, \label{eq:interp1}\\
            & \begin{cases}
                f_i\geq 0 & \text{ if } \alpha < 1, \\
                f_i>0 & \text{ if } \alpha \geq 1,
            \end{cases}
        \end{align}
        where $\tilde f_i= 
        \begin{cases}
            f_i^{1-\alpha}, & \text{if } \alpha\neq 1, \\
            \log(f_i), &\text{if } \alpha=1.
        \end{cases}$
        
        \noindent In addition, $S$ is $\FLip{M,0}$-interpolable if and only if it satisfies \eqref{eq:interp1}.
\begin{proof}
     By Proposition \ref{prop:def2}, these conditions are necessary for interpolation. Let us now prove they are also sufficient.

     Suppose $S$ satisfies \eqref{eq:interp1}. We construct a function $f\in \FLip{M,\alpha,(+)}$ interpolating $S$. Let $\tilde f(x) = \min_k \tilde f_k+M|x-x_{k}|$ and
     \begin{equation}
         f(x) = \begin{cases}
             \tilde f(x)^{\alpha-1} & \text{if } \alpha \neq 1, \\
             e^{\tilde f(x)} & \text{if } \alpha = 1.
         \end{cases}
     \end{equation}
     This function interpolates $S$, i.e.,
     \begin{align}
         \text{when }\alpha \neq 1~:~&f(x_i) = (\min_k \tilde{f}_k + M|x_i-x_k|)^{\alpha-1} = \tilde f_i^{\alpha-1}= f_i, \\
         \text{when }\alpha = 1~:~&f(x_i) =e^{\min_k \tilde{f}_k + M|x_i-x_k|} = e^{\tilde f_i}= f_i.
     \end{align}
     since by \eqref{eq:interp1}, we have
     \begin{align*}
         \tilde f_i\leq \tilde f_j+M|x_j-x_i|,\ \forall j\in [N].
     \end{align*}
     Moreover, we can show that $f\in \FLip{M,\alpha,+}$ since $\forall x,y\in \R$, we have 
     \begin{align*}
         |\tilde f(x)-\tilde f(y)|&=| \tilde f_{k_\star(x)}+M|x-x_{k_\star(x)}|-\tilde f_{j_\star(y)}-M|y-x_{j_\star(y)}||\\
         &\leq | \tilde f_{j_\star(y)}+M|x-x_{j_\star(y)}|-\tilde f_{j_\star(y)}-M|y-x_{j_\star(y)}||\\
         &=M| |x-x_{j_\star(y)}|-|y-x_{j_\star(y)}||\leq M | x-y|,
     \end{align*}
     where $k_\star(x)$ and $j_\star(y)$ are the optimal indices in the definition of $\tilde f(x)$ and $\tilde f(y)$. The last inequality follows from the reverse triangle inequality. In addition, if $f_i\geq 0 \ \forall i\in [N]$ then $f(x)\geq 0 \ \forall x\in \R$, and the same holds if $f_i>0$. Finally, $\tilde f(x)\in \C^0$ and piecewise $\C^1$.
 \qed \end{proof}
\end{theorem}

Theorem \ref{prop:interp_basic} covers the particular case of interpolation conditions for Lipschitz continuous functions $\FLip{M,0}$, see, e.g. \cite{valentine1945lipschitz}.
\begin{remark}
    By Lemma \ref{lemma:interpnof}, Theorem \ref{prop:interp_basic} furnishes interpolation conditions without function values for $\int \FLip{M,\alpha,(+)}$, and interpolation conditions without function values and first derivative for $\int^{(2)} \FLip{M,\alpha,(+)}$.
\end{remark}
\subsection{Interpolation conditions for $\int \FLip{M,\alpha,(+)}$}
First, we prove $\FLip{M,\alpha,(+)}$ is extremally interpolable and provide its extremal interpolants.
\begin{proposition}\label{prop:extrem_interp_basic}
    Let $\alpha,M\geq 0$. Then, $\FLip{M,\alpha,(+)}$ is extremally interpolable, and the extremal interpolants (Definition \ref{def:gextr}) of an $\FLip{M,\alpha,(+)}$-interpolable set  $S=\{(x_i,f_i)\}_{i=1,2}$, where $x_1<x_2$, are given by:
    \\
    $\bullet \quad [f_{\min}, \text{Case 1}] \ $ If $\alpha \geq 1$, if we consider $\FLip{M,0}$, or if $\alpha <1$ and $\tilde f_2+\tilde f_1\geq M(x_2-x_1)$:
    \begin{align*}
        f_{\min}(x)=\begin{cases}
       \nu\left(\tilde f_1-\frac{\ba}{|\ba|}M(x-x_1)\right)& x\in [x_1,z]\\
      \nu\left(\tilde f_2+\frac{\ba}{|\ba|}M(x-x_2)\right) & x \in [z,x_2]
    \end{cases}, \text{ where } z=\frac{x_1+x_2}{2}+\frac{\ba}{|\ba|}\frac{\tilde f_1-\tilde f_2}{2M}.
    \end{align*}
    $\bullet \quad [f_{\min}, \text{Case 2}] \ $ Else:
        \begin{align*}
        f_{\min}(x)=\begin{cases}
       \nu\left(\tilde f_1-\frac{\ba}{|\ba|}M(x-x_1)\right)& x\in [x_1,x_1+\frac{\tilde f_1}{M}]\\
       0& x\in [x_1+\frac{\tilde f_1}{M}, x_2-\frac{\tilde f_2}{M}]\\
      \nu\left(\tilde f_2+\frac{\ba}{|\ba|}M(x-x_2)\right) & x \in [x_2-\frac{\tilde f_2}{M},x_2]
    \end{cases}.
    \end{align*}
    $\bullet\quad  [f_{\max}, \text{Case 1}] \ $ If $\alpha \leq 1$, or if $\alpha >1$ and $\tilde f_2+\tilde f_1> M(x_2-x_1)$:
    \begin{align*}
        f_{\max}(x)=\begin{cases}
    \nu\left(\tilde f_1+\frac{\ba}{|\ba|}M(x-x_1)\right) & x\in [x_1,y]\\
    \nu\left(\tilde f_2-\frac{\ba}{|\ba|}M(x-x_2)\right) & x \in [y,x_2]
    \end{cases}, \text{ where } y=\frac{x_1+x_2}{2}-\frac{\ba}{|\ba|}\frac{\tilde f_1-\tilde f_2}{2M}.
    \end{align*}
    $\bullet \quad [f_{\max}, \text{Case 2}] \ $ Else:
        \begin{align*}
        f_{\max}(x)=\begin{cases}
    \nu\left(\tilde f_1+\frac{\ba}{|\ba|}M(x-x_1)\right)& x\in [x_1,x_1+\frac{\tilde f_1}{M}]\\
    \infty& x\in [x_1+\frac{\tilde f_1}{M}, x_2-\frac{\tilde f_2}{M}]\\
    \nu\left(\tilde f_2-\frac{\ba}{|\ba|}M(x-x_2)\right) & x \in [x_2-\frac{\tilde f_2}{M},x_2]
    \end{cases},
    \end{align*}
    where
        \begin{align*}
            &\tilde f_i= f_i^{1-\alpha} &&\text{ if } \alpha\neq 1,\hspace{2cm} \log\Par{f_i} \text{ otherwise},\\
         &       \nu(x)= x^{\ba} &&\text{ if } \alpha\neq 1,\hspace{2cm} e^x \text{ otherwise}.
    \end{align*} 
    \end{proposition}
    \begin{proof}
    First, observe that all quantities in Proposition \ref{prop:extrem_interp_basic} are well-defined, i.e., $x_1\leq z\leq x_2$ and $x_1\leq y \leq x_2$. Indeed, e.g.,
\begin{align*}
    z-x_1=\frac{x_2-x_1}{2}+\frac{\ba}{|\ba|}\frac{\tilde f_1-\tilde f_2}{2M}\geq 0
\end{align*}
    by satisfaction of \eqref{eq:def1}.

We now show that $f_{\min}, f_{\max}$, as illustrated on Figure \ref{fig:proof_extr_interp}, are the extremal interpolating envelopes of $S$ as defined in \eqref{eq:defgmin}.
\begin{figure}[ht!]
  \centering
  \begin{subfigure}[b]{\textwidth}
    \includegraphics[width=\textwidth]{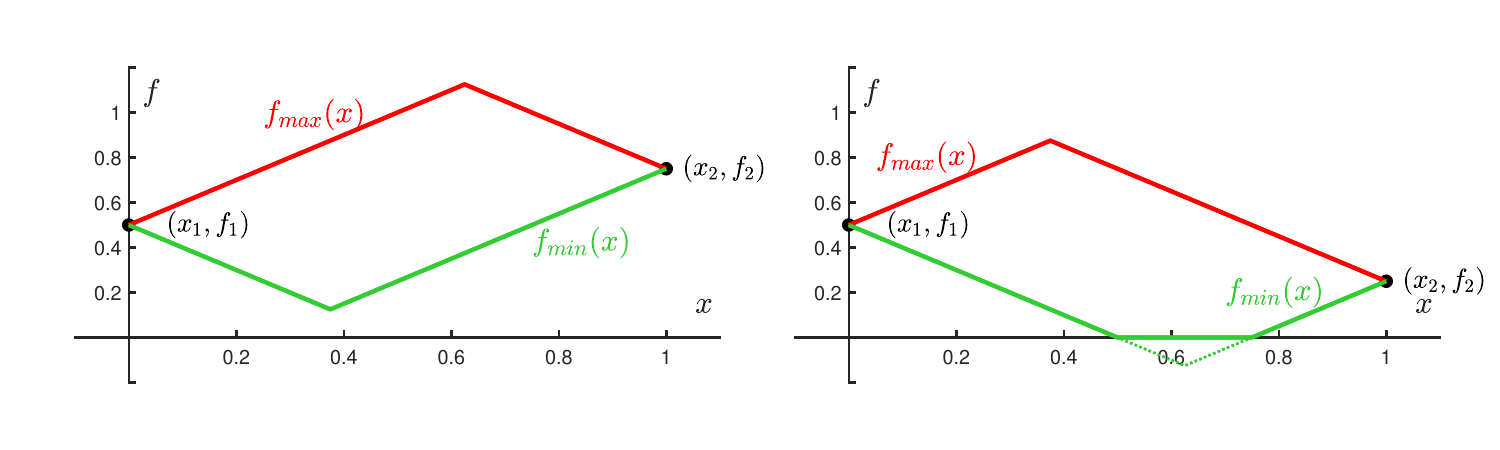}
    \caption{Illustration of the extremal interpolants when $\alpha<1$ ($\alpha=0$ in this example). The left-hand side figure considers $\tilde S_1$, and the associated extremal interpolants consist of two pieces, i.e., $\tilde f_{\min}$ is naturally always nonnegative. The right-hand side figure considers $\tilde S_2$, and the associated minimal interpolant consists of three pieces, i.e., $\tilde f_{\min}$ is forced to be zero while the $M$-Lipschitz minimal interpolant of $\tilde S_2$ is negative. }
    \label{subfig:alpha0}
  \end{subfigure}
  \hfill
  \begin{subfigure}[b]{\textwidth}
    \includegraphics[width=\textwidth]{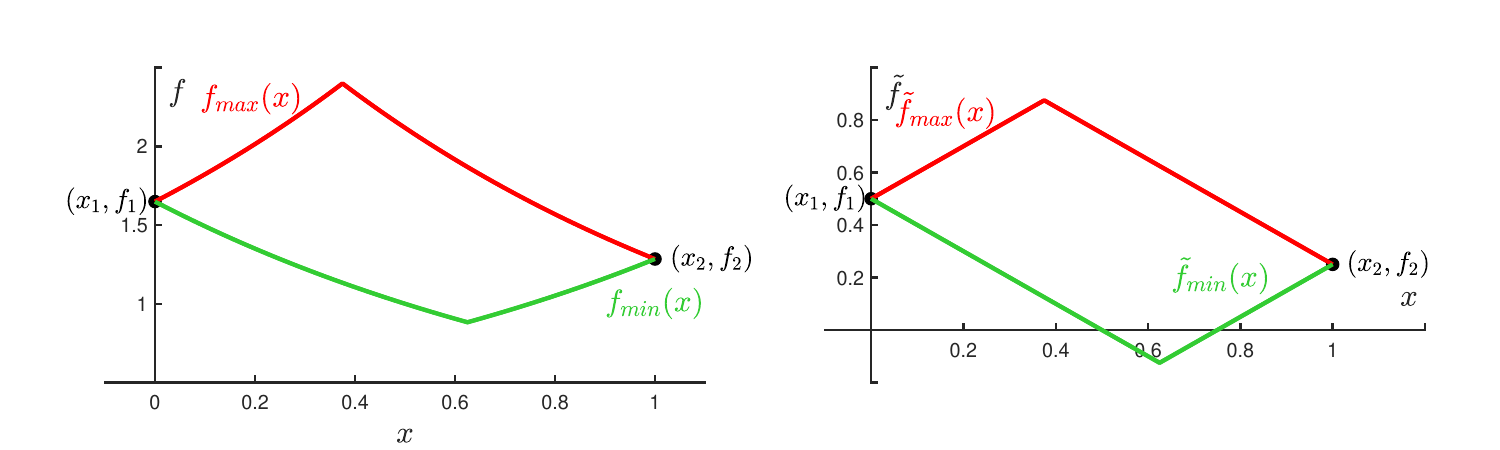}
    \caption{Extremal interpolants when $\alpha=1$, on $\tilde S_2$. The right-hand side figure displays $\tilde f_{\min}$ and $\tilde f_{\max}$, the extremal $M$-Lipschitz interpolants of $\tilde S_2$, even when negative. The left-hand side figure displays $f_{\min}$ and $f_{\max}$, the actual extremal interpolants of $S_2$, that are always nonnegative.}
    \label{subfig:alpha1}
  \end{subfigure}
  \hfill
  \begin{subfigure}[b]{\textwidth}
    \includegraphics[width=\textwidth]{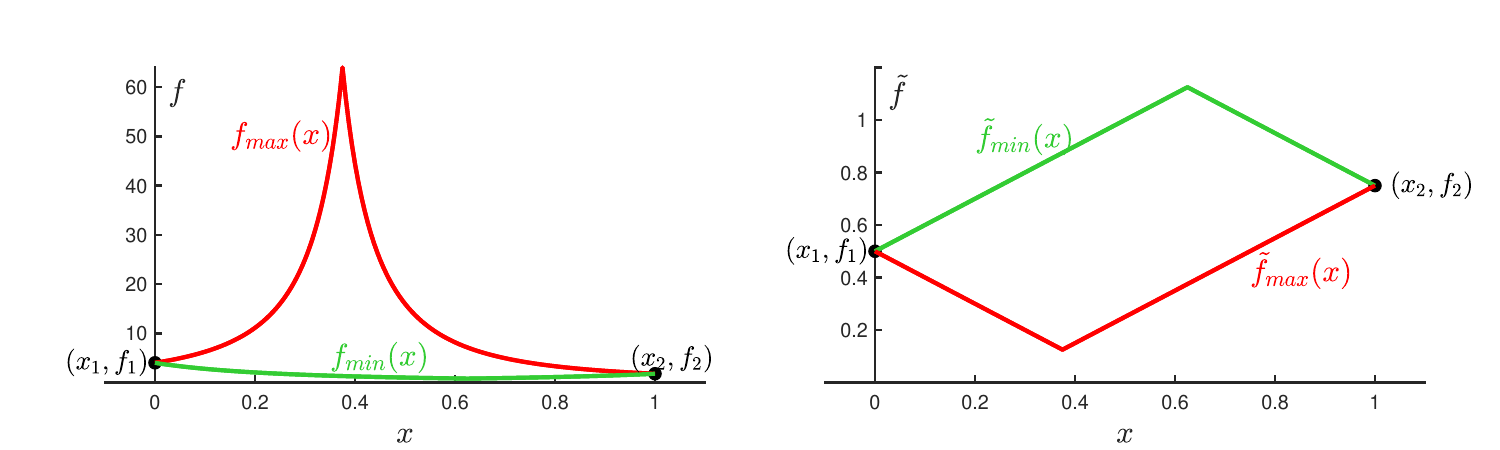}
    \caption{Extremal interpolants when $\alpha>1$ ($\alpha=\frac 32$ in this example) and $\tilde S_1$. The right-hand side figure displays $\tilde f_{\min}$ and $\tilde f_{\max}$, the permuted extremal $M$-Lipschitz interpolants of $\tilde S_1$, that are naturally nonnegative. The left-hand side figure displays $f_{\min}$ and $f_{\max}$, the actual extremal interpolants of $S_1$. }
    \label{subfig:alpha32_case1}
  \end{subfigure}
    \hfill
  \begin{subfigure}[b]{\textwidth}
    \includegraphics[width=\textwidth]{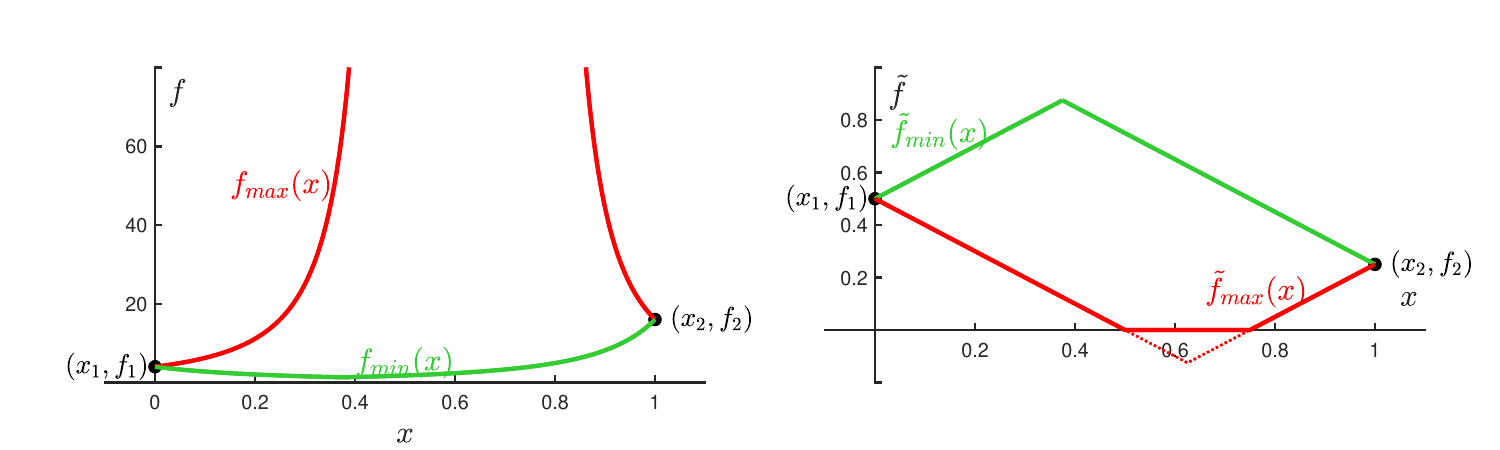}
    \caption{Extremal interpolants when $\alpha>1$ ($\alpha=\frac 32$ in this example) and $\tilde S_2$. The right-hand side figure displays $\tilde f_{\min}$ and $\tilde f_{\max}$, the permuted extremal $M$-Lipschitz interpolants of $\tilde S_2$, forced to be zero when these become nonnegative. The left-hand side figure displays $f_{\min}$ and $f_{\max}$, the actual extremal interpolants of $S_2$.}
    \label{subfig:alpha32_case2}
  \end{subfigure}
  \caption{Given $\alpha\geq 0$, $M=1$, and $S=\{(x_i,f_i)\}_{i=1,2}$, extremal interpolants $f_{\min}$, $f_{\max}$ (Definition \ref{def:gextr}) of $S$, depending on $\alpha$. The considered sets are either $\tilde{S}_1=\{(x_i,\tilde f_i)\}_{i=1,2}=\{(0,\frac{1}{2}),(1,\frac{3}{4})\}$, or $\tilde{S}_2=\{(x_i,\tilde f_i)\}_{i=1,2}=\{(0,\frac{1}{2}),(1,\frac{1}{4})\}$.}
  \label{fig:proof_extr_interp}
\end{figure}
It holds that $f_{\min}$, $f_{\max}$ interpolate $S$, since $\nu(\tilde{f_i})=f_i, \ i=1,2$. For any $x\in [x_1,x_2]$, they are also the extremal $f$ satisfying \eqref{eq:interp1} with respect to $S$, i.e., $f_{\min ~(\text{resp. } \max)}(x)=$ 
\begin{align*}
    \min~(\text{resp. } \max)_f \ f  \text{ s.t. } &|\tilde f-\tilde f_i|\leq M |x-x_i|, \ i=1,2,\\ & f\geq 0 \text{ (except for $\FLip{M,0}$).}
\end{align*}

\noindent \textit{Case $\alpha<1$ (Figure \ref{subfig:alpha0}).} Since $\ba>0$, $\nu(x)$ is an increasing nonnegative function of $x$, when $x\geq 0$, and either (i) $\nu(x)$ does not exists for $x<0$  (e.g., $\alpha=\frac 13$), (ii) $\nu(x)$ is negative for $x<0$  (e.g., $\alpha=0$), or (iii) $\nu(x)$ is a decreasing nonnegative function for $x<0$  (e.g., $\alpha=\frac 12$).

Hence, at all $x\in [x_1,x_2]$, computing the extremal interpolants of $S$, $f_{\min ~(\text{resp. } \max)}$, is equivalent to computing the extremal $M$-Lipschitz interpolants of $\{(x_i,\tilde f_i)\}_{i=1,2}$, whenever these are nonnegative. On the contrary, when these interpolants become negative, the extremal interpolants of $S$ are set to $0$.

Specifically, at all $x\in [x_1,x_2]$, $\tilde{f}_{\min}(x)$ takes the maximal value between $0$ and the lower boundaries $\tilde f\geq \tilde f_i- M |x-x_i|$. If $\tilde f_1+\tilde f_2\geq M(x_2-x_1)$, these boundaries are always larger than $0$, and $f_{\min}$ consists of two pieces. Else, $f_{\min}$ consists of three parts, since the associated function $\tilde f_{\min}(x)$ is set to $0$ whenever the lower boundaries $\tilde f_i- M |x-x_i|$ become negative. In addition, at all $x\in [x_1,x_2]$, $\tilde f_{\max}(x)$ takes the minimal value of the upper boundaries $\tilde f\leq \tilde f_i+ M |x-x_i|$, and always consists of two parts, since these boundaries are always nonnegative.

\noindent \textit{Case $\alpha=1$ (Figure \ref{subfig:alpha1}).} Since $\nu(x)$ is an increasing nonnegative function of $x$, $\forall x\in \R$, computing the extremal interpolants of $S$, $f_{\min~(\text{resp. } \max)}$, is equivalent to computing the extremal $M$-Lipschitz interpolants $\tilde f_{\min ~(\text{resp. } \max)}$ of $\{(x_i,\tilde f_i)\}_{i=1,2}$.

Specifically, at all $x\in [x_1,x_2]$, $\tilde{f}_{\min}(x)$ takes the maximal value between the lower boundaries $\tilde f\geq \tilde f_i- M |x-x_i|$, even when such boundaries become negative. Similarly, at all $x\in [x_1,x_2]$, $\tilde f_{\max}(x)$ takes the minimal value of the upper boundaries $\tilde f\leq \tilde f_i+ M |x-x_i|$. Both extremal interpolants always consist of two parts.

\noindent \textit{Case $\alpha>1$ (Figures \ref{subfig:alpha32_case1} and \ref{subfig:alpha32_case2}).} Since $\ba<0$, $\nu(x)$ is a \emph{decreasing} positive function of $x$, when $x> 0$, with asymptote in $x=0$, and either (i) $\nu(x)$ does not exists for $x<0$  (e.g., $\alpha=3$), (ii) $\nu(x)$ is negative for $x<0$  (e.g., $\alpha=2$), or (iii)~$\nu(x)$ is an increasing positive function for $x<0$  (e.g., $\alpha=\frac 32$).

Specifically, at all $x\in [x_1,x_2]$, $\tilde{f}_{\min}(x)$ takes the minimal value of the upper boundaries $\tilde f\leq \tilde f_i+ M |x-x_i|$ and always consists of two parts, since these boundaries are always positive. In addition, at all $x\in [x_1,x_2]$, $\tilde f_{\max}(x)$ takes the maximal value between $0$ and the lower boundaries $\tilde f\geq \tilde f_i- M |x-x_i|$. If $\tilde f_1+\tilde f_2> M(x_2-x_1)$, these boundaries are always strictly larger than $0$, and $f_{\max}$ consists of two pieces, i.e., the associated functions to these boundaries cross each other before meeting their respective asymptote. Else, $f_{\max}$ consists of three parts, since the associated function $\tilde f_{\max}(x)$ is set to $0$ whenever the lower boundaries $\tilde f_i- M |x-x_i|$ become negative. On this interval, $f_{\max}$ reaches $+\infty$.

To conclude the proof, it remains to show that $\FLip{M,\alpha,(+)}$ is extremally interpolable, that is, $f_{\min}, \ f_{\max} \in \FLip{M,\alpha,(+)}$. This holds since (i) $f_{\min}$ and $f_{\max}$ are nonnegative (and positive for $\alpha>1$) except for the extremal interpolants of $\FLip{M,0}$, (ii) $f_{\min}$ and $f_{\max}$ belong to $\bar{C}^0$ (or $\C^0$ if $\alpha \leq 1$), and are piecewise $\bar{\C}^1$ (or $\C^1$ when $\alpha \leq 1$), and (iii) $\FLip{M,\alpha,(+)}$ is order $0$-connectable and $f_{\min}, \ f_{\max}$ are functions by parts satisfying on each interval:
\begin{align*}
    |\nu\left(\tilde f_i\pm M(x-x_i)\right)'|&=M|\ba|\nu\left(\tilde f_i\pm M(x-x_i)\right)^{\alpha}, \quad i=1,2.
\end{align*} 
 \qed \end{proof}
 \FloatBarrier
\begin{remark}
    The cases $\FLip{M,0}$ and $\FLip{M,1,\alpha}$ are the only ones for which the extremal interpolants take a single expression.
\end{remark}
We build on these extremal interpolants to prove $\FLip{M,\alpha,(+)}$ is extremally completable when $\alpha >1$ (the case $\alpha \leq 1$ follows from Proposition \ref{prop:properties}).
\begin{proposition}\label{prop:extr_conn}
    Let $\alpha,M\geq 0$, and $\alpha>1$.  It holds that $\int \FLip{M,\alpha,(+)}$ is extremally completable.
\end{proposition}
The proof of Proposition \ref{prop:extr_conn} is deferred to Appendix \ref{app:properties}. Combining Propositions \ref{prop:extr_conn}, \ref{prop:extrem_interp_basic}, and \ref{prop:properties}, and Theorem \ref{lem:technique}, we obtain interpolation conditions for $\int \FLip{M,\alpha,(+)}$. For the sake of clarity, we consider the cases $\alpha=1$ (Theorem \ref{thm:quasi_self}) and $\alpha \neq 1$ (Theorem \ref{thm:IC_generalized_self_concordant}) separately.
\begin{theorem}[Interpolation conditions for $\int \FLip{M,1,+}$]\label{thm:quasi_self}
        Let $M\geq 0$. A set $S=\{(x_i,f_i,g_i)\}_{i\in [N]}$ is $\int \FLip{M,1,(+)}$-interpolable if and only if, $\forall i,j\in[N]$, $g_i \geq 0$ and 
\begin{align}
            f_j-f_i&\geq \frac{1}{M} (g_i+g_j)-\frac{2}{M}\sqrt{ g_i g_j} e^{-\frac{M}{2}(x_j-x_i)}.\label{condqsc3}
\end{align}
\end{theorem}
\begin{theorem}[Interpolation conditions for $\int \FLip{M,\alpha,(+)}$ ($\alpha\neq1)$]\label{thm:IC_generalized_self_concordant}
        Let $\alpha \neq 1,M \geq 0$.
    A set $S=\{(x_i,f_i,g_i)\}_{i\in [N]}$ is $\int \FLip{M,\alpha,+}$-interpolable if and only if, $\forall i,j\in [N]$, $g_i=0$ and $f_i=f_j$, or $\forall i,j\in[N]$:
            \begin{align} 
            &|\tilde g_i-\tilde g_j|\leq M |x_i-x_j|, \text{ and } g_i\geq 0 \text{ if } \alpha< 1, g_i>0 \text{ else},\label{eq:interp1_level1}\\
            &\text{If }  \tilde g_i+\tilde g_j\geq \frac{\ba}{|\ba|}M(x_j-x_i):\nonumber \\
            & f_j-f_i\geq \frac{\ba}{|\ba|M(\ba+1)} \left( \tilde g_i^{\ba+1}+\tilde g_j^{\ba+1}-\frac{1}{2^\ba}(\tilde g_i+\tilde g_j-\frac{\ba}{|\ba|}M(x_j-x_i))^{\ba+1} \right),\label{eq:interp1_level2}\\ 
            &\text{If }  \alpha < 1 \text{ and }\tilde g_i+\tilde g_j\leq M(x_j-x_i):\nonumber\\
            & f_j-f_i\geq \frac{1}{M(\ba+1)} \left( \tilde g_i^{\ba+1}+\tilde g_j^{\ba+1} \right).\label{eq:interp1_level3}
        \end{align}
        where $ \tilde g_i=g_i^{1/\ba}$. In addition, $S$ is $\int \FLip{M,0}$-interpolable if and only if it satisfies \eqref{eq:interp1_level2}.
\end{theorem}
The proofs of Theorem \ref{thm:quasi_self} and \ref{thm:IC_generalized_self_concordant} are deferred to Appendix \ref{app:proof_theo_GSC}.
\begin{remark}
    By Proposition \ref{lemma:interpnof}, Theorems \ref{thm:quasi_self} and \ref{thm:IC_generalized_self_concordant} furnish interpolation conditions without function values for $\int^{(2)} \FLip{M,\alpha,(+)}$.
\end{remark}
\begin{remark}\label{rem:nec_everywhere}
    Necessarily, the interpolation conditions of Theorems \ref{thm:quasi_self} and \ref{thm:IC_generalized_self_concordant} are satisfied everywhere by any function in $\int \FLip{M,\alpha,(+)}$. They are thus completely equivalent to the initial definition of $\int \FLip{M,\alpha,(+)}$, $|f''(x)|\leq |\ba|Mf'(x)^{\alpha}$, when imposed everywhere, even tough proving it from scratch is not straightforward.
\end{remark}
\subsection{Applications of Theorems \ref{thm:quasi_self} and \ref{thm:IC_generalized_self_concordant}}
We propose a list of corollaries to Theorems \ref{thm:quasi_self} and \ref{thm:IC_generalized_self_concordant}, furnishing interpolation conditions without function values to several specific classes $\int^{(2)}\FLip{M,\alpha,(+)}$. Whenever possible, we compare these conditions with existing results in the literature.

\subsubsection{Class $\HLip{M}$ of functions with $M$-Lipschitz Hessian}
We recover the interpolation conditions for $M$-smooth functions derived in \cite[Theorem 4]{taylor2017smooth}, or equivalently, the interpolation conditions without function values for $\HLip{M}$.
\begin{corollary}
    A set $\{(x_i,f_i,g_i,h_i)\}_{i\in[N]}$ is $\HLip{M}$-interpolable without functions values if and only if $\forall i,j \in [N]$:
    \begin{align}
        g_j - g_i -h_i (x_j-x_i) & \geq -\frac{M}{2}(x_j - x_i)^2 +\frac{1}{4M} ( h_j - h_i+M(x_j - x_i))^2.\label{eq:cismooth}
   \end{align}\label{prop:Hinterpnofi}
\begin{proof}
    Since $\HLip{M}=\int^{(2)}\FLip{M,0}$, consider Theorem \ref{thm:IC_generalized_self_concordant} as applied to $\int \FLip{M,0}$.
 \qed \end{proof}
\end{corollary}
\subsubsection{Class $\HLip{M,+}$ of $\mu$-strongly convex functions with $M$-Lipschitz Hessian}
\begin{corollary} \label{th:ICg_convex_hessian_Lipschitz}
    A set $\{(x_i,f_i,g_i,h_i)\}_{i\in[N]}$ is $\HLip{M,+}$-interpolable without functions values if and only if $\forall i,j \in [N]$:
    \begin{align}
       & g_j - g_i   \geq h_i (x_j-x_i)-\frac{M}{2}(x_j - x_i)^2 +\frac{1}{4M} ( h_j - h_i+M(x_j - x_i))^2\label{eq:condintermconvmu}\\
        &h_i\geq \mu\label{eq:condhplusmu}\\
        &\text{If }x_j-x_i \geq \frac{h_i+h_j-2\mu}{M},\text{ then }g_j-g_i \geq \mu (x_j-x_i)+\frac{(h_i-\mu)^2+(h_j-\mu)^2}{2M}. \label{eq:condifmu}
    \end{align}
\begin{proof}
    When $\mu=0$, $\HLip{M,+}=\int^{(2)} \FLip{M,0,+}$, hence Corollary \ref{th:ICg_convex_hessian_Lipschitz} is an application of Theorem \ref{thm:IC_generalized_self_concordant} as applied to $\FLip{M,0,+}$. The case $\mu\geq 0$ follows from the observation that $f \in \int^{(2)}\FLip{M,0,+} \Leftrightarrow f+\frac{\mu}{2}x^2 \in \HLip{M,+}$, since the Lipschitzness condition remains valid.
 \qed \end{proof}
\end{corollary}
\begin{remark}
    Condition \eqref{eq:condintermconvmu} is exactly \eqref{eq:cismooth} and hence is an interpolation condition for the $M$-smoothness of $S$, while \eqref{eq:condhplusmu} (lower bound on $h$) ensures $S$ to be consistent with an increasing function, of slope at least $\mu$. However, juxtaposing these conditions alone is not an interpolation condition for the class of $M$-smooth strongly monotone functions. One must add \eqref{eq:condifmu}, which links both properties, and ensures the existence of some $M$-Lipschitz continuous function $h(x)$ interpolating $h_i$ and $h_j$, and whose integral is smaller than or equal to $g_j - g_i$. 
\end{remark}
\subsubsection{Class $\Sc=\int^{(2)} \FLip{M,\frac{3}{2},+}$ of self-concordant functions}
\begin{corollary}\label{th:IC_sc}
    A set $\quadrnof$ is $\Sc$-interpolable if and only if, $\forall i,j \in [N]$, $h_i=0$ and $g_i=g_j$, or $\forall i,j \in [N]$,
    \begin{align} 
    &|\tilde h_j-\tilde h_i|\leq M| x_j-x_i| \text{ and }h_i> 0\label{sccond2}\\
    &\text{If $\tilde h_i+\tilde h_j> -M(x_j-x_i)$, then }
     g_j-g_i\geq \frac{1}{M\tilde h_i}+\frac{1}{M\tilde h_j}-\frac{4}{M(\tilde h_i+\tilde h_j+M(x_j-x_i))}, \label{sccond3}
    \end{align}
where $\tilde h_i=h_i^{-1/2}$.
\begin{proof}
    Consider Theorem \ref{thm:IC_generalized_self_concordant} as applied to $\int^{(2)} \FLip{M,\frac{3}{2},+}$.
 \qed \end{proof}
\end{corollary}
\begin{remark}
    The class of $M$-self-concordant functions is affine-invariant, namely, if $f(x) \in \Sc$ then $g(x)=f(ax+b) \in \Sc~~\forall a,b$. This property is reflected in the interpolation conditions since replacing $(x_i,x_j)$ by $(\frac{x_i-b}{a}, \frac{x_j-b}{a})$, $(g_i,g_j)$ by $a(g_i,g_j)$ and $(h_i,h_j)$ by $a^2 (h_i,h_j)$ does not modify them.
\end{remark}
\subsubsection{Class $\QSc=\int^{(2)}\FLip{M,1,+}$ of quasi-self-concordant functions} 
\begin{corollary}\label{th:IC_qsc}
    A set $\quadrnof$ is $\QSc$-interpolable if and only if, $\forall i,j \in [N]$, $h_i \geq 0$ and 
    \begin{align} 
     &g_j-g_i\geq \frac{h_i+h_j}{M}-\frac{2}{M}\sqrt{h_ih_j}e^{-\frac{M}{2}(x_j-x_i)}. \label{qsccond3}
    \end{align}
\begin{proof}
    Consider Theorem \ref{thm:quasi_self}.
 \qed \end{proof}
\end{corollary}
\begin{remark}
    The class of $M$-quasi-self-concordant is scale-invariant, namely, if $f \in \QSc$ then $c f \in \QSc~\forall c>0$. This property is reflected in the interpolation conditions since replacing $(g_i,g_j,h_i,h_j)$ by $c(g_i,g_j,h_i,h_j)$ does not modify them.
\end{remark}

It is known \cite[Lemma 2.7]{doikov2023minimizing} that quasi-self-concordant functions satisfy:
\begin{equation}\label{eq:qsc_existing}
    f'(y)-f'(x)-f''(x)(y-x) \leq \frac{1}{M} f''(x) (e^{M|y-x|}-M|y-x|-1).
\end{equation}
We rely on \eqref{qsccond3} to strengthen this condition (with the new term in blue):
\begin{lemma}
If $f\in \QSc$, then $\forall x,y \in \R$,
\begin{align}\label{eq:qsc_improved}
\begin{split}
    f'(y) - f'(x) - f''(x)(y-x) \leq & \frac{1}{M} f''(x)\Par{e^{M|y-x|}-M|y-x|-1} \\
    & \textcolor{\mycolor}{- \frac{1}{M} \Par{\sqrt{f''(y)} - \sqrt{f''(x)e^{M(y-x)}}}^2}.
    \end{split}
\end{align}
\begin{proof}
   By Corollary \ref{th:IC_qsc} and Remark \ref{rem:nec_everywhere}, $f$ satisfies, $\forall x,y\in \R$,
    \begin{align*}
        f'(y) - f'(x) & \leq -\frac{f''(x)}{M} - \frac{1}{M} \left(f''(y) - 2\sqrt{f''(x) f''(y)e^{M (y-x)}}\right) \\
        & = -\frac{f''(x)}{M} - \frac{1}{M} \left( -f''(x)e^{M(y-x)} +\Par{\sqrt{f''(y)} - \sqrt{f''(x)e^{M(y-x)}}}^2 \right) \\
        & \leq -\frac{f''(x)}{M} - \frac{1}{M} \big( -f''(x)e^{M|y-x|} + f''(x)M(|y-x|-(y-x)) \\&\quad+ \Par{\sqrt{f''(y)} - \sqrt{f''(x)e^{M(y-x)}}}^2 \big), \\
    \end{align*}
    where we used the identity $e^{|t|} -e^{t} -|t|+t \geq 0$.
     \qed \end{proof}
\end{lemma}
\subsection{Interpolation conditions for $\int^{(2)} \FLip{M,0}$}

We rely on Theorems \ref{thm:IC_generalized_self_concordant} and \ref{lem:technique} to obtain interpolation conditions with function values for $\HLip{M}$.
\subsubsection{Interpolation conditions for $\HLip{M}$}
\begin{theorem}\label{th:ICf_hessian_Lipschitz} A set $\quadr$ is $\HLip{M}$-interpolable, if, and only if, $\forall i,j\in [N] $
    \begin{equation}\label{eq:cond1}
        |\dhhi| \leq M |\dxi|
    \end{equation}
    if $\dhhi + M|\dxi| \neq 0$, then
    \begin{align}
        \dfi \geq& -\frac{M}{6} |\dxi|^3 \label{eq:cond2}  \\&+ \frac{\Par{\dgi +\frac{M}{2} |\dxi|(\dxi)}^2}{2\Par{h_j-h_i+M|\dxi|}}\nonumber \\&+ \frac{\Par{h_j-h_i + M|\dxi|}^3}{96M^2}, \nonumber
    \end{align}
    if $\dhhi + M|\dxi| = 0$, then
        \begin{align}
        \dgi=& -\frac{M}{2} |\dxi|(\dxi), \label{eq:hMxlimite0} \\
        \dfi=& - \frac{M}{6}|\dxi|^3.\label{eq:hMxlimite} 
    \end{align}
    \label{thm:H-intrp}
\end{theorem}
The proof of Theorem \ref{thm:H-intrp} is deferred to Appendix \ref{app:thmH}, and applies the method proposed in Section \ref{sec:intrp} on the extremal interpolants defined in Figure \ref{fig:method_hessian_Lip}.
\begin{figure}[ht!]
    \centering
    \includegraphics[width=0.7\linewidth]{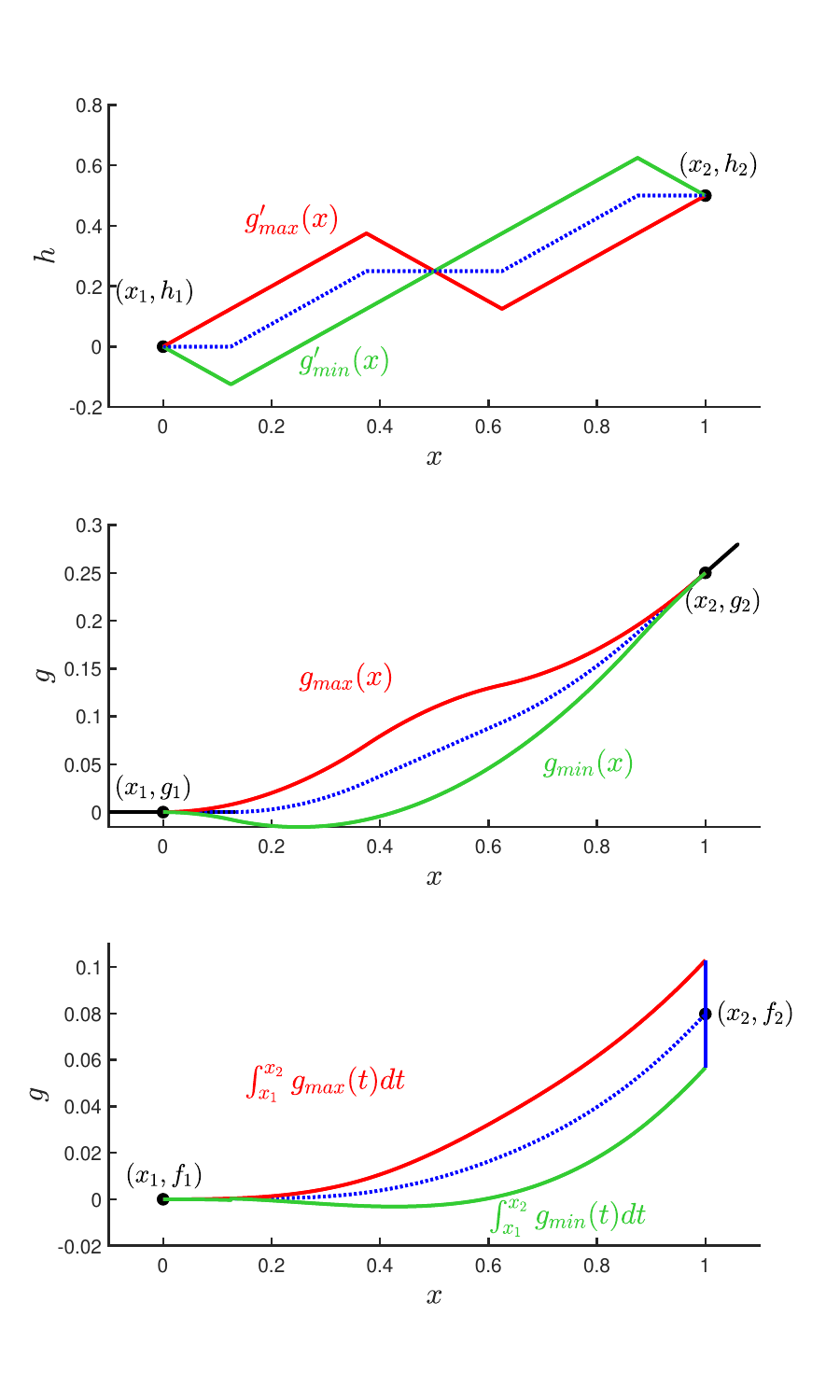}
    \caption{Derivation of interpolation conditions ensuring $S=\{(x_i,f_i,g_i,h_i)\}_{i=1,2}$ to be $\HH_{M}$-interpolable. The middle figure displays the extremal $M$-smooth interpolants $g_{\min}$ (green curve) and $ g_{\max}$ (red curve) of $\{(x_i,g_i,h_i)\}_{i=1,2}$. The upper figure displays their derivatives, interpolating  $\{(x_i,h_i)\}_{i=1,2}$. These derivatives are the intersection of slope $M$ functions, with given integral $g_2-g_1$. The lower figure displays the integrated extremal interpolants (red and green curves), which define an interval including all possible functions in $\HH_M$ interpolating $S$, except $f_2$. The interval $[f_1+\int_{x_1}^{x_2}g_{\min}(x) \mathrm{d}x, f_1+\int_{x_1}^{x_2}g_{\max}(x) \mathrm{d}x]$ (blue line) is exactly the interval of admissible values for $f_2$. Whenever $f_2$ belongs to this interval, there exists a convex combination of the integrals of $g_{\min}$ and $g_{\max}$ (blue dotted curve) which is in $\HH_{M}$ and interpolates $S$.}
    \label{fig:method_hessian_Lip}
\end{figure}
It is known \cite[Lemma 1.2.4]{nesterov2018lectures} that functions in $\HLip{M}$ satisfy a cubic bound:
    \begin{equation}\label{eq:cubic_bound} 
    -\frac{M}{6} |y-x|^3 \leq f(y) - f(x) - f'(x)(y-x) - \frac{f''(x)}{2} (y-x)^2 \leq \frac{M}{6} |y-x|^3 \ \forall x,y \in \R.
\end{equation}

We now show the conditions of Theorem \ref{thm:H-intrp} strengthen \eqref{eq:cubic_bound}, since the additional terms in \eqref{eq:cond2} are non-positive. In addition, we show that a set satisfying these conditions necessarily satisfies Condition \eqref{eq:cismooth}, as required by Theorem \ref{lem:technique}, and an upper bound on $f_j-f_i$, symmetrical to \eqref{eq:cond2}.
\begin{lemma} \label{lem:full_expr}
    Consider a set $S=\{(x_i,g_i,h_i,f_i)\}_{i\in[N]}$. If $S$ satisfies \eqref{eq:cond1} and \eqref{eq:cond2} for both pairs, then for both pairs $S$ satisfies \eqref{eq:cismooth}, the discretized version of \eqref{eq:cubic_bound}, and
    \begin{align}
        \dfi   \leq &\frac{M}{6} |\dxi|^3 \label{eq:cond22sup}\\&\textcolor{\mycolor}{- \frac{\Par{\dgi -M |\dxi|(\dxi)}^2}{2\Par{M|\dxi|-\dhhi}}}\nonumber\\& \textcolor{\mycolor}{-\frac{\Par{ M|\dxi|-(\dhhi)}^3}{96M^2}}.\nonumber
    \end{align}
\end{lemma}
The proof of Lemma \ref{lem:full_expr} is deferred to Appendix \ref{app:full_expr}. \\

We now propose a relaxation of \eqref{eq:cond2}, satisfied by all functions in $\HLip{M}$:
\begin{corollary} \label{cor:improved_cubic_bound}
    If $f$ is a univariate twice-differentiable function in $\HLip{M}$, then $\forall x,y \in \R$ such that $|f''(x)-f''(y)|\neq M|x-y|$, we have:
    \begin{align}
        f(y) - f(x) &- f'(x) (y-x) - \frac{f''(x)}{2}(y-x)^2 
         \leq \frac{M}{6} |y-x|^3\nonumber\\
         & \textcolor{\mycolor}{- \frac{M}{3}\Par{\frac{\left|f'(y)-f'(x)-f''(x)(y-x) - M|y-x|(y-x)\right|}{M}}^{\frac{3}{2}}} \label{eq:improved_cubic_bound2}
   \end{align}
   \begin{proof}
   By Remark \ref{rem:nec_everywhere}, $f \in \HLip{M}$ satisfies \eqref{eq:cond2} everywhere. In addition, for any $a,b,c\geq 0$ it holds:
       \begin{align}
           3 \frac{a}{c} + \frac{c^3}{b} \geq 4 \Par{\frac{a^3}{c^3}\frac{c^3}{b}}^{1/4}=4 \Par{\frac{a^3}{b}}^{1/4}.
       \end{align}
       Letting $ a = \frac 16 (f'(y)-f'(x)-f''(x)(y-x) - M|y-x|(y-x))^2$, $b= 96 M^2$, and
           $c=M|y-x|-(f''(x)-f''(y))$
       yields \eqref{eq:improved_cubic_bound2}.
    \qed \end{proof}
\end{corollary}
Inequality \eqref{eq:improved_cubic_bound2} will be exploited in Section \ref{sec:results}, in Theorem \ref{lem:new_descent_lemma}.
\section{Performance guarantees of optimization methods}\label{sec:results}
In the previous sections, we obtained interpolation conditions, i.e., ``exact descriptions", of different function classes (see Table \ref{tab:tab2}). 
We now exploit these interpolation conditions, analytically or numerically with the PEP framework presented in Section \ref{sect:introduction}, to propose new convergence guarantees for second-order methods on classes of univariate functions. Table \ref{tab:tab1} summarizes results of the literature that we were able to improve (either analytically or numerically) in the univariate case, or for which we proved tightness. We also propose to analyze a few setups that were not considered before.

We selected a subset of existing second-order optimization methods for which we could obtain fair and interesting comparisons with the literature.

All the numerical experiments are performed in Matlab and can be found at
\begin{center}
    \url{https://github.com/NizarBousselmi/Second-Order-Univariate-PEP}.
\end{center}
As mentioned, the non-convex PEPs are solved with the non-convex solver of Gurobi \cite{gurobi}, which can deal with non-convex quadratic but also non-convex general constraints..

\subsection{Improved descent lemma of the Cubic Regularized Newton method}
The Cubic Regularized Newton method (CNM) \cite{nesterov2006cubic} minimizes at each iteration the quadratic approximation of the function around the current iterate regularized by a cubic penalty. In the univariate case, the next iterate $x_{k+1}$ is computed as the global solution of the following problem:
\begin{equation}\label{eq:CNM}
    x_{k+1} = \mathrm{Arg} \min_x f(x_k) + f'(x_k) (x-x_k) + \frac12 f''(x_k) (x-x_k)^2 + \frac{M}{6} |x-x_k|^3. \tag{CNM}
\end{equation}
In the multivariate case, there is no analytical expression for the solution $x_{k+1}$ of this problem. In practice, one can obtain $x_{k+1}$ by solving a convex univariate problem \cite[Section 5]{nesterov2006cubic}. From the existing multivariate descent lemma and convergence rate from \cite{nesterov2006cubic}, we can write the following global rate of convergence on non-convex Hessian Lipschitz functions:
\begin{theorem}[\cite{nesterov2006cubic}, Theorem 1]
    The iterates of the Cubic Regularized Newton method \eqref{eq:CNM} on Hessian $M$-Lipschitz univariate functions satisfy
    \begin{equation}\label{eq:old_descent_Lemma}
        f(x_k) - f(x_{k+1}) \geq \frac{M}{12} \max \left\{ \sqrt{\frac{|f'(x_{k+1})|}{M}}, - \frac{2}{3}\frac{f''(x_{k+1})}{M} \right\}^3.
    \end{equation}
    Moreover, if the function is bounded below by $f_\star$, then
    \begin{equation}\label{eq:CNM_old_sublinear}
        \min_{k=1,\ldots,N} |f'(x_k)| \leq 4 M^{\frac13} \left( \frac{3(f(x_0) - f_\star)}{2N} \right)^{\frac23}.
    \end{equation}
\end{theorem}

Exploiting the refined description of Hessian Lipschitz functions, and in particular the improved cubic bound of Corollary \ref{cor:improved_cubic_bound} allows to improve this original descent lemma \eqref{eq:old_descent_Lemma} for univariate functions by a factor of 5. It results in an improvement of factor $5^{\frac23}\approx 2.9$ on the sub-linear global convergence rate \eqref{eq:CNM_old_sublinear} of \ref{eq:CNM} on univariate functions. 
\begin{theorem}[Improved descent lemma and gradient convergence rate]\label{lem:new_descent_lemma}
    The iterates of the Cubic Regularized Newton method \eqref{eq:CNM} on Hessian $M$-Lipschitz univariate functions satisfy
    \begin{equation}\label{eq:improved_descent_lemma}
        f(x_k) - f(x_{k+1}) \geq \frac{5M}{12} \sqrt{\frac{|f'(x_{k+1})|}{M}}^3.
    \end{equation}
    Moreover, if the function is bounded below by $f_\star$, then
    \begin{equation}\label{eq:CNM_new_sublinear}
        \min_{k=1,\ldots,N} |f'(x_k)| \leq \frac{4M^{\frac13}}{5^{\frac23}} \left( \frac{3(f(x_0) - f_\star)}{2N} \right)^{\frac23}.
    \end{equation}
    \begin{proof}
    We first establish \eqref{eq:improved_descent_lemma}. The \ref{eq:CNM} iterates satisfy the following properties (see \cite[Equation (2.5), Proposition 1, and Lemmas 2 and 3]{nesterov2006cubic})
    \begin{align}
        f'(x_k) + f''(x_k)(x_{k+1}-x_k) + \frac{M}{2} (x_{k+1} - x_k) |x_{k+1} - x_k| = 0, \label{eq:FOO} \\
        f''(x_k) + \frac{M}{2} |x_{k+1} - x_k| \geq 0, \label{eq:SOO} \\
        f'(x_k)(x_{k+1} - x_k) \leq 0, \label{eq:proof_CNM1}\\
        M|x_{k+1}-x_k|^2 \geq |f'(x_{k+1})| \label{eq:proof_CNM2}.
    \end{align}
    Exploiting the new term (in blue) from the improved cubic bound \eqref{eq:improved_cubic_bound2} of Corollary \ref{cor:improved_cubic_bound} yields
        \begin{align*} 
        \begin{split}
            f(x_{k+1})  \overset{\eqref{eq:improved_cubic_bound2}}{\leq} & f(x_k) + f'(x_k) (x_{k+1} - x_k) + \frac{f''(x_k)}{2} (x_{k+1} - x_k)^2  + \frac{M}{6} |x_{k+1}-x_k|^3 \\
            & \textcolor{\mycolor}{- \frac{M}{3}\Par{\frac{\left| f'(x_{k+1})-f'(x_k)-f''(x_k)(x_{k+1}-x_k) - \frac{M}{2}(x_{k+1}-x_k)|x_{k+1}-x_k|\right|}{M}}^{\frac{3}{2}}} \\
            \overset{\eqref{eq:FOO}}{=} & f(x_k) + \frac{ f'(x_k)}{2} (x_{k+1} - x_k) - \frac{M}{12} |x_{k+1} - x_k|^3  \textcolor{\mycolor}{-\frac{M}{3}\Par{\frac{\left|f'(x_{k+1})\right|}{M}}^{\frac{3}{2}}} \\
            \overset{\eqref{eq:proof_CNM1}}{\leq} & f(x_k) - \frac{M}{12} |x_{k+1} - x_k|^3 \textcolor{\mycolor}{-\frac{M}{3}\Par{\frac{\left|f'(x_{k+1})\right|}{M}}^{\frac{3}{2}}}\\
            \overset{\eqref{eq:proof_CNM2}}{\leq} & f(x_k) - \frac{M}{12} \Par{\frac{|f'(x_{k+1})|}{M}}^{\frac{3}{2}} \textcolor{\mycolor}{-\frac{M}{3}\Par{\frac{\left|f'(x_{k+1})\right|}{M}}^{\frac{3}{2}}} \\
            = &  f(x_k) - \frac{\textcolor{\mycolor}{5}M}{12} \Par{\frac{|f'(x_{k+1})|}{M}}^{\frac{3}{2}}
        \end{split}
        \end{align*}
        establishing the result. Telescoping this new descent lemma \eqref{eq:improved_descent_lemma} yields the convergence rate in \eqref{eq:CNM_new_sublinear}
        \begin{align*}
        \begin{split}
            f(x_0) - f_\star \geq f(x_0) - f(x_N) = \sum_{k=0}^{N-1} f(x_k)-f(x_{k+1}) & \overset{\eqref{eq:improved_descent_lemma}}{\geq} \frac{5M}{12} \sum_{k=0}^{N-1} \Par{\frac{| f'(x_{k+1})|}{M}}^{\frac{3}{2}} \\
            & \geq \frac{5MN}{12}  \min_{k=1,\ldots,N} \Par{\frac{| f'(x_{k})|}{M}}^{\frac{3}{2}}.
        \end{split}
        \end{align*}
         \qed \end{proof}
\end{theorem}
        Finally, the bound on the Hessian follows from the fact that the measure on the gradient of the iterates is always greater than the measure on the Hessian of the iterates for univariate functions.
        \begin{theorem}
            The iterates of the Cubic Regularized Newton method \eqref{eq:CNM} on Hessian $M$-Lipschitz univariate functions satisfy
            \begin{equation} \label{eq:goalhessian}
            \sqrt{\frac{|f'(x_{k+1})|}{M}} \geq -\frac{2}{3M}f''(x_{k+1}).   
            \end{equation}
            \begin{proof}
            Suppose w.l.o.g.\@ that $|x_{k+1}-x_k| = x_{k+1}-x_k\geq 0$ and recall that 
            \begin{equation}\label{eq:old_4.1}
                |f''(x_{k+1}) - f''(x_k) |\leq M |x_{k+1} - x_k|
            \end{equation}
            therefore, we have
            {\small
            \begin{align*}
                |f'(x_{k+1})| & \geq -f'(x_{k+1})\\ & \overset{\eqref{eq:cismooth}}{\geq} -f'(x_k)-\frac{M (x_{k+1}-x_k)^2}{4}+\frac{(f''(x_{k+1})-f''(x_{k}))^2}{4M}-\frac{1}{2}(f''(x_{k+1})+f''(x_{k}))(x_{k+1}-x_k)\\
                &\overset{\eqref{eq:FOO}}{=}\frac{M (x_{k+1}-x_k)^2}{4}+\frac{(f''(x_{k+1})-f''(x_{k}))^2}{4M}-\frac{1}{2}(f''(x_{k+1})-f''(x_{k}))(x_{k+1}-x_k)\\
                &=\frac{1}{4M}(f''(x_{k+1})-f''(x_{k})-M (x_{k+1}-x_k))^2 \\
                \Rightarrow \sqrt{\frac{|f'(x_{k+1})|}{M}} & \geq \frac{1}{2M}|f''(x_{k+1})-f''(x_{k})-M |x_{k+1}-x_k|| \\
                & \overset{\eqref{eq:old_4.1}}{\geq} \frac{1}{2M}\left(M |x_{k+1}-x_k|-f''(x_{k+1})+f''(x_{k})\right) \\
                & = \frac{1}{2M}\left(\frac{2}{3} M |x_{k+1}-x_k|+ \frac{1}{3} M |x_{k+1}-x_k|-f''(x_{k+1})+f''(x_{k})\right) \\
                & \overset{\eqref{eq:SOO},\eqref{eq:old_4.1}}{\geq} -\frac{4}{3}  f''(x_{k})+ \frac{1}{3} (f''(x_{k})-f''(x_{k+1})) -f''(x_{k+1})+f''(x_{k})) \\
                & = -\frac{2}{3M}f''(x_{k+1}).
            \end{align*}
            }
         \qed \end{proof}
\end{theorem}

For a single iteration, the improved descent lemma \eqref{eq:improved_descent_lemma} is satisfied with equality for the function $f(x) =  M\frac{x^3}{6} - \frac{x^2}{2}$ on the starting point $x_0=0$ (see Appendix \ref{app:improved_descent_lemma_is_tight}).

\begin{remark}\label{rem:CNM_gradient_dominated}
    Thanks to the improved descent lemma \eqref{eq:improved_descent_lemma}, we can also improve the results on the convergence rate of \ref{eq:CNM} on univariate gradient-dominated functions \cite[Section 4.2]{nesterov2006cubic}. We can replace quantities $\hat{\omega}$ and $\tilde{\omega}$ respectively by $\frac{\hat{\omega}}{5^2} $ and $5^4\tilde{\omega}$ in \cite[Theorems 6.1 and 6.2]{nesterov2006cubic}.
\end{remark}

\begin{remark}
We can also analyze an inexact variant of CNM with arbitrary $g$ and $h$ instead of exact derivatives $f'$ and $f''$. Then, we obtain an improved descent lemma:
\begin{align}
\begin{split}
    f(x_{k+1}) \leq & f(x_k) + \frac23 \sqrt{\frac{12}{M}} |f'(x_k)-g|^{\frac32} +\frac{72}{M^2}|f''(x_k)-h|^3-\frac{M}{36} |x_{k+1}-x_k|^3 \\
    & - \frac{1}{3\sqrt{M}} \Par{|f'(x_{k+1})| - |f'(x_k)-g| - |(f''(x_k)-h)(x_{k+1} - x_k)|}^\frac32,
\end{split}
\end{align}
which can be compared to \cite[Theorem C.1]{chayti2023unified} and has better factors.
\end{remark}
\begin{remark}
    Since our interpolation conditions only hold for univariate functions, Theorem \ref{lem:new_descent_lemma} only holds in this setting. However, we believe and conjecture that it also holds in the multivariate case. To support this conjecture, we computed the decrease $f(x_k) - f(x_{k+1})$ of CNM on functions of the form $f(x):\R^d \rightarrow \R: x\mapsto f(x) = \sum_{i=1}^n a_i x_i^3 + \frac12 x^T Q x + b^T x + c$ for $10^6$ values for randomly generated parameters $a$, $Q$, $b$, $c$ and $d=1,\ldots,10$, and observed that the descent \eqref{eq:improved_descent_lemma} (with the norm of the gradient instead of the absolute value of the derivative) with the new factor 5 always held.
\end{remark}

\subsubsection*{Tightness of multiple steps}
Even if the improved descent lemma is tight for a single iteration, the worst-case decrease after multiple iterations can be better than predicted by \eqref{eq:improved_descent_lemma} only. Indeed, a better bound can result from the use of a more sophisticated combination of the interpolation conditions. The PEP methodology automatically computes numerically this optimal combination. Figure \ref{fig:CNM_descent_lemma} compares the worst-case performance for multiple iterations of \ref{eq:CNM} computed by PEP (blue dots) and as predicted by \cite[Theorem 1]{nesterov2006cubic} (red line) and by Theorem \ref{lem:new_descent_lemma} (blue line). We also plot the gradient residual of \ref{eq:CNM} for the univariate function $f^{(4)}$ of \cite{cartis2010complexity} for which \ref{eq:CNM} shows a $\mathcal{O}(1/k^{\frac23})$ convergence.
\begin{figure}[ht!]
    \centering
    \figureH{
        %\includesvg[width = 0.97\textwidth]{figures/CNM_new_descent_lemma_several_iterV6}
        \includegraphics[width = 1\textwidth]{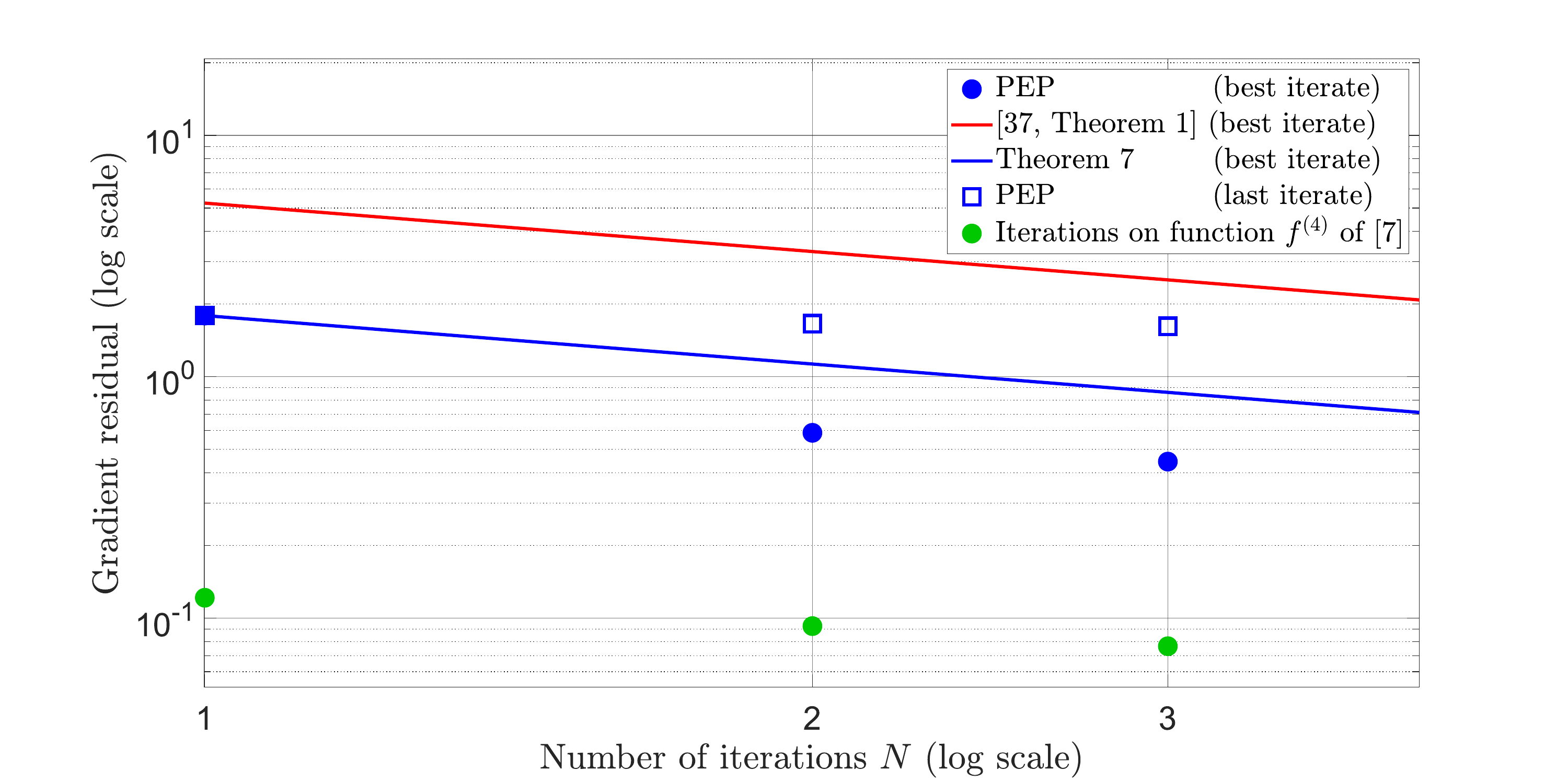}
    }
    \setlength{\belowcaptionskip}{-5pt}
    \setlength{\abovecaptionskip}{-2pt}
    \caption{Worst-case performance of the best $\min_{k=1,\ldots,N}  |f'(x_k)|$ and last $|f'(x_N)|$ iterates of the Cubic Newton method on Hessian $M$-Lipschitz functions for varying number of iterations $N$ and $M = f_0 - f_N = 1$.}
    \label{fig:CNM_descent_lemma}
\end{figure}
Figure \ref{fig:CNM_descent_lemma} illustrates the two possible kinds of conservatism of a performance guarantee. On one hand, we must use the correct interpolation conditions (blue lines instead of red lines), and on the other hand, we must combine them optimally (blue dots instead of blue line).

\subsubsection*{Lower bound on the general worst-case performance}
Since we solve the PEP in the univariate case, we do not have the general worst-case performance of the method. However, the univariate worst-case is an example of a function for which the method is slow, in other words, it provides a lower bound on the general worst-case performance. In \cite{cartis2010complexity}, they exhibited a univariate function $f^{(4)}$ for which \ref{eq:CNM} has the $\mathcal{O}(1/k^{\frac23})$ convergence (green dots in Figure \ref{fig:CNM_descent_lemma}), thus providing a lower bound on the exact worst-case performance. With PEP, we can provide numerically a worse function (blue dots in Figure \ref{fig:CNM_descent_lemma}) and therefore close the gap between the best known upper (red line) and lower bounds on the general multivariate worst-case performance of \ref{eq:CNM}.

\subsubsection*{Last-iterate rate of convergence}
PEP allows the study of different performance criteria. In particular, in addition to the classical best iterate analysis of \cite[Theorem 1]{nesterov2006cubic}, we can also consider the last-iterate convergence of \ref{eq:CNM} with PEP without additional effort. In Figure \ref{fig:CNM_descent_lemma} the last-iterate convergence (blue squares) has a slower order of convergence than the best iterate convergence. Such last-iterate analysis of \ref{eq:CNM} has never been done and is straightforward with PEP.

\subsubsection*{Parameter selection}
A tight worst-case performance guarantee allows selecting the optimal parameters optimizing the worst-case. We consider the Cubic Regularized Newton method with step size $\alpha$ (CNM-$\alpha$) defined as the solution of the modified CNM iteration
\begin{equation}
    x_{k+1} = \mathrm{Arg} \min_x f(x) + f'(x_k) (x-x_k) + \frac12  f''(x_k) (x-x_k)^2 + \frac{M}{6\alpha} |x-x_k|^3.
\end{equation}
Interpolation conditions together with PEP allow to analyze the worst-case performance of CNM-$\alpha$ for different values of $\alpha$. In particular, we can select the step size $\alpha$ optimizing the worst-case performance. Figure \ref{fig:CNM_alpha} shows the worst-case performance of one iteration of CNM-$\alpha$ for varying step size $\alpha$ and different values of Hessian Lipschitz constant $M$.
\begin{figure}[ht!]
    \centering
    \figureH{
        %\includesvg[width = 0.97\textwidth]{figures/CNM_varying_alphaV4}
        \includegraphics[width = 1\textwidth]{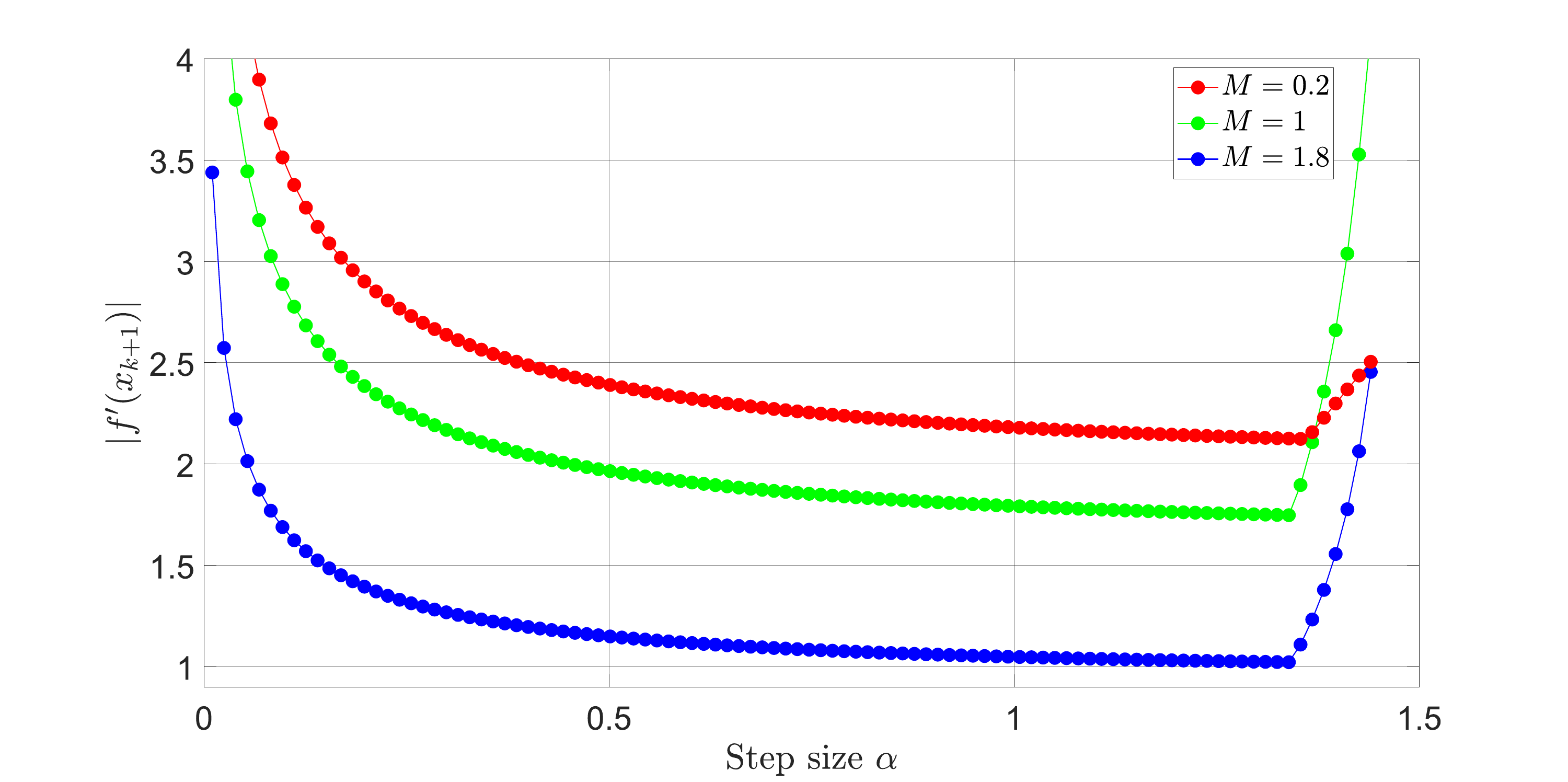}
    }
    \setlength{\belowcaptionskip}{-5pt}
        \setlength{\abovecaptionskip}{-0pt}
    \caption{Worst-case performance $|f'(x_{k+1})|$ of the Cubic Regularized Newton method for different step size $\alpha$ on Hessian $M$-Lipschitz functions with $f_k - f_{k+1} = 1$ and different values of $M$.}
    \label{fig:CNM_alpha}
\end{figure}
Following a similar reasoning that \cite[Section 3.5]{taylor2017smooth}, it can be shown that this worst-case scales with $(f_k-f_{k+1})^\frac{2}{3}$. Further analysis of such numerical results can lead to the development of new optimized methods.

\subsection{Gradient Regularized Newton method on convex and strongly convex functions}
The Gradient Regularized Newton method 2 (GNM2) \cite{mishchenko2023regularized} is a variant of the Cubic Newton method that achieves a $\mathcal{O}(1/k^2)$ global convergence (faster than the $\mathcal{O}(1/k^\frac23)$ of \ref{eq:CNM}) and, unlike \ref{eq:CNM}, has a fully explicit iteration
\begin{equation}\label{eq:GNM_1/2}
    x_{k+1} = x_k  - \frac{f'(x_k)}{f''(x_k)+\sqrt{\frac{M}{2} |f'(x_k)|}}. \tag{GNM2}
\end{equation}

\ref{eq:GNM_1/2} has a superlinear convergence on strongly convex functions formalized in the following theorem (proved in the multivariate case):
\begin{theorem}[Theorem 2.7 of \cite{mishchenko2023regularized}]
    The iterations of \eqref{eq:GNM_1/2} on $\mu$-strongly convex and Hessian $M$-Lipschitz univariate functions satisfy
    \begin{equation}\label{eq:bound_GRN_local}
        |f'(x_{k+1})| \leq \frac{M}{2\mu^2} |f'(x_k)|^2 + \frac{1}{\mu} \sqrt{\frac{M}{2}} |f'(x_k)|^\frac{3}{2}.
    \end{equation}
    Moreover, if $|f'(x_0)| < \frac{\mu^2}{2M}$, then
    \begin{equation}
        \frac{|f'(x_{k+1})|}{|f'(x_{k})|} \leq \sqrt{\frac{2M}{\mu^2} |f'(x_{k})|} < 1.
    \end{equation}
\end{theorem}
Figure \ref{fig:GRN_local_strongly} compares the worst-case local convergence of \ref{eq:GNM_1/2} on $\mu$-strongly convex Hessian $M$-Lipschitz univariate functions from \eqref{eq:bound_GRN_local} (red dots) and from PEP (blue dots) for varying number of iterations $N$, initial gradient $|f'(x_0)| \leq R =\frac{\mu^2}{4M}$, $\mu = 0.1$ and $M=1$.
\begin{figure}[ht!]
    \centering
    \figureH{
    %\includesvg[width=\textwidth]{figures/GRN_local_stronglyV6}
    \includegraphics[width=\textwidth]{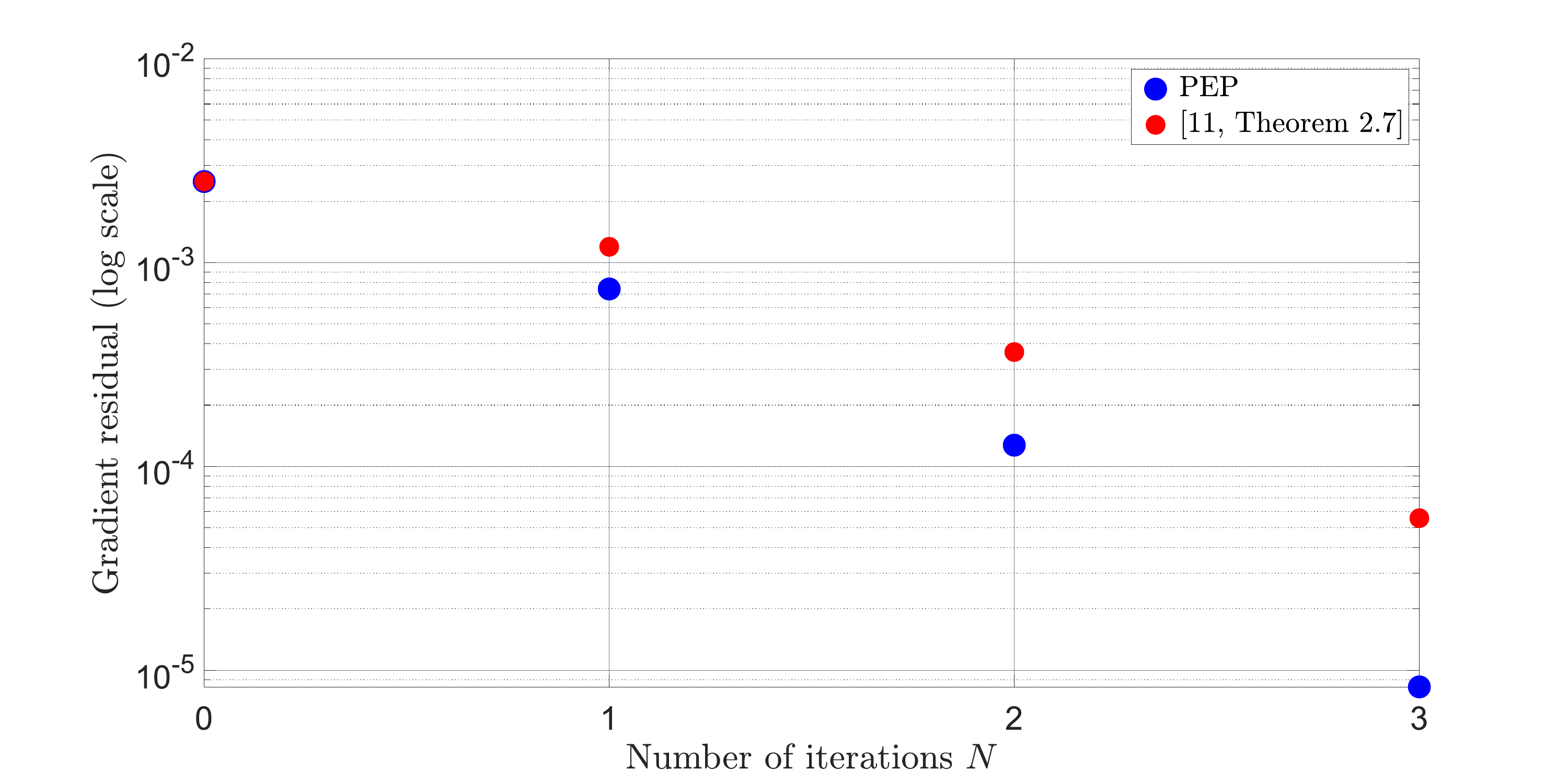}
    }
    \setlength{\abovecaptionskip}{-5pt}
    \setlength{\belowcaptionskip}{-5pt}
    \caption{Worst-case performance of \ref{eq:GNM_1/2} for varying number of iterations $N$ and $|f'(x_0)|\leq \frac{\mu^2}{4M}$, $M = 1$, and $\mu = 0.1$.}
    \label{fig:GRN_local_strongly}
\end{figure}

\subsection{Local quadratic convergence of Newton's method on non-convex Hessian Lipschitz functions}
Newton's method benefits from a local quadratic convergence. The following result from \cite{nesterov2018lectures} formally states this rate of convergence.
\begin{theorem}[Theorem 1.2.5 of \cite{nesterov2018lectures}]\label{th:local_conv_Newton}
    If
    \begin{itemize}
        \item $f$ has a $M$-Lipschitz continuous Hessian,
        \item $\exists x_\star$ such that $\nabla f(x_\star) = 0$, $\nabla^2 f(x_\star)=\mu I \succ 0$,
        \item $\frac{M}{\mu} ||x_0 - x_\star|| \leq \frac{2}{3}$,
    \end{itemize}
    then all Newton iterations  $x_{k+1} = x_k - \nabla^2 f(x_k)^{-1} \nabla f(x_k)$ satisfy 
    \begin{align}
        ||x_{k+1} - x_\star|| & \leq  \frac{  \frac{M}{\mu} ||x_k-x_\star|| ^2 }{2\left(1-\frac{M}{\mu}||x_k-x_\star||\right)} \label{eq:bound_x}
    \end{align}
\end{theorem}
This theorem is well established, but, to the best of our knowledge, its tightness remained an open question. Relying on a PEP analysis, we identified a simple function attaining \eqref{eq:bound_x}, establishing its tightness for any number of iterations.
\begin{theorem}\label{th:local_quadratic_tight}
    Theorem \ref{th:local_conv_Newton} (i.e., \cite[Theorem 1.2.5]{nesterov2018lectures}) is tight and attained by the following univariate cubic by parts function:
    \begin{equation}
        f(x) = -M\frac{|x|^3}{6} + \mu \frac{x^2}{2}.
    \end{equation}
    \begin{proof}
        We have {\small $x_\star = 0$} and {\small $ x_{k+1} = \frac{-\frac{M}{\mu}x_k |x_k|}{2(1-\frac{M}{\mu}|x_k|)}$} that attains \eqref{eq:bound_x} when {\small$\frac{M}{\mu}|x_k| <1$}.
     \qed \end{proof}
\end{theorem}
This is another example where the univariate case is sufficiently ``rich'' to provide the general worst-case behavior. Such surprisingly simple worst-case functions were already observed, for example, for the Gradient method applied to smooth (strongly) convex functions \cite{taylor2017smooth}.

\subsubsection*{Gradient method applied to Hessian Lipschitz functions}
The Gradient method has a local linear rate of convergence when applied to Hessian Lipschitz functions \cite{nesterov2018lectures}. Since  the tool is not restricted to second-order schemes, we can also examine the tightness of the following theorem:
\begin{theorem}[\cite{nesterov2018lectures}, Theorem 1.2.4 and Equation (1.2.26)] \label{th:Nesterov_GM}
If
\begin{itemize}
    \item $f$ has $M$-Lipschitz continuous Hessian,
    \item $\exists x_\star$ such that $\nabla f(x_\star) =0$, $\mu I \preceq \nabla^2 f(x_\star) \preceq LI$,
    \item $\|x_0 - x_\star\| < \frac{2 \mu}{M}$,
\end{itemize}
then all Gradient iterations with constant step size $x_{k+1} = x_k - h \nabla f(x_k)$ satisfy
{\small
\begin{equation}\label{eq:bound_GM}
    \|x_{k+1} - x_\star\| \leq \max \left\{ 1-h\Par{\mu-\frac{M}{2} \|x_k - x_\star\|}, h\Par{L+\frac{M}{2} \|x_k-x_\star\|} -1\right\} \|x_k - x_\star\|.
\end{equation}    
}
\end{theorem}

PEP results exactly match the worst-case behavior guaranteed by Theorem \ref{th:Nesterov_GM}, meaning that it is also unimprovable. Indeed, we can exhibit simple univariate functions attaining this bound for any number of iterations..
\begin{theorem}\label{th:local_linear_tight}
    Theorem \ref{th:Nesterov_GM} (i.e., \cite[Equation (1.2.26)]{nesterov2018lectures} is tight and attained by the two following functions:
    \begin{align}
        f(x) = -M\frac{|x|^3}{6} + \mu \frac{x^2}{2} \text{ and }g(x) = M\frac{|x|^3}{6} + L \frac{x^2}{2} 
    \end{align}
    \begin{proof}
        The iterations of the Gradient method on functions $f$ and $g$ are
        \begin{align*}
            x_{k+1}  & = x_k - h f'(x_k)  = x_k - h \Par{-\frac{M}{2}x_k|x_k| + \mu x_k}  = \Par{1 - h\Par{\mu - \frac{M}{2} |x_k|}}x_k, \\
            x_{k+1} & = x_k - h g'(x_k) =  x_k + h\Par{\frac{M}{2}x_k|x_k| + L x_k} = \Par{1 + h\Par{L + \frac{M}{2} |x_k|}}x_k.
        \end{align*}
     \qed \end{proof}
\end{theorem}

Figure \ref{fig:GM_2_regimes} shows the iterations of the Gradient method on functions $f$ and $g$. Interestingly, it exhibits the same kind of behavior as for smooth convex functions, i.e., short steps slowly converge whereas long steps overshoot \cite[Figure 3]{taylor2017smooth}.

\begin{figure}[ht!]
     \centering
     \begin{subfigure}[b]{0.49\textwidth}
         \centering
         \figureH{
         %\includesvg[width = \textwidth]{figures/GM_f}
         \includegraphics[width = \textwidth]{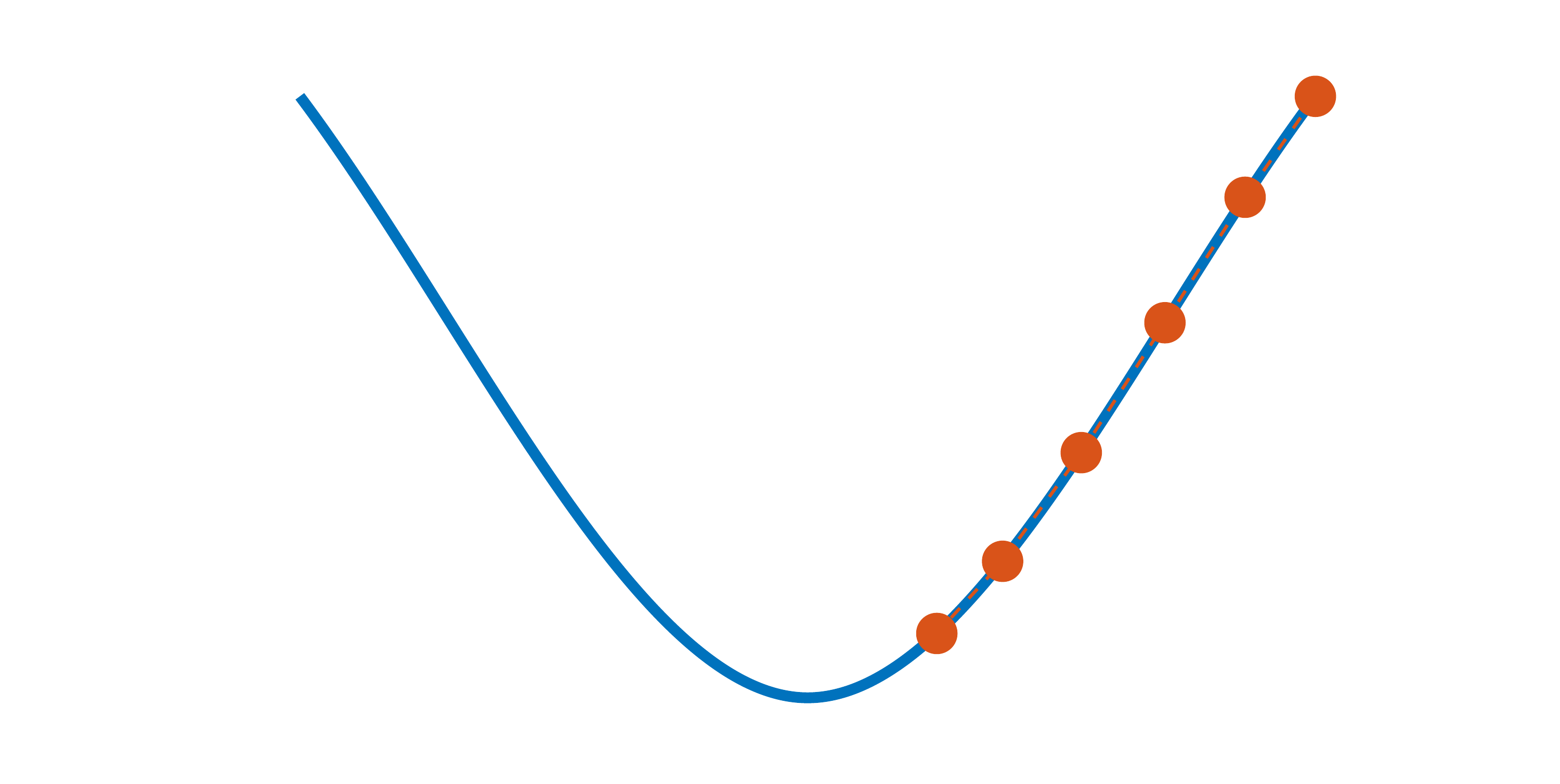}
         }
     \end{subfigure}
     \hfill
     \begin{subfigure}[b]{0.49\textwidth}
         \centering
         \figureH{
         %\includesvg[width = \textwidth]{figures/GM_g}
         \includegraphics[width = \textwidth]{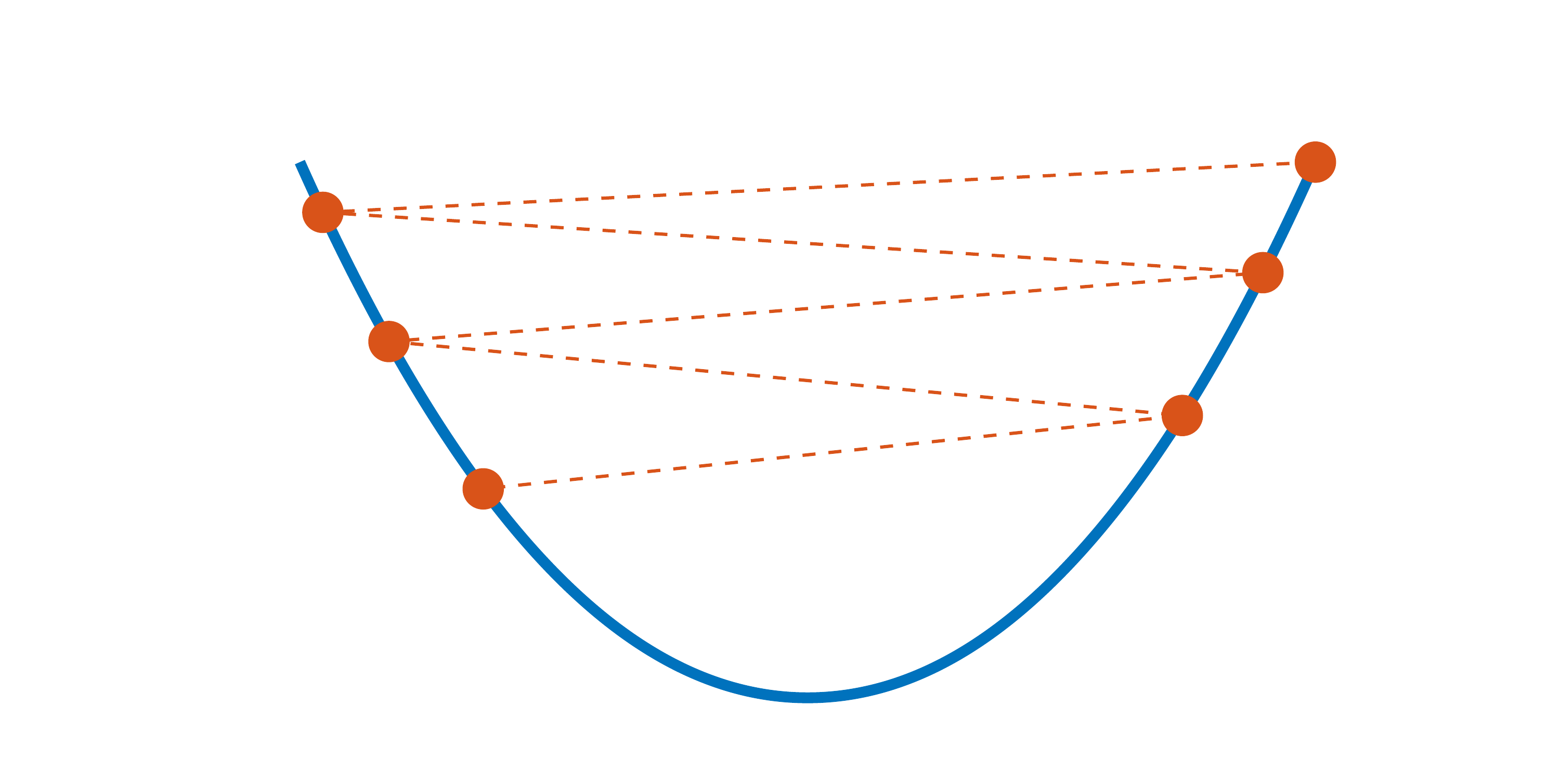}
         }
     \end{subfigure}
        \caption{$N=5$ iterations of gradient method on $f$ (left) and $g$ (right) with step sizes $h=\frac{2}{L+\mu}$ (left) and $h=\frac{2.1}{L+\mu}$ (right) and $|x_0 - x_\star| = 0.42$, $M=L=1$, $\mu = 0.3$.}
        \label{fig:GM_2_regimes}
\end{figure}

\subsubsection*{Fixed damped Newton method}
A simple variant of the Newton method is the fixed damped Newton method (DNM) adding a damping coefficient $\alpha \in (0,1)$ to the classical Newton step 
\begin{equation}
    x_{k+1} = x_k - \alpha \frac{f'(x_k)}{f''(x_k)}.
\end{equation}
This modification expands the region of convergence of the classical Newton's method although the adaptive damped Newton method is even more efficient \cite{hildebrand2021optimal,ivanova2024optimal,nesterov2018lectures}. We analyze the local behavior of DNM on the same setting as Theorem \ref{th:local_conv_Newton} (i.e., Hessian Lipschitz functions) for different initial distances and damping coefficients.

The worst-case performance of DNM on univariate Hessian Lipschitz functions returned by PEP allowed us to identify the analytical expression of functions attaining this worst-case. Depending on the parameters of the settings, we observed that the worst-case performance is exactly attained by one of the functions of the following family of functions parametrized by $a$

\begin{align}
    f(x) & = -M\frac{|x|^3}{6} + \mu \frac{x^2}{2} \\
    g_{a}(x) & = 
    \begin{cases}
        -M\frac{x^3}{6} + \mu\frac{x^2}{2} & \text{ if } x\leq \frac{\mu+a}{2M}, \\
        M\frac{x^3}{6} + a\frac{x^2}{2} + A_1x + B_1 & \text{ if } x\geq \frac{\mu+a}{2M},
    \end{cases} \label{eq:f3} \\
    h_{a}(x) & = \begin{cases}
        -M\frac{x^3}{6} + a \frac{x^2}{2} & \text{ if } x\leq \frac{a-\mu}{2M}, \\
       M\frac{x^3}{6} + \mu \frac{x^2}{2} + A_2x + B_2 & \text{ if } \frac{a-\mu}{2M} \leq x \leq 0, \\
       -M\frac{x^3}{6} + \mu \frac{x^2}{2} + A_3x + B_3 & \text{ if } x\geq 0 ,
    \end{cases} \label{eq:f4}
\end{align}
where $A_i$ and $B_i$ are (tediously but elementary to compute) constants such that the different parts of the function (and its gradient) connect correctly. Given the parameters of the setting, we must consider the worst functions of the families, denoted $g_\star$ and $h_\star$.

Figure \ref{fig:DNM_r} shows the evolution of the performance of DNM for varying initial distances.
\begin{figure}[ht!]
    \centering
    \figureH{
    %\includesvg[width=\textwidth]{figures/DNM_distV2}
    \includegraphics[width=\textwidth]{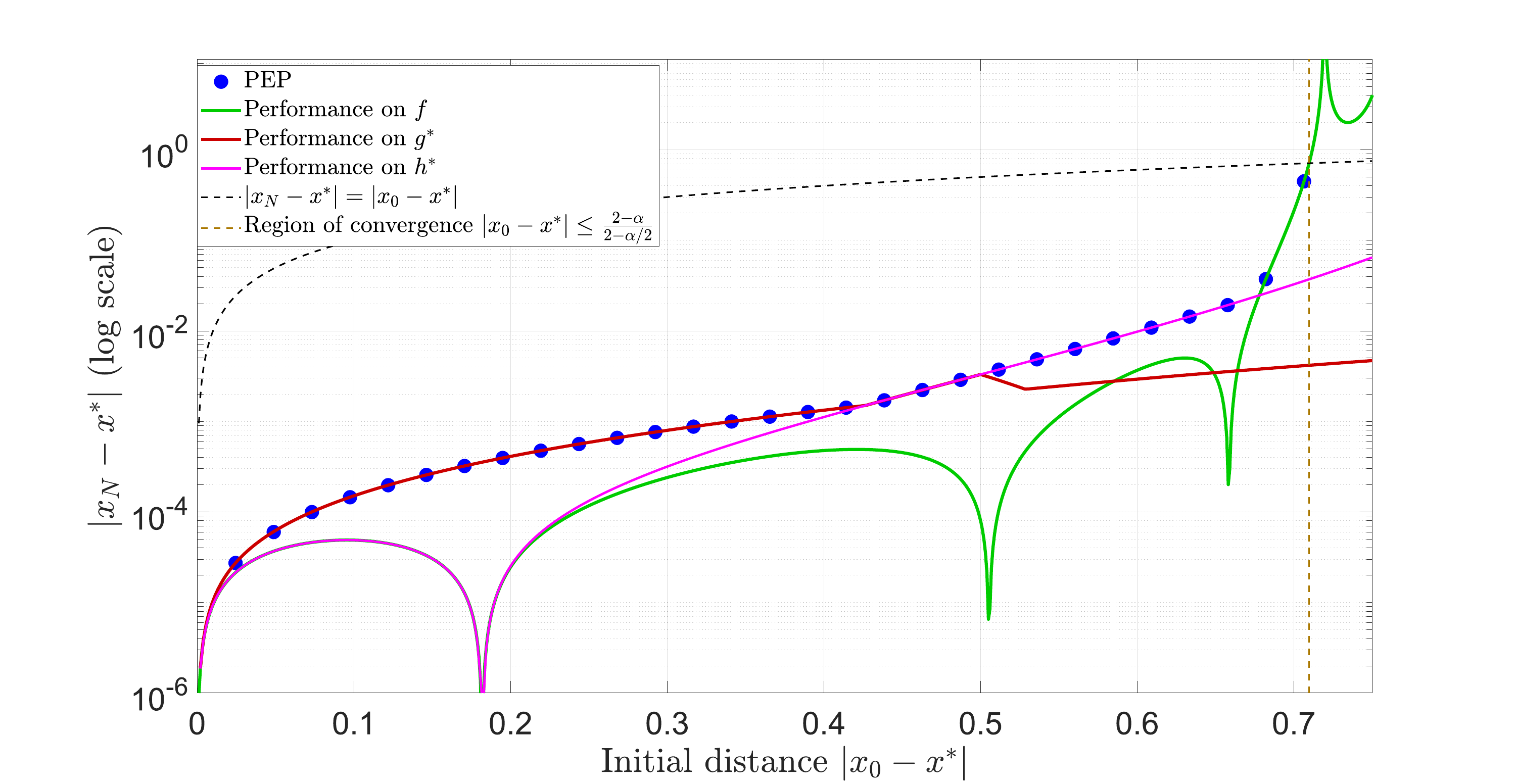}
    }
    \setlength{\abovecaptionskip}{-5pt}
    \setlength{\belowcaptionskip}{-5pt}
    \caption{Worst-case performance by PEP (blue dots) of $N=3$ iterations of damped Newton method for varying initial distance $|x_0 - x_\star|$ and $M = \mu = 1$, $\alpha = 0.9$. We also represented the performance of DNM on functions $f$ (green curve), $g_\star$ (red curve), and $h_\star$ (magenta curve).}
    \label{fig:DNM_r}
\end{figure}
We can also observe the impact of the step size on the worst-case performance in Figure \ref{fig:DNM_alpha}. It allows to select the optimal step size.
\begin{figure}[ht!]
    \centering
    \figureH{
    %\includesvg[width = \textwidth]{figures/DNM_alphaV2}
    \includegraphics[width = \textwidth]{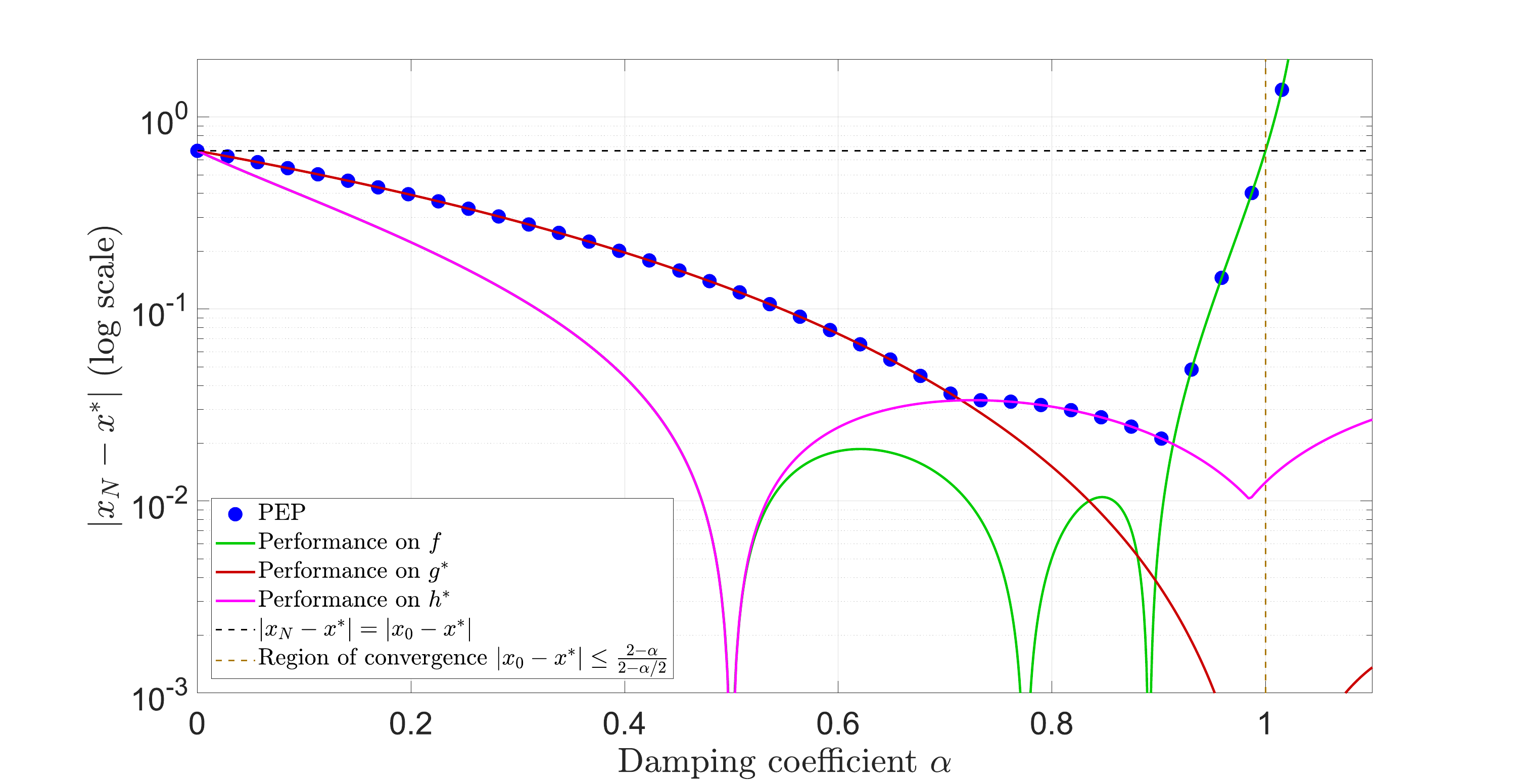}
    }
    \caption{Worst-case performance by PEP (blue dots) of $N=3$ iterations of damped Newton method with $M = \mu = 1$ and $\frac{M}{\mu}|x_0 - x_\star| = \frac23$ on univariate functions for varying damped coefficient. We also represented the performance of DNM on functions $f$ (green curve), $g_\star$ (red curve), and $h_\star$ (magenta curve).}
    \label{fig:DNM_alpha}
\end{figure}

\subsection{Self-concordant functions}
We now analyze the convergence of Newton and damped Newton methods on self-concordant functions using the interpolation conditions of Theorem \ref{th:IC_sc}. We consider standard-self-concordant functions, i.e., $M=1$. We consider the evolution of the \textit{Newton decrement} defined as the size of the Newton step in the local norm, namely, in the univariate case
\begin{equation}
    \lambda_f(x_k) = \frac{|f'(x_k)|}{\sqrt{f''(x_k)}}
\end{equation}
We only need interpolation without function value since neither the methods nor the performance criterion use function values.
\subsubsection{Newton's method}\label{sect:sc_nm}
The following Theorem \cite[Equation (11)]{hildebrand2021optimal} provides the tight worst-case performance of a single iteration of Newton method on univariate functions.
\begin{theorem}[\cite{hildebrand2021optimal}, Equation (11)] \label{th:sc_nm}
    Given the initial Newton decrement $\lambda_f(x_k) \leq 1$. The Newton decrement of a single iteration of Newton method, $\lambda_f(x_{k+1})$, on standard-self-concordant functions satisfies
    \begin{equation}\label{eq:sc_nm}
        \lambda_f(x_{k+1}) \leq 4 - \lambda_f(x_k)^2 - 4 \sqrt{1-\lambda_f(x_k)^2}.
    \end{equation}
    Moreover, there exists a standard-self-concordant function attaining exactly this bound.
\end{theorem}
This theorem is tight and cannot be improved. However, to illustrate the scope of the tool, we (i) propose an alternative proof to the theorem that consists of explicitly solving the associated PEP, (ii) provide the analytical expression of a standard-self-concordant function attaining the bound of the theorem, (iii) extend numerically the theorem to multiple steps.
\paragraph{Alternative proof of Theorem \ref{th:sc_nm}}
We propose to formulate the PEP associated to Theorem \ref{th:sc_nm} and to solve it analytically. The PEP can be written as
\begin{equation}\label{eq:PEP_sc}
    \begin{aligned}
      \underset{S=\{(x_0,g_0,h_0),(x_1,g_1,h_1)\}}{\max \ } & \frac{|g_1|}{\sqrt{h_1}}\\
     \text{ s.t. } &  x_{1}=x_0-\frac{g_0}{h_0}, \\
     &\frac{|g_0|}{\sqrt{h_0}} = R, \\
     & S \text{ is standard-self-concordant-interpolable.}
\end{aligned}
\end{equation}
For simplicity, we consider variable $\tilde h_i = \frac{1}{\sqrt{h_i}}$ instead of $h_i$. We assume $x_0 = 0$ and $x_1 = 1$ w.l.o.g.\@ since a translation can move $x_0$ to zero and a linear change of variable $f(x) \rightarrow f(c x)$ can move $x_1$ to one. Therefore, the two equality constraints of \eqref{eq:PEP_sc} lead to $g_0 = -R^2$ and $\tilde h_0 = \frac{1}{R}$. The interpolation conditions are (see Theorem \ref{th:IC_sc})
\begin{align}
        & | \tilde h_1 - \tilde h_2 | \leq 1 \\ 
        & g_1 - g_0 \geq \frac{1}{\tilde h_0} + \frac{1}{\tilde h_1} - \frac{4}{\tilde h_0 + \tilde h_1 +1}\\
        & \text{If } \tilde{h}_0 + \tilde h_1 > 1 \text{ then } g_0 - g_1 \geq \frac{1}{\tilde h_0} + \frac{1}{\tilde h_1} - \frac{4}{\tilde h_0 + \tilde h_1 -1}.
\end{align}
We considered that $h_1 \neq 0$, otherwise \eqref{eq:PEP_sc} would be unbounded. The problem can now be written as
\begin{equation}\label{eq:PEP_sc2}
    \begin{aligned}
      \underset{g_1 \tilde h_1}{\max \ } & |g_1|\tilde h_1\\
     \text{ s.t. } &  | \tilde h_1 - \tilde h_0 | \leq 1, \\
     & g_1 - g_0 \geq \frac{1}{\tilde h_0} + \frac{1}{\tilde h_1} - \frac{4}{\tilde h_0 + \tilde h_1 +1}, \\
     & \text{If } \tilde{h}_0 + \tilde h_1 > 1 \text{ then } g_0 - g_1 \geq \frac{1}{\tilde h_0} + \frac{1}{\tilde h_1} - \frac{4}{\tilde h_0 + \tilde h_1 -1}
\end{aligned}
\end{equation}
with $g_0 = -R^2$ and $\tilde h_0 = \frac 1R$. One can show that the optimal solution will satisfy
\begin{equation}
     g_1 =  g_0 - \frac{1}{\tilde h_0} - \frac{1}{\tilde h_1} + \frac{4}{\tilde h_0 + \tilde h_1 -1} \geq 0.
\end{equation}
Therefore, we maximize the performance $p(\tilde h_1)$
\begin{equation}
    p(\tilde h_1) = g_1 \tilde h_1 = g_0 \tilde h_1 - \frac{\tilde h_1}{\tilde h_0} - 1 + \frac{4\tilde h_1}{\tilde h_0 + \tilde h_1 - 1} = -R(R+1) \tilde h_1 - 1 \frac{4\tilde h_1 R}{1+ \tilde h_1 R - R}.
\end{equation}
The first-order optimality condition $p'(\tilde h_1)=0$ yields
\begin{equation}
    \tilde h_1 R = R-1 + 2 \sqrt{\frac{1-R}{1+R}}
\end{equation}
and therefore the optimal value of \eqref{eq:PEP_sc} is
\begin{equation}
    p(\tilde h_1) = g_1 \tilde h_1 = 4-R^2 - 4\sqrt{1-R^2}
\end{equation}
with $R$ and $p(\tilde h_1)$ the Newton decrement before and after the iteration, which is the bound of Theorem \ref{th:sc_nm}.

\paragraph{Function attaining the bound for a single iteration}
Relying on the tool, we found a function attaining the bound \eqref{th:sc_nm} for a single iteration. One can check that Newton iteration from the point $x_k = 0$ on the following functions yields the result
\begin{equation}
    f(x) = 
    \begin{cases}
        Bx - \log(x-A) & \text{if } x\leq \frac12 (A-R) \\
        \Par{B+\frac{4}{R+A}}x - \log \Par{x+R} & \text{if } x > \frac12 (A-R),
    \end{cases}
\end{equation}
where $R = \lambda_f(x_k)$, $A = R\frac{-R^2 + (-4+2\sqrt{1-R^2})R + 5 - \sqrt{1-R^2}}{R^2-1+2\sqrt{1-R^2}}$, and $B=\frac{R^2-1+2\sqrt{1-R^2}}{R(R-1)}$.

\paragraph{Worst-case performance for multiple iterations}
As in the previous case of Theorem \ref{lem:new_descent_lemma} for \eqref{eq:CNM}, even if Theorem \ref{th:sc_nm} is tight and unimprovable for a single iteration, the worst-case performance could be better for multiple iterations. We can analyze the worst-case behavior of multiple iterations of Newton's method on self-concordant functions with the tool. Figure \ref{fig:nm_sc} shows the worst-case Newton decrement $\lambda_f(x_N)$ of $N=2$ iterations of Newton method on standard-self-concordant functions for varying initial Newton decrement $\lambda_f(x_0)$. We compare PEP results (blue dots), \ref{th:sc_nm} from \cite[Equation (11)]{hildebrand2021optimal} (solid red line), and \cite[Theorem 5.2.2.1]{nesterov2018lectures} (broken red line).
\begin{figure}[ht!]
    \centering
    \figureH{
    %\includesvg[width=\textwidth]{figures/NM_on_SC_N2V3}
    \includegraphics[width=\textwidth]{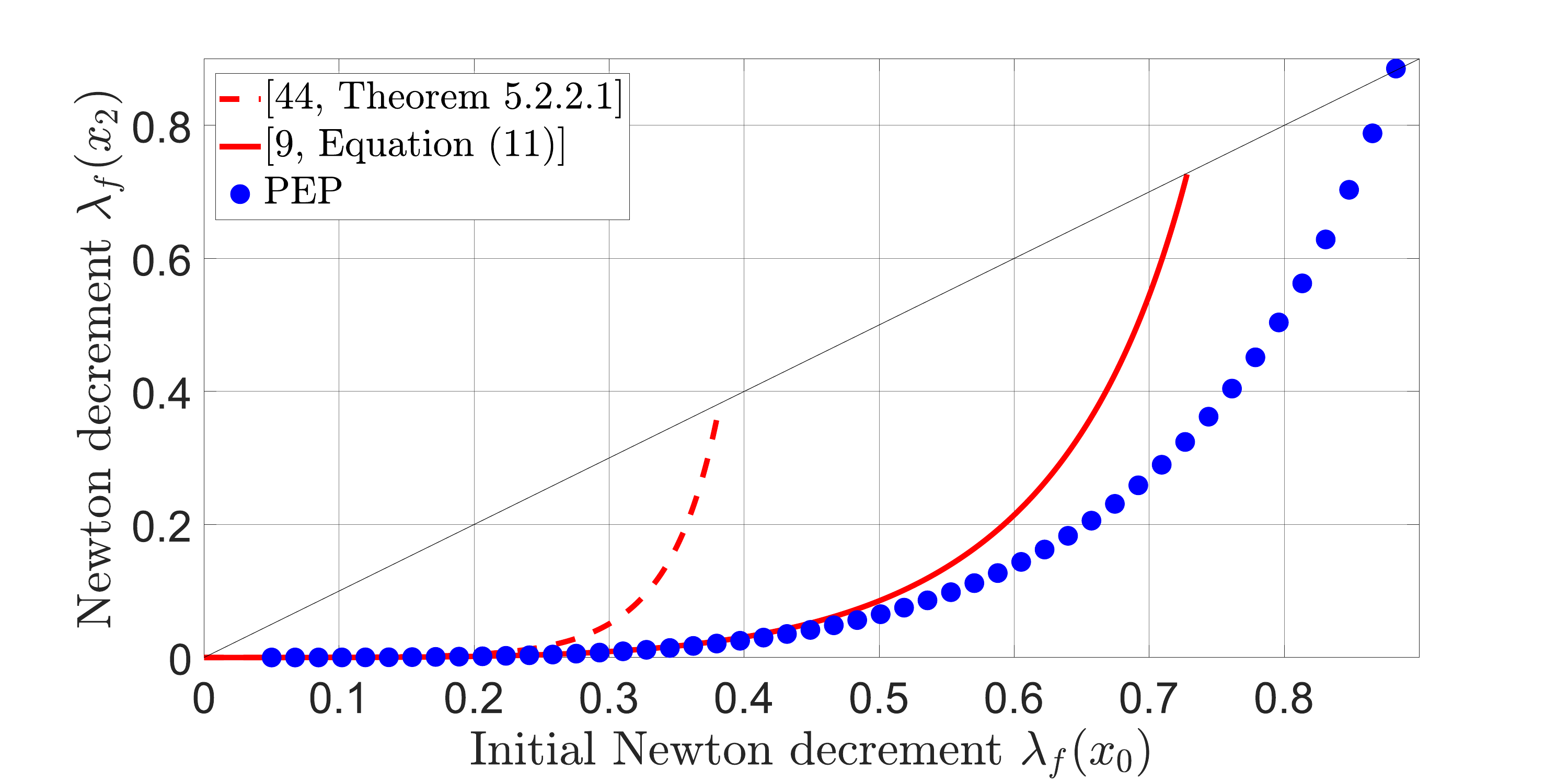}
    }
    \caption{Worst-case performance $\lambda(x_{2})$ for varying initial $\lambda(x_0)$ of two iterations of Newton method (black dots) compared to bounds \cite[Theorem 5.2.2.1]{nesterov2018lectures} (dashed lines), \cite[Equation (11)]{hildebrand2021optimal} (solid lines), and the PEP results (dots).}
    \label{fig:nm_sc}
\end{figure}
We observe that the performance is actually better than predicted by previous theoretical bounds and also that the region of convergence is larger.

\subsubsection{Damped Newton method}
We now consider the fixed damped Newton method
\begin{equation}\label{eq:DNM}
    x_{k+1} = x_k - \gamma_k \frac{f'(x_k)}{f''(x_k)}. \tag{DNM}
\end{equation}
The following Theorem \cite[Equation (11)]{hildebrand2021optimal} provides the tight worst-case performance of a single iteration of damped Newton method on univariate functions for some step sizes $\gamma_k$.
\begin{theorem}[\cite{hildebrand2021optimal}, Equation (11)]\label{th:dnm_sc}
    A single iteration of damped Newton method \eqref{eq:DNM} with $\gamma_k \leq 2 \frac{\sqrt{1+\lambda_f(x_k)^3}-1}{\lambda_f(x_k)^3}$ on standard-self-concordant functions satisfies
    \begin{equation}\label{eq:dnm_sc}
        \lambda_f(x_{k+1}) \leq \lambda_f(x_k) - \gamma \lambda_f(x_k) + \gamma \lambda_f(x_k)^2.
    \end{equation}
    Moreover, there exists a standard-self-concordant function attaining exactly this bound.
\end{theorem}
Note that \cite[Equation (11)]{hildebrand2021optimal} covered all values of $\gamma_k$ and that the case $\gamma_k = \frac{1}{1+\lambda_f(x_k)}$ is always covered by Theorem \ref{th:dnm_sc}. This theorem is tight for a single iteration and, relying on our tool, we can exhibit a function attaining its bound \eqref{eq:dnm_sc}. Indeed, one can check that a single iteration of \eqref{eq:DNM} (with any $\gamma_k$) from the point $x_k = 0$ on the following function yields bound \eqref{eq:dnm_sc}
\begin{equation}\label{eq:worst_fun_dnm_sc}
    f(x) = \frac{R-1}{R}x - \log(R-x)
\end{equation}
where $R = \lambda_f(x_k)$.

Moreover, we observed with PEP that bound \eqref{eq:dnm_sc} is actually tight even for multiple iterations (unlike Theorem \ref{th:sc_nm}). Therefore, we propose the following theorem on the rate of convergence of \eqref{eq:DNM} on univariate standard-self-concordant functions.
\begin{theorem}\label{th:dnm_sc_our}
    After $N$ iterations of \eqref{eq:DNM} with $\gamma_k \leq 2 \frac{\sqrt{1+\lambda_f(x_k)^3}-1}{\lambda_f(x_k)^3}$ for $k=0,\ldots,N-1$ on standard-self-concordant functions, we have
    \begin{equation}
        \lambda_f(x_{k+1}) \leq \lambda_f(x_k) - \gamma \lambda_f(x_k) + \gamma \lambda_f(x_k)^2, \quad \forall k=0,\ldots,N-1.
    \end{equation}
    Moreover, there exists a standard-self-concordant function attaining exactly this bound.
    \begin{proof}
        The bound follows directly from Theorem \ref{th:sc_nm} and the attaining function is \eqref{eq:worst_fun_dnm_sc}.
     \qed \end{proof}
\end{theorem}
 
\subsection{Quasi-self-concordant functions}\label{sect:NM_qsc}
The interpolation conditions for the class of quasi-self-concordant functions (Theorem \ref{th:IC_qsc}) involve the exponential function, which makes the associated non-convex PEP challenging to solve. However, Gurobi 11 \cite{gurobi} allows dealing with such non-convex constraints.

\subsubsection{Gradient Regularized Newton method}
The Gradient Regularized Newton method 1 (GNM1) \cite{doikov2023minimizing} is a variant of the Cubic Newton method that achieves a linear rate of convergence on quasi-self-concordant functions \cite[Theorem 3.3]{doikov2023minimizing} with the following explicit iteration
\begin{equation}\label{eq:GNM1}
    x_{k+1} = x_k - \frac{f'(x_k)}{f''(x_k) + M |f'(x_k)|}. \tag{GNM1}
\end{equation}
\ref{eq:GNM1} also exhibits a quadratic local convergence on quasi-self-concordant functions \cite[Theorem 4.1]{doikov2023minimizing}. The proof of this quadratic convergence relies on the following descent lemma (proved in the multivariate case).
\begin{lemma}[\cite{doikov2023minimizing}, Equation (49)]
    The iterations of \eqref{eq:GNM1} on univariate $M$-quasi-self-concordant functions satisfy
    \begin{equation}
        \eta(x_{k+1}) \leq e^{\eta(x_k)}( e^{\eta(x_k)} + \eta(x_k)^2 -\eta(x_k) -1)
    \end{equation}
    where $\eta(x) = M\frac{|f'(x)|}{f''(x)}$.
\end{lemma}
Relying on the tool, we observed numerically that this lemma can be improved in the univariate case. And, we can provably improve it by solving analytically the associated PEP.
\begin{lemma}\label{lem:descent_qsc}
    The iterations of \eqref{eq:GNM1} on univariate $M$-quasi-self-concordant functions satisfy
    \begin{equation}
        \eta(x_{k+1}) \leq e^{\frac{\eta(x_k)}{\eta(x_k)+1}}(\eta(x_k)-1)+1
    \end{equation}
    where $\eta(x) = M\frac{|f'(x)|}{f''(x)}$.
    \begin{proof}
        We propose to formulate the associated PEP and to solve it analytically. The PEP can be written as
        \begin{equation}\label{eq:PEP_qsc}
            \begin{aligned}
                \underset{S=\{(x_0,g_0,h_0),(x_1,g_1,h_1)\}}{\max \ } & \frac{|g_1|}{h_1}\\
                \text{ s.t. } &  x_{1}=x_0-\frac{g_0}{h_0+M|g_0|}, \\
                 &\frac{|g_0|}{h_0} = R, \\
                & S \text{ is $M$-quasi-self-concordant-interpolable.}
            \end{aligned}
        \end{equation}
        We assume $x_0 = 0$ and $g_0 = 1$ w.l.o.g.\@ since a translation can move $x_0$ to zero and a scaling $f(x) \rightarrow c f( x)$ can move $g_0$ to one. Therefore, the two equality constraints of \eqref{eq:PEP_qsc} lead to $h_0 = \frac 1R$ and $x_1 = \frac{-R}{MR+1}$. The interpolation conditions are (see Theorem \ref{th:IC_qsc})
        \begin{align}
                & g_1 - g_0 \geq \frac{h_0 + h_1}{M} - \frac{2}{M}  \sqrt{h_0 h_1}e^{-\frac M2 (x_1 - x_0)} \\
                & g_0 - g_1 \geq \frac{h_1 + h_0}{M} - \frac{2}{M}  \sqrt{h_1 h_0}e^{-\frac M2 (x_0 - x_1)}.
        \end{align}
        Figure \ref{fig:qsc} is an illustration of the feasible domain for $(g_1,h_1)$. One can check that the optimal solution is
        \begin{align}
            h_1^* & = h_0 e^{-\frac{\eta_0}{\eta_0+1}} \\
            g_1^*& = g_0 + \frac{-h_0+h_1^*}{M}
        \end{align}
        yielding the optimal value
        \begin{equation}
            M\frac{|g_1^*|}{h_1^*} = e^{\frac{\eta_0}{\eta_0+1}}(\eta_0-1)+1
        \end{equation}
        where $\eta_0 = M R = M \frac{|g_0|}{h_0}$.
     \qed \end{proof}
\end{lemma}
Figure \ref{fig:LRN1} compares the numerical and analytical solution of PEP (blue dots and curve) with \cite[Equation (49)]{doikov2023minimizing} (red curve) on the worst-case performance of one iteration of \eqref{eq:GNM1} on $1$-quasi-self-concordant functions.

\begin{figure}[ht!]
    \centering
    \figureH{
    %\includesvg[width=\textwidth]{figures/LRN_QSCV1.svg}
    \includegraphics[width=\textwidth]{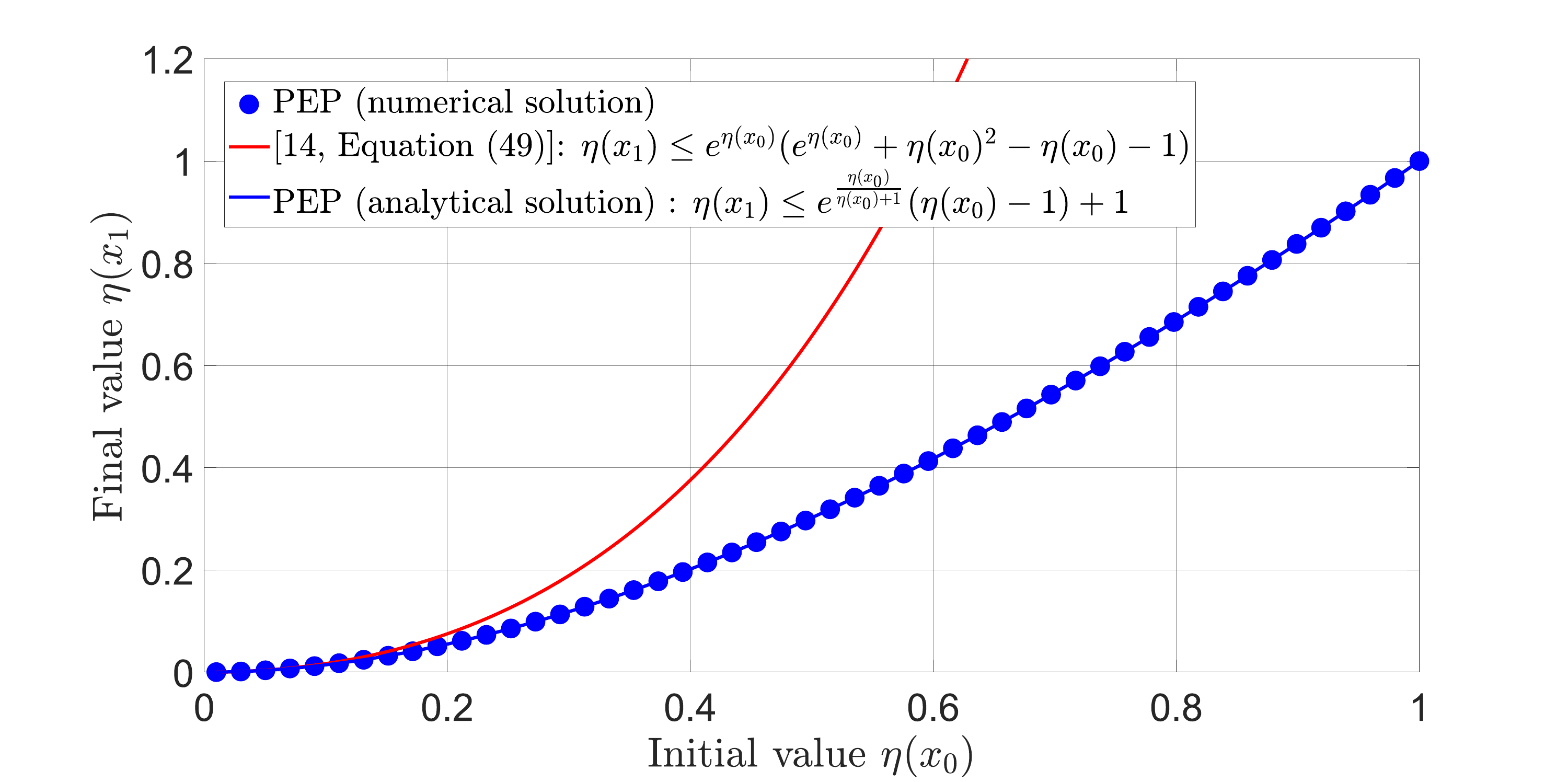}
    }
    \caption{Worst-case performance $\eta(x_{1})=M\frac{|f'(x_{1})|}{f''(x_{k})}$ of one iteration of \eqref{eq:GNM1} on $1$-quasi-self-concordant functions for different values of $\eta(x_{0})=M\frac{|f'(x_{0})|}{f''(x_{0})}$. We compare the numerical solutions of PEP (blue dots), the known descent lemma of \cite[Equation (49)]{doikov2023minimizing} (red curve), and the analytical solution of PEP (blue curve).}
    \label{fig:LRN1}
\end{figure}

We can also observe the effect of several iterations of \eqref{eq:GNM1} on Figure \ref{fig:LRN2}.
\begin{figure}[ht!]
    \centering
    \figureH{
    %\includesvg[width=\textwidth]{figures/LRN_QSC_NV1.svg}
    \includegraphics[width=\textwidth]{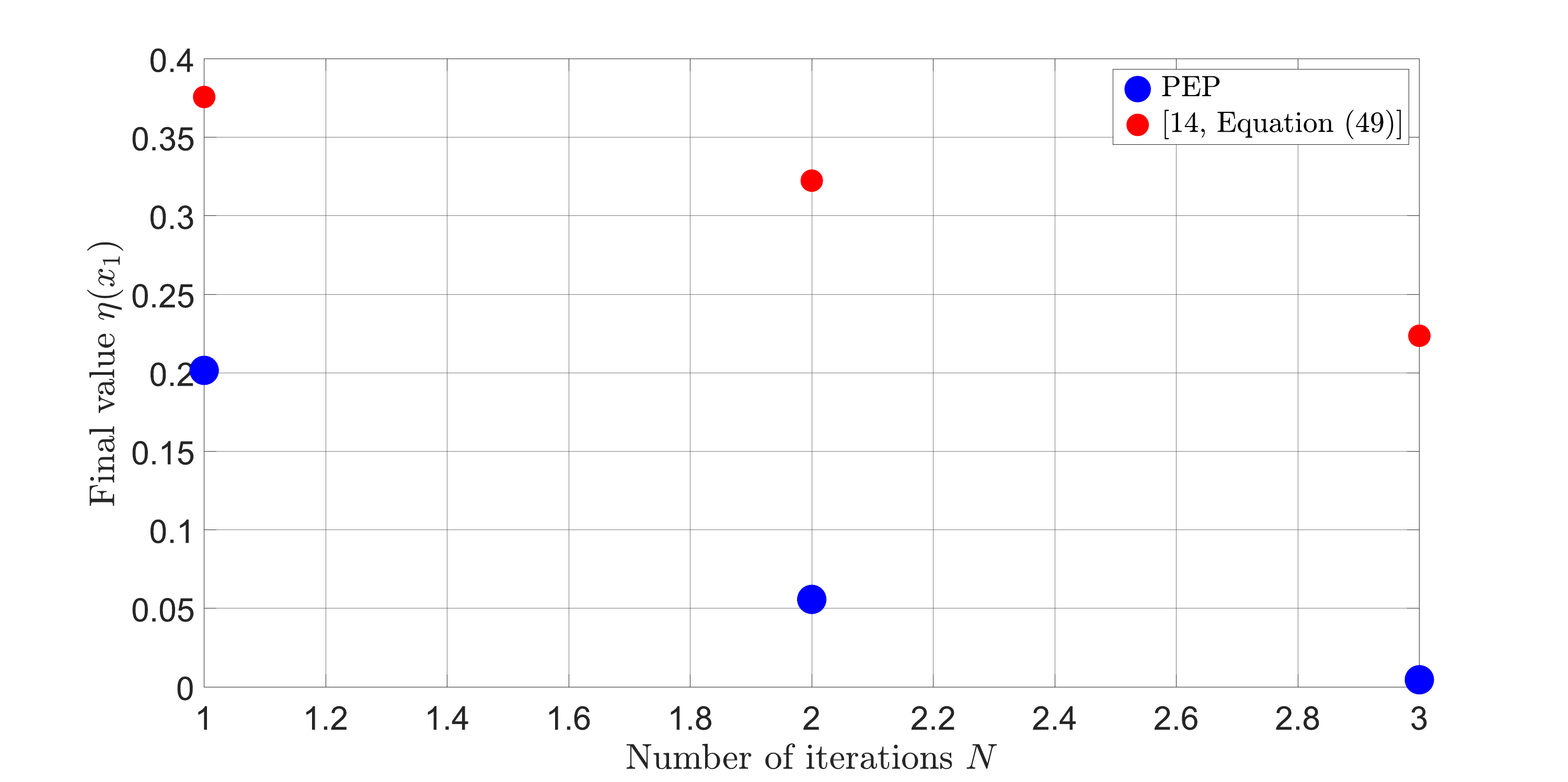}
    }
    \caption{Worst-case performance of \ref{eq:GNM1} for varying number of iterations $N$ on $1$-quasi-self-concordant functions with $\eta(x_0) = M \frac{|f'(x_0)|}{f''(x_0)} \leq 0.4$.}
    \label{fig:LRN2}
\end{figure}
\subsubsection{Newton's method}
Newton method also exhibits a local quadratic rate of convergence on quasi-self-concordant and strongly convex functions \cite[Theorem 4.1]{doikov2023minimizing}. The proof of the quadratic convergence relies on the following lemma:
\begin{lemma}[\cite{doikov2023minimizing}, Equation (53)]\label{lem:lemma_qsc}
    A single iteration of Newton method $x_+ = x - \frac{f'(x)}{f''(x)}$ on $M$-quasi-self-concordant $\mu$-strongly convex univariate functions satisfy
    \begin{equation}\label{eq:lemma_qsc}
        \frac{M}{\mu} |f'(x_+)| \leq e^{\frac{M}{\mu} |f'(x)|} - \frac{M}{\mu} |f'(x)| -1.
    \end{equation}
\end{lemma}
Relying on the tool, we can show that this lemma is tight for a single iteration by exhibiting a starting point and a function reaching \eqref{eq:lemma_qsc}.
\begin{lemma}\label{lem:lemma_qsc_improved}
    Lemma \ref{lem:lemma_qsc} (i.e., \cite[Equation (53)]{doikov2023minimizing}) is tight and attained by the starting point $x=0$ and the function
\begin{equation}
    f(x) = 
    \begin{cases}
        \frac{\mu}{M^2} e^{-Mx} +  \frac{2\mu}{M} x & \text{ if } x\leq 0, \\
        \frac{\mu}{M^2} e^{Mx} & \text{ if } x> 0.
    \end{cases}
\end{equation}
\begin{proof}
We have $f'(0) = \frac{\mu}{M}$, $f''(0)=\mu$, $x_+ = -\frac{1}{M}$, $f'(x_+)=\frac{\mu}{M}(-e+2)$, and
\begin{equation}
    \frac{M}{\mu}|f'(x_+)| = e^{\frac{M}{\mu} |f'(0)|} - \frac{M}{\mu} |f'(0)| -1 = e-2.
\end{equation}
 \qed \end{proof}
\end{lemma}

\subsection{Comparing methods together}
As can be seen in Table \ref{tab:tab1} and in the rest of the literature, there is no comparison of the different second-order methods on the same settings, i.e., same function class, performance measure, and initial condition. The tool allows to perform such comparisons on any setting of choice. Figure \ref{fig:all_methods} compares the worst-case performance $|f'(x_{k+1})|$ of $N=1$ iteration of the following methods on $\mu$-strongly convex Hessian $M$-Lipschitz functions for varying $\mu$ and initial gradient $|f'(x_k)| = 1$ and $M=1$:
\begin{itemize}
    \item Newton method: $x_{k+1} = x_k - \frac{f'(x_k)}{f''(x_k)}$;
    
    \item Cubic Regularized Newton method:
    \begin{equation*}
        x_{k+1} = \mathrm{Arg} \min_x f(x_k) + f'(x_k) (x-x_k) + \frac12 f''(x_k) (x-x_k)^2 + \frac{M}{6} |x-x_k|^3;
    \end{equation*}
    \item Gradient Regularized Newton method 1: $x_{k+1} = x_k  - \frac{f'(x_k)}{f''(x_k)+M |f'(x_k)|}$;
    \item Gradient Regularized Newton method 2: $x_{k+1} = x_k  - \frac{f'(x_k)}{f''(x_k)+\sqrt{\frac{M}{2} |f'(x_k)|}}$;
    
    \item Adaptive damped Newton method: $ x_{k+1} = x_k - \frac{1}{1+M\sqrt{\frac{f'(x_k)^2}{f''(x_k)}}}\frac{f'(x_k)}{f''(x_k)}$.
\end{itemize}

\begin{figure}[ht!]
    \centering
    \figureH{
    \includegraphics[width = \textwidth]{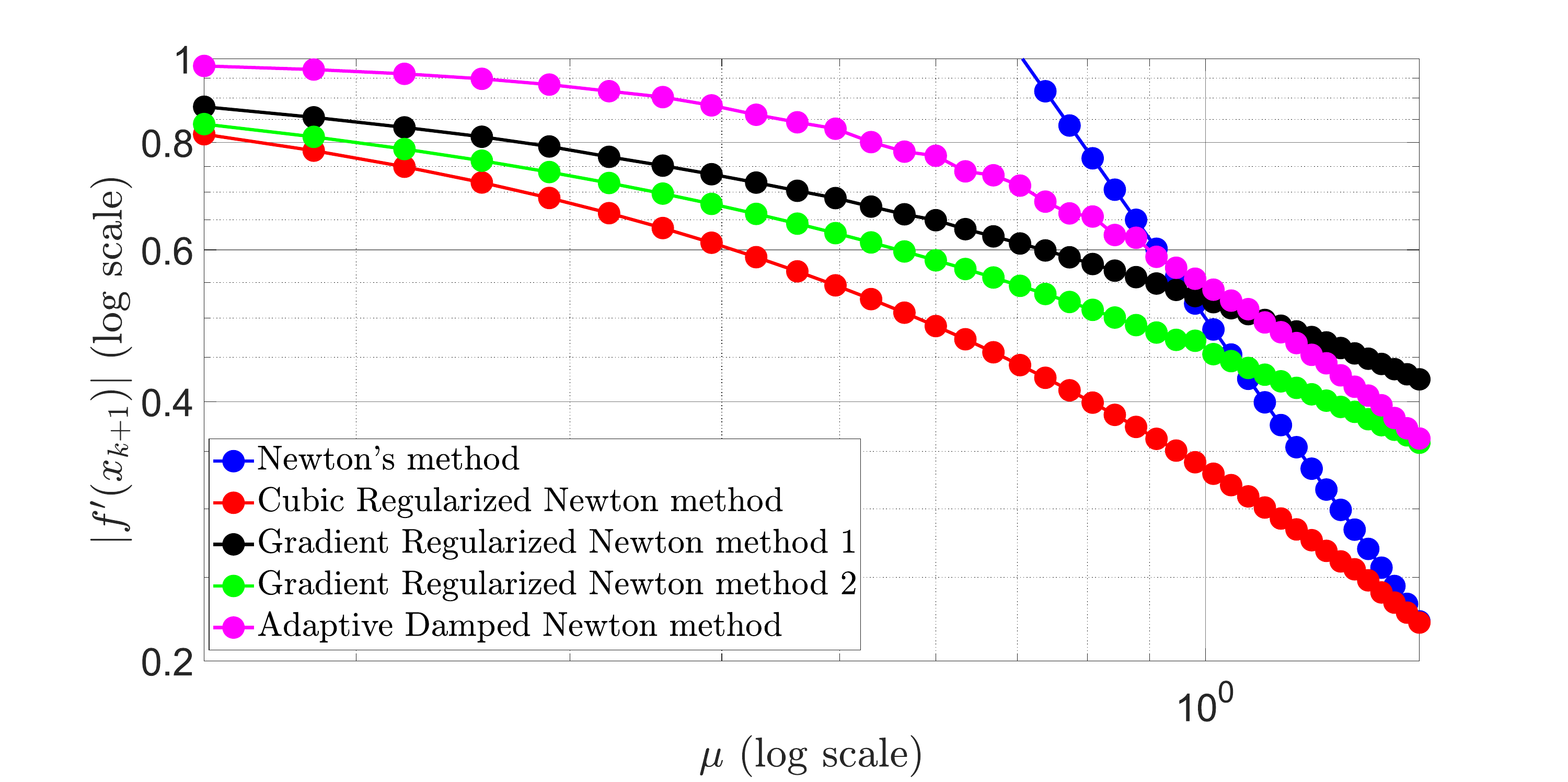}
    }
     \setlength{\abovecaptionskip}{-5pt}
    \setlength{\belowcaptionskip}{-5pt}
    \caption{Worst-case performance $|f'(x_{k+1})|$ of different second-order methods on $\mu$-strongly convex Hessian $M$-Lipschitz functions for varying $\mu$ and initial gradient $|f'(x_k)| = 1$ and $M=1$.}
    \label{fig:all_methods}
\end{figure}
We observe that CNM, GNM1, GNM2, and ADNM share similar behavior with CNM being faster. We also see the local quadratic convergence of NM when $\mu$ is sufficiently large. The figure exhibits the contraction factor $\frac{|f'(x_{k+1})|}{|f'(x_{k})|}$ after a single iteration from which we can deduce an upper bound on the decrease of the gradient after $N$ iterations.
\section{Conclusion and future perspectives}
We took a step forward into tightly analyzing the performance of second-order methods with respect to several performance measures and on several function classes. Indeed, we were able to obtain exact performance guarantees in the \emph{univariate} setting, by (i) providing exact discrete representations of many univariate function classes of interest, and (ii) proposing a tractable second-order PEP formulation by leveraging advances in solving non-convex PEPs.

This allowed us to (i) get an insight into interpolation conditions for second-order function classes (i.e., classes of interest for second-order optimization methods), which are already non-trivial when restricted to the univariate case, (ii)~prove tightness of existing multivariate bounds by exhibiting worst-case instances, (iii)~improve existing convergence rates, either analytically or numerically, and (iv)~compare methods on a fair basis, i.e., using common setup and performance criteria, as opposed to existing results from the literature, which vary in their assumptions.

We see several directions for future research. First, the analysis of methods on univariate functions could be further developed, and lead to a better understanding of bottlenecks in existing methods or to the design of optimal ones. In particular, leveraging on the method to obtain interpolation conditions for various function classes and on solving non-convex PEPs, one could tightly analyze other function classes, e.g., functions with higher-order properties, and other methods, e.g., zeroth-order, quasi-Newton, or adaptive methods. A second line of research, more challenging, would be to extend our results to the multivariate case. This would require (i) obtaining multivariate interpolation conditions, for instance by building on the interpolation conditions we derived, which naturally furnish candidate multivariate conditions, and (ii) finding a way to efficiently solve PEP formulations involving matrix variables for the Hessians.

\begin{acknowledgements}
The authors would like to thank Adrien Taylor, Moslem Zamani, Nikita Doikov, and Anton Rodomanov for fruitful discussions. A. Rubbens and N. Bousselmi are, respectively, Research Fellow and FRIA grantees of the Fonds de la Recherche Scientifique - FNRS. 
\end{acknowledgements}

\bibliographystyle{spmpsci}      % mathematics and physical sciences
\bibliography{bibliography}   % name your BibTeX data base

\appendix
\section{Proof of Lemma \ref{lem:interval}} \label{app:interval}
\begin{proof}[Proof of Lemma \ref{lem:interval}]
    \textit{(Necessity)}
    Suppose $S=\{(x_i,f_i^0,f_i^1,...,f_i^m)\}_{i\in[N]}$ is $\F$-interpolable, and let $f\in \F$ be a function interpolating $S$, and in particular all pairs $(i,j)$. For each pair $(i,j)$, the restriction of $f$ to the interval $[x_i,x_j]$ is a function in $\F$ interpolating $\left\{ (x_i,f_i^0,f_i^1,...,f_i^m),(x_j,f_j^0,f_j^1,...,f_j^m)\right\}$, and \eqref{eq:goal} is satisfied.
\\

    \textit{(Sufficiency)} Suppose \eqref{eq:goal} is satisfied, let $x_1 \leq \ldots \leq x_N$ and let $f^{(i)} \in \F$ be a function interpolating the pair $(i,i+1)~~\forall i\in [N-1]$.
    Define
     \begin{align}
        \mathcal{I}_i = \begin{cases}
             ]-\infty,x_{2}] &~~i=1\\
            [x_i,x_{i+1}] &~~i=2,\cdots,N-2\\
            [x_{N-1},\infty[&~~i=N-1,
        \end{cases}
    \end{align}
    where $\mathcal{I}_1$ and $\mathcal{I}_{N-1}$ are defined such that $\bigcup\limits_{i=1}^{N-1} \mathcal{I}_{i} = \mathbb{R}$ and not just $[x_1,x_N]$. Let $f:\R\to\R$ be defined by:
    \begin{align}
        f(x) := f^{(i)}(x) \quad \forall x\in\mathcal{I}_i, \ \forall i\in[N-1].
    \end{align}
    By construction, $f$ interpolates $S=\{(x_i,f_i^0,f_i^1,...,f_i^m)\}_{i\in[N]}$. Furthermore, by order $m$ piecewise invariance of $\F$, it holds that $f\in \F$, which completes the proof.
 \qed \end{proof}
\section{Proof of Proposition \ref{prop:extr_conn}}\label{app:properties}
\begin{proof}
The proof is illustrated on Figure \ref{fig:proof_completable}.
\begin{figure}[ht!]
    \centering
    \includegraphics[width=\linewidth]{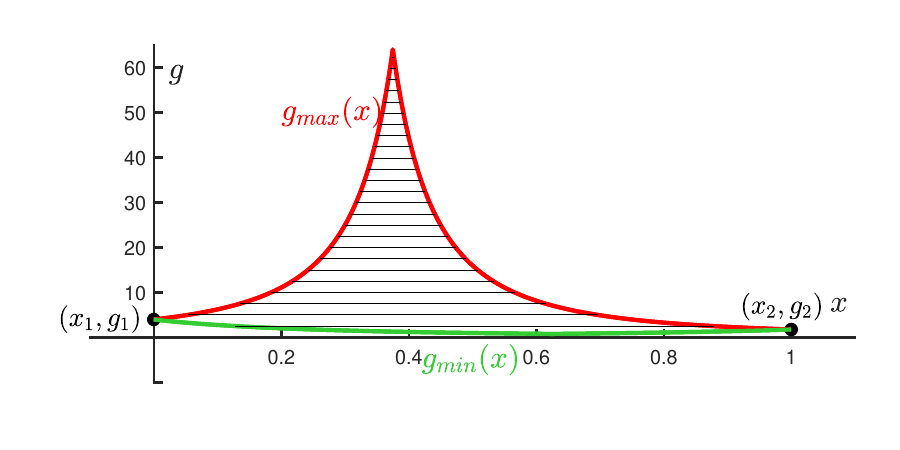}
    \caption{Illustration of the proof of Proposition \ref{prop:extr_conn}. Given $S=\{(x_i,f_i,g_i)\}_{i=1,2}$, $\tilde S=\{(x_i,g_i)\}_{i=1,2}$, and its extremal interpolants $g_{\min}$ and $g_{\max}$, we construct a function $g\in \F_{M,\alpha,+}$, given in \eqref{eq:g(x,a)}; interpolating $\tilde S$, and whose integral $f_1+\int_{x_1}^{x_2} g(z,a) \mathrm{d}z$ interpolates between $f_1+\int_{x_1}^{x_2} g_{\min}(z) \mathrm{d}z$ and $f_1+\int_{x_1}^{x_2} g_{\max}(z) \mathrm{d}z$. Specifically, we continuously decrease the integral of $g_{\max}$ by ``cutting" $g_{\max}$ at height $a$.}
    \label{fig:proof_completable}
\end{figure}
Let $S=\{(x_i,f_i^0,f_i^1)\}_{i=1,2}$ and $g_{\min/\max}$ be defined as in \eqref{eq:defgmin}, interpolating $\tilde S=\{(x_i,f_i^1\}_{i=1,2}$. Let $a\geq 0$ and define the following function
\begin{align}\label{eq:g(x,a)}
    g(x,a)=\min \{g_{\max}(x),\max\{g_{\min}(x),a\}\}.
\end{align}
By construction, $g$ is defined as the juxtaposition of functions in $\FLip{M,\alpha,+}$ on at most three intervals: either a constant function, either $g_{\min}$ or $g_{\max}$. Hence, since $\FLip{M,\alpha,+}$ is order 0-connectable, $g\in \FLip{M,\alpha,+}$. In addition, $g$ interpolates $\tilde S$ since $g_{\max}(x_i)=g_{\min}(x_i)=f_i^1$, $i=1,2$. Let \begin{align*}
        f(x,a):=f_1^0-\int_{-\infty}^{x_1} g(z,a) \mathrm{d}z +\int_{-\infty}^x g(z,a) \mathrm{d}z.
\end{align*}
By Lemma \ref{lem:fromgtof}, $f\in \int \FLip{M,\alpha,+}$, $f'(x_i,a)=f_i^1$, $i=1,2$, and $f(x_1,a)=f_1^0$. In addition, 
\begin{align*}
    f(x_2,a)=f_1^0+\int_{x_1}^{x_2} g(z,a) \mathrm{d}z.
\end{align*}
Since $f_1^0+\int_{x_1}^{x_2} g_{\min}(z) \mathrm{d}z\leq f_2^0\leq f_1^0+\int_{x_1}^{x_2} g_{\max}(z) \mathrm{d}z$ and 
$\int_{x_1}^{x_2}g(z,a) \mathrm{d}z$ depends continuously on $a$, with a minimum in $\int_{x_1}^{x_2} g_{\min}(z) \mathrm{d}z$ and a maximum in $\int_{x_1}^{x_2} g_{\max}(z) \mathrm{d}z$, there exists some $a\geq 0$ for which $f(x_2,a)=f_2^0$, hence $\int \FLip{M,\alpha,+}$ is extremally completable.
 \qed \end{proof}

\section{Proof of Theorems \ref{thm:quasi_self} and \ref{thm:IC_generalized_self_concordant}}\label{app:proof_theo_GSC}
\begin{proof}[Proof of Theorem \ref{thm:quasi_self}]
    Classes $\FLip{M,1,(+)}$ and $\int \FLip{M,1,(+)}$ satisfy all assumptions in Theorem \ref{lem:technique}. Hence, $S$ is interpolable if and only if it satisfies 
    \begin{align}
        |\log\Par{g_i}-\log\Par{g_j}|\leq M|x_i-x_j|,\label{condqsc2}
    \end{align} and $\forall x_i<x_j$:
    \begin{align}
    \begin{cases}
        f_j-f_i \geq\frac{g_i+g_j}{M}-\frac{2}{M}\sqrt{g_i g_j} e^{-\frac{M}{2}(x_j-x_i)}\\
        f_j-f_i\leq -\frac{g_i+g_j}{M}+\frac{2}{M}\sqrt{g_i g_j}e^{\frac{M}{2}(x_j-x_i)}
    \end{cases} \nonumber\\
        \Leftrightarrow  f_i-f_j\leq \frac{g_i+g_j}{M}-\frac{2}{M}\sqrt{g_i g_j}e^{-\frac{M}{2}(x_i-x_j)}.\label{eq:condsc_interl}
    \end{align}
    One can show that imposing \eqref{condqsc3} on all pairs is strictly equivalent to \eqref{condqsc2} and \eqref{eq:condsc_interl}.

    In addition, \eqref{condqsc3} implies \eqref{condqsc2} when applied to all pairs. Indeed, summing \eqref{condqsc3} imposed on $(i,j)$ and $(j,i)$ yields
    \begin{equation*}
        \frac{g_i + g_j}{\sqrt{g_i g_j}} \leq e^{-\frac M2 (x_j-x_i)} + e^{\frac M2 (x_j-x_i)} 
    \end{equation*}
    and using $\frac 1t + t \leq \frac{1}{\sqrt{s}} + \sqrt{s} \Leftrightarrow \frac{1}{s} \leq t^2 \leq s$ when $s\geq 1$, $t\geq 0$ with $t=\sqrt{\frac{g_j}{g_i}}$ and $s = e^{M(x_j-x_i)}$ yields
    \begin{equation}
        e^{-M(x_j-x_i)} \leq \frac{g_j}{g_i} \leq e^{M(x_j-x_i)},
    \end{equation}
    which is equivalent to \eqref{condqsc2}.
 \qed \end{proof}

\begin{proof}[Proof of Theorem \ref{thm:IC_generalized_self_concordant}]
Classes $\FLip{M,\alpha,+}$ and $\int \FLip{M,\alpha,+}$ satisfy all assumptions in Theorem \ref{lem:technique}. Hence, $S$ is interpolable if and only if it satisfies \eqref{eq:interp1_level1}, and $\forall x_i<x_j$:
\begin{align*}
    &\bullet \text{ Integration of $f_{\min}(x)$, Case 1: }\\
    &\text{If $\alpha > 1$, or if $\alpha <1$ and $\tilde g_i+\tilde g_j\geq M(x_j-x_i)=\frac{\ba}{|\ba|} M(x_j-x_i)$, or considering $\FLip{M,0}$:}\\   
    &f_j-f_i\geq \frac{\ba}{|\ba|M(\ba+1)} \left( \tilde g_i^{\ba+1}+\tilde g_j^{\ba+1}-\frac{1}{2^\ba}(\tilde g_i+\tilde g_j-\frac{\ba}{|\ba|}M(x_j-x_i))^{\ba+1} \right)\\
    &\bullet\text{ Integration of $f_{\min}(x)$, Case 2: If $\alpha <1$ and $\tilde g_i+\tilde g_j\leq M(x_j-x_i)$:}\\   
    &f_j-f_i\geq \frac{1}{M(\ba+1)} \left( \tilde g_i^{\ba+1}+\tilde g_j^{\ba+1}\right)\\
    &\bullet\text{ Integration of $f_{\max}(x)$, Case 1:}\\
    &\text{If $\alpha < 1$, or if $\alpha >1$ and $\tilde g_i+\tilde g_j\geq M(x_j-x_i)=\frac{\ba}{|\ba|} M(x_i-x_j)$:}\\   
    &f_j-f_i\leq \frac{\ba}{|\ba|M(\ba+1)} \left( -\tilde g_i^{\ba+1}-\tilde g_j^{\ba+1}+\frac{1}{2^\ba}(\tilde g_i+\tilde g_j+\frac{\ba}{|\ba|}M(x_j-x_i))^{\ba+1} \right)\\
    \Leftrightarrow& f_i-f_j\geq \frac{\ba}{|\ba|M(\ba+1)} \left( \tilde g_i^{\ba+1}+\tilde g_j^{\ba+1}-\frac{1}{2^\ba}(\tilde g_i+\tilde g_j-\frac{\ba}{|\ba|}M(x_i-x_j))^{\ba+1} \right).
\end{align*}

The second case for $f_{\max}$ leads to no further constraint. Considering $\FLip{M,\alpha,+}$, imposing \eqref{eq:interp1_level2} and \eqref{eq:interp1_level3} on all pairs is strictly equivalent to these conditions, since $\frac{\ba}{|\ba|}=-1$ if $\alpha>1$, $\frac{\ba}{|\ba|}=1$ otherwise.

Considering $\FLip{M,0}$, there are no subcases to consider, and imposing \eqref{eq:interp1_level2} on all pairs is strictly equivalent to these conditions. Furthermore, satisfaction of \eqref{eq:interp1_level2} on both pairs $(i,j)$ and $(j,i)$ implies satisfaction of $|g_i-g_j|\leq M|x_j-x_i|$, since adding both inequalities gives:
\begin{align*}
    &0\geq g_i^2+g_j^2-\frac{1}{4}(g_i+g_j-M(x_j-x_i))^2-\frac{1}{4}(g_i+g_j+M(x_j-x_i))^2\\
    \Leftrightarrow& 0\geq \frac{(g_i-g_j)^2}{2}-\frac{M^2}{2} (x_j-x_i)^2 \Leftrightarrow |g_i-g_j|\leq M|x_i-x_j|.
\end{align*} 
 \qed \end{proof}
\section{Proof of Theorem \ref{thm:H-intrp}}\label{app:thmH}
For the sake of readability, we let $\Delta x_{ij}=x_j-x_i,\
    \Delta x_{xi}=x-x_i,\
    \Delta h_{ij}=h_j-h_i,$ $T^g_{ij}=g_j-g_i-h_i(x_j-x_i),$ and $ T^f_{ij}=f_j-f_i-g_i(x_j-x_i)-\frac{h_i}{2}(x_j-x_i)^2,$ and use
\begin{align}
     \Delta x_{ij}&=-\Delta x_{ji}, \label{xijji}
    \\ \Delta h_{ij}&=-\Delta h_{ji}, \label{hijji}
    \\ T^g_{ij}&=-T^g_{ji}+\Delta h_{ij} \Delta x_{ij}, \label{gijji}
    \\ T^f_{ij}&=-T^f_{ji}+T^g_{ij}\Delta x_{ij}-\frac{1}{2} \Delta h_{ij} \Delta x_{ij}^2. \label{fijji}
\end{align}

By Proposition  \ref{prop:properties}, $\HLip{M}$ is convex and order $2$ connectable. We now show that $\int \F_M$ is extremally interpolable, and provide its extremal interpolants:
\begin{lemma}[Extremal interpolants of $\HLip{M}$] \label{lem:min_max_grad}
Consider an interval $[x_1,x_2]$ and a $\HLip{M}$-interpolable pair of points $S=\{(x_i,f_i,g_i,h_i)\}_{i=1,2}$. Then, in $[x_1,x_2]$, the expressions of $g_{\min}$ and $g_{\max}$ as defined in \eqref{eq:defgmin} are given by:
\\
When $\dhh\neq \pm M\dx,$
\begin{align}
    g_{\min}(x)=\begin{cases}
        g_1+ h_1\Delta x_{x1} -\frac M2 \Delta x_{x1}^2 & \text{if } x\in [x_1,\yun] \\
        g_1+h_1\Delta x_{x1} -\frac M2 (\yun-x_1)^2&\\ +\frac M2 (x^2-\yun^2)+M(x_1-2\yun)(x-\yun) & \text{if } x\in [\yun,\ydeux] \\
        g_2+h_2\Delta x_{x2}-\frac M2\Delta x_{x2}^2 & \text{if } x\in [\ydeux,x_2],
    \end{cases} \label{def:gmin}
\end{align}
where
\begin{align*}
        \yun & = x_2-\frac{\dg +\frac{M}{2} \dx^2}{\dhh+M\dx} -\frac{\dhh+M\dx}{4M} \\
    \ydeux & = x_2-\frac{\dg +\frac{M}{2} \dx^2}{\dhh+M\dx} +\frac{\dhh+M\dx}{4M},
\end{align*}
and 
\begin{align}
    g_{\max}(x)=\begin{cases}
        g_1+h_1\Delta x_{x1}+\frac M2 \Delta x_{x1}^2 & \text{if } x\in [x_1,\zun] \\
        g_1+h_1\Delta x_{x1}+\frac M2 (\zun-x_1)^2  &\\-\frac M2 (x^2-\zun^2)-M(x_1-2\zun)(x-\zun) & \text{if } x\in [\zun,\zdeux] \\
        g_2+h_2\Delta x_{x2}+\frac M2\Delta x_{x2}^2& \text{if } x\in [\zdeux,x_2],
    \end{cases}
\end{align}
where 
\begin{align*}
        \zun & = x_2 +\frac{\dg -\frac{M}{2} \dx^2}{-\dhh+M\dx} -\frac{(-\dhh+M\dx)}{4M}\\
    \zdeux & = x_2 +\frac{\dg -\frac{M}{2} \dx^2}{-\dhh+M\dx} +\frac{(-\dhh+M\dx)}{4M}.
\end{align*}
When $\dhh= \pm M\dx:$
\begin{align}
    g_{\max}(x)=g_{\min}(x)&=
        g_1+h_1\Delta x_{x1}\pm\frac M2 \Delta x_{x1}^2=g_2+h_2\Delta x_{x2}\pm\frac M2 \Delta x_{x2}^2.
\end{align}
\label{prop:extremalgrad}
\begin{proof}
Suppose $\dhh\neq \pm \dx$. For any $x\in [x_1,x_2]$, let
    \begin{equation}
    g'_{\min}(x) = 
    \begin{cases}
        h_1-M\Delta x_{x1} & \text{if } x\in [x_1,\yun] \\
        h_1+M(x+x_1-2\yun) & \text{if } x\in [\yun,\ydeux] \\
        h_2-M\Delta x_{x2} & \text{if } x\in [\ydeux,x_2],
    \end{cases}
\end{equation}
It holds that $g'_{\min}$ is the first derivative of $g_{\min}$ since $\forall x\in [x_1,x_2], \ g_{\min}(x)=\int_{x_1}^x g'_{\min}(t)\mathrm{d}t$. If $x \in [x_1,\ydeux]$, this is straightforward, and if $x\in[\ydeux,x_2]$ it follows from:
\begin{align*}
    \int_{x_1}^x g'_{\min}(t)\mathrm{d}t&=g_1+h_1(\ydeux-x_1)-\frac{M}{2} (\yun-x_1)^2+\frac{M}{2}(\ydeux^2-\yun^2)\\&\quad+M(x_1-2\yun)(\ydeux-\yun) +h_2(x-\ydeux)-\frac{M}{2}(x^2-\ydeux^2)+Mx_2(x-\ydeux)\\
    &=g_2+h_2\Delta x_{x2}-\frac{M}{2} \Delta x_{x2}^2-\dg -\dhh (\ydeux-x_2)\\&\quad-M\left( (\ydeux-x_2)\dx +\frac{\dx^2}{2}-(\ydeux-\yun)^2\right)\\
    &=g_2+h_2\Delta x_{x2}-\frac{M}{2} \Delta x_{x2}^2-(\dg+M\frac{\dx^2}{2}) \\&\quad-(\dhh+M\dx) (\ydeux-x_2)+\frac{\dhh+M\dx}{4M}
    \\&    =g_2+h_2\Delta x_{x2}-\frac{M}{2} \Delta x_{x2}^2.
\end{align*}

We first observe that $g_{\min}$ interpolates $S$ since  \begin{align*}
    g_{\min}(x_1)=g_1, \ g_{\min}(x_2)=g_2, \ g'_{\min}(x_1)=h_1 \text{ and }g'_{\min}(x_2)=h_2.
\end{align*}
In addition, $g_{\min}\in \int\F_{M}$, since $g'_{\min}\in \F_M$. Indeed,  (i) $g'_{\min}\in\C^0$ and is piecewise $\C^1$, and (ii) it is a linear function by part, of slope $M$. Inclusion in $\F_M$ follows from order $0$-connectability of $\F_M$.

It remains to show that $g_{\min}$ and $g'_{\min}$ are the minimal $g$ and associated $h$ satisfying \eqref{eq:cismooth} with respect to $S$, which completes the proof. To this end, consider the following optimization problem: $g_{\min}(x)=$
\begin{equation}
\left\{\begin{aligned}
\min_{g,h} \quad & g \\
\textrm{s.t.} \quad & g_i - g  \geq \frac{1}{4M}(h-h_i)^2 - \frac{M}{4}\Delta x_{xi}^2-\frac12(h+h_i)\Delta x_{xi}, \ i=1,2 &\quad (\mu_i) \\
                    & g - g_i  \geq \frac{1}{4M}(h_i-h)^2 - \frac{M}{4}\Delta x_{xi}^2+\frac12(h_i+h)\Delta x_{xi}, \ i=1,2 &\quad (\lambda_i)
\end{aligned} \right.\label{eq:g_min}
\end{equation} 
where $\mu_i$, $\lambda_i$ are the associated dual variables. We consider its KKT conditions.
\begin{alignat}{2}
    &\text{Stat. cond. in $g$: }&&1+\sum_i(\mu_i-\lambda_i)=0\nonumber\\
    &\text{Stat. cond. in $h$: }&&\sum_i(\mu_i(\frac{1}{2M}(h-h_i)-\frac{1}{2}\Delta x_{xi})\nonumber\\
    & &&+\lambda_i(\frac{1}{2M}(h-h_i)+\frac{1}{2}\Delta x_{xi}))=0\nonumber\\
    &\text{Dual feas.: }&&\mu_i,\lambda_i\geq 0\nonumber\\&\text{Primal feas.: }&&g_i - g  \geq \frac{1}{4M}(h-h_i)^2 - \frac{M}{4}\Delta x_{xi}^2-\frac12(h+h_i)\Delta x_{xi}, \ i=1,2 \label{1}
    \\&& &g - g_i  \geq \frac{1}{4M}(h_i-h)^2 - \frac{M}{4}\Delta x_{xi}^2+\frac12(h_i+h)\Delta x_{xi}, \ i=1,2\label{3}\\
    &\text{Slackness: }&& \mu_i (g_i - g  -( \frac{1}{4M}(h-h_i)^2 - \frac{M}{4}\Delta x_{xi}^2-\frac12(h+h_i)\Delta x_{xi}))=0\nonumber
    \\
    && &\lambda_i(g - g_i - (\frac{1}{4M}(h_i-h)^2 - \frac{M}{4}\Delta x_{xi}^2+\frac12(h_i+h)\Delta x_{xi} ))=0.\nonumber
\end{alignat}
Let $x\in [x_1,\yun]$. Then, setting all dual coefficients zeros except for $\lambda_1=1$, $h=h_1-M\Delta x_{x1}$ and $g=g_1+h_1\Delta x_{x1}-\frac{M}{2}\Delta x_{x1}^2$ is a valid solution. Indeed, all conditions follow directly from the definition of $g$ and $h$, except for \eqref{1}, $i=2$, which becomes equivalent to \eqref{eq:cismooth} and is thus satisfied by assumption on $S$, and \eqref{3}, $i=2$, which becomes:
\begin{align*}
    &\dg-\frac{3M}{4}\Delta x_{x1}^2-\frac M4 \Delta x_{x2}-\frac M2 \Delta x_{x1}\Delta x_{x2}+\frac{\dhh^2}{4M}+\frac{\dhh}{2} (2x-x_1-x_2)\leq 0\\
    &\Leftrightarrow \Delta x_{x2}(\dhh+M\dx)\leq -\dg-\frac{\dhh^2}{4M}-\frac{3M\dx^2}{4}-\frac{\dhh\dx}{2}\\
    &\Leftrightarrow x\leq \yun.
\end{align*}
Let $x\in [\ydeux,x_2]$. Then, setting all dual coefficients zeros except for $\lambda_2=1$, $h=h_2-M\Delta x_{x2}$ and $g=g_2+h_2\Delta x_{x2}-\frac{M}{2}\Delta x_{x2}^2$ is a valid solution. Again, all conditions follow directly from the definition of $g$ and $h$, except for \eqref{1}, $i=1$, which becomes equivalent to \eqref{eq:cismooth} as applied to the pair $(j,i)$ and is thus satisfied by assumption on $S$, and \eqref{3}, $i=1$, which becomes:
\begin{align*}
    &\dg-\frac{3M}{4}\Delta x_{x2}^2+\frac M4 \Delta x_{x1}+\frac M2 \Delta x_{x1}\Delta x_{x2}-\frac{\dhh^2}{4M}+\frac{\dhh}{2} (2x+\frac{x_1}{2}-\frac{3x_2}{2})\geq 0\\
    &\Leftrightarrow \Delta x_{x2}(\dhh+M\dx)\geq -\dg+\frac{\dhh^2}{4M}-\frac{M\dx^2}{4}+\frac{\dhh\dx}{2}\\
    &\Leftrightarrow x\geq \ydeux.
\end{align*}

Else, setting $\mu_1=\mu_2=0$, \begin{align*}
    &\lambda_2=\frac{2M(x-\yun)}{\dhh+M\dx}\underset{x\geq \yun, \eqref{eq:cond1}}{\geq} 0, \ \lambda_1=1-\lambda_0\underset{x\leq \ydeux}{\geq}1- \frac{2M(\ydeux-\yun)}{\dhh+M\dx}\underset{\text{Def. of } y_2}{=}0\\&h=h_1+M(x+x_1-2\yun)
\end{align*} and \begin{align*}
    g&=g_1+\frac{M}{4}(x+x_1-2\yun)^2-\frac M4\Delta x_{x1}^2+\frac{1}{2}(2h_1+M(x+x_1-2\yun))\Delta x_{x1}\\
    &=g_1+h_1\Delta x_{x1} -\frac M2 (\yun-x_1)^2 +\frac M2 (x^2-\yun^2)+M(x_1-2\yun)(x-\yun),
\end{align*} is a valid solution. Indeed, all conditions follow directly from the definition of $g$, $h$, $\lambda_0$ and $\lambda_1$ except for \eqref{1}, $i=1,2$, which become equivalent to \eqref{eq:cond1} and are thus satisfied by assumption on $S$. 

Suppose now $\dhh=-M\dx$, which implies $\dg=-\frac{M}{2}\dx^2$ since \eqref{eq:cismooth} applied to both pairs becomes:
\begin{align*}
    \dg&\geq -\frac{M}{2}\dx^2\\
    T_{21}^g&\geq -\frac{M}{2}\dx^2 \underset{\eqref{gijji}}{\Leftrightarrow} \dg \leq \dhh \dx+\frac{M}{2}\dx^2=-\frac{M}{2}\dx^2.
\end{align*}
In this case, let $g'_{\min}(x)=h_1-M\Delta x_{x1}=h_2-M\Delta x_{x2}$. Therefore, $g'_{\min}$ is the first derivative of $g_{\min}$, and this function interpolates $S$. Moreover, in this case, $\mu_0=\mu_1=\lambda_1=0$, $\lambda_0=1$, \begin{align*}
    g&=g_1+h_1\Delta x_{x1}-\frac{M}{2}\Delta x_{x1}^2\underset{\dg=-\frac{M}{2}\dx^2}{=}g_2+h_2\Delta x_{x2}-\frac{M}{2}\Delta x_{x2}^2
    \\h&=h_1-M\Delta x_{x1}=h_2-M\Delta x_{x2}.
\end{align*} are always a solution satisfying the KKT conditions associated with Problem \eqref{eq:g_min}. Indeed, all conditions follow directly from these two definitions of $g$ and $h$.

The computation of $g_{\max}$ follows the same argument, except with the associated derivative being:
    \begin{equation}
    g'_{\max}(x) = 
    \begin{cases}
        h_1+M\Delta x_{x1} & \text{if } x\in [x_1,\zun] \\
        h_1-M(x+x_1-2\zun) & \text{if } x\in [\zun,\zdeux] \\
        h_2+M\Delta x_{x2} & \text{if } x\in [\zdeux,x_2].
    \end{cases}
\end{equation}
 \qed \end{proof}
\end{lemma}
Exploiting Theorem  \ref{lem:technique} and Lemma \ref{prop:extremalgrad}, we are now ready to proceed to the proof of Theorem \ref{thm:H-intrp}. Indeed, in Proposition \ref{prop:CIextension}, we prove necessity and sufficiency of 
 \eqref{eq:cond1}, \eqref{eq:cond2}, \eqref{eq:cond22sup} and \eqref{eq:cismooth} to $\HLip{M}$-interpolability of a set $S=\{(x_i,g_i,h_i,f_i)\}_{i=1,2}$, relying on the integration of the extremal gradients derived in Lemma \ref{lem:min_max_grad}. By Lemma \ref{lem:full_expr}, these conditions are equivalent to the conditions of Theorem \ref{thm:H-intrp}. And, by Lemma \ref{lem:interval}, these interpolation conditions hold for any set $\quadr$. Therefore, Proposition \ref{prop:CIextension} concludes the proof of Theorem \ref{thm:H-intrp}.

\begin{proposition}\label{prop:CIextension}
    A set $S=\{(x_i,g_i,h_i,f_i)\}_{i=1,2}$ is $\HLip{M}$-interpolable if and only if, $\forall i,j=1,2 $, 
    \\
   1) If $ \Delta h_{ij} + M| \Delta x_{ij}| \neq 0$, then
   \begin{align}
        | \Delta h_{ij}|\leq& M| \Delta x_{ij}|, \label{eq:c1} \\ 
        T^f_{ij} \geq &-\frac{M}{6} | \Delta x_{ij}|^3 \label{eq:c2}+ \frac{\Par{T^g_{ij} +\frac{M}{2}  \Delta x_{ij} | \Delta x_{ij}|}^2}{2\Par{ \Delta h_{ij}+M| \Delta x_{ij}|}} + \frac{\Par{ \Delta h_{ij} + M| \Delta x_{ij}|}^3}{96M^2}, \\
        T^f_{ij}   \leq &\frac{M}{6} | \Delta x_{ij}|^3 \label{eq:c3}- \frac{\Par{T^g_{ij} -\frac{M}{2}  \Delta x_{ij} | \Delta x_{ij}|}^2}{2\Par{M| \Delta x_{ij}|- \Delta h_{ij}}} -\frac{\Par{ M| \Delta x_{ij}|- \Delta h_{ij}}^3}{96M^2}, \\
        T^g_{ij}  \geq &-\frac{M}{2} \Delta x_{ij}^2 +\frac{1}{4M} (  \Delta h_{ij} +M \Delta x_{ij})^2,\label{eq:c4}
    \end{align}
   \\
   2) If $ \Delta h_{ij} \pm M| \Delta x_{ij}|=0$, then
   \begin{align}
       T^g_{ij}&=\mp \frac{M}{2} \Delta x_{ij}| \Delta x_{ij}| \\
       T^f_{ij}&=\mp \frac{M}{6}| \Delta x_{ij}|^3
   \end{align}
\label{prop:H-intrp}  

\begin{proof}
Suppose $x_1<x_2$. By Lemma \ref{lem:technique}, it holds that $S$ is $\FLip{M}$-interpolable if and only if it satisfies $|\dhhi|\leq M |\dxi|$ (Proposition \ref{prop:def2}), \eqref{eq:c4} and
\begin{align}
    \int_{x_1}^{x_2}g_{\min}(x)\mathrm{d}x \leq f_2-f_1\leq  \int_{x_1}^{x_2} g_{\max} (x)\mathrm{d}x,
\end{align}
 where $g_{\min}$ and $g_{\max}$ are given in Lemma \ref{prop:extremalgrad}.

Consider first the case $\dhh\neq \pm M|\dx|$. Integration of extremal interpolants gives:
\begin{align}
\df &\geq -\frac{M}{6} \dx^3 +\frac{(\dg+\frac{M}{2}\dx^2)^2}{2(\dhh+ M \dx) } + \frac{(\dhh+ M \dx)^3}{96M^2}\label{condijinf} \\
\df &\leq \frac{M}{6} \dx^3 -\frac{(\dg-\frac{M}{2}\dx^2)^2}{2( M \dx-\dhh) } - \frac{( M \dx-\dhh)^3}{96M^2} . \label{condijsup}
       \end{align}
Suppose now $x_2 < x_1$. Then, taking Lemma \ref{prop:extremalgrad} with $x_2$ and $x_1$ permuted and following the same argument as in the case $x_1<x_2$, we obtain the following inequalities to be satisfied:
\begin{align}
T^f_{21}&\geq-\frac{M}{6} \Delta x_{21}^3 +\frac{(T^g_{21}+\frac{M}{2}\Delta x_{21}^2)^2}{2(\Delta h_{21}+ M \Delta x_{21}) } + \frac{ (\Delta h_{21}+ M \Delta x_{21})^3}{96M^2}\label{eq:21inf}\\
T^f_{21}&\leq \frac{M}{6} \Delta x_{21}^3 +\frac{(T^g_{21}-\frac{M}{2}\Delta x_{21}^2)^2}{2(\Delta h_{21}- M \Delta x_{21}) } + \frac{(\Delta h_{21}- M \Delta x_{21})^3}{96M^2} .\label{eq:21sup}
       \end{align}
       By \eqref{xijji}, \eqref{hijji}, \eqref{gijji} and \eqref{fijji}, Condition \eqref{eq:21inf} becomes:
\begin{align}
&-\df+\dg\dx-\frac{1}{2}\dhh \dx^2\geq \frac{M}{6} \dx^3 - \frac{(\dhh+ M \dx)^3}{96M^2}   \nonumber\\&\hspace{5cm}+\frac{(-\dg+\dhh\dx+\frac{M}{2}\dx^2)^2}{-2(\dhh+ M \dx) } \nonumber
\\\Leftrightarrow&
-\df+\dg\dx-\frac{1}{2}\dhh \dx^2\geq \frac{M}{6} \dx^3 -\frac{(\dg+\frac{M}{2}\dx^2)^2}{2(\dhh+ M \dx) } \nonumber \\&\qquad \qquad\qquad \qquad \qquad\qquad \qquad+\dg\dx-\frac{1}{2}\dhh \dx^2 - \frac{(\dhh+ M \dx)^3}{96M^2} \nonumber
\\
\Leftrightarrow&
\ \df\leq - \frac{M}{6} \dx^3 +\frac{(\dg+\frac{M}{2}\dx^2)^2}{2(\dhh+M \dx) }+ \frac{(\dhh+ M \dx)^3}{96M^2} 
\label{condjiinf}
\end{align}
Similarly, \eqref{eq:21sup} becomes:
\begin{align}
    -&\df+\dg\dx-\frac{1}{2}\dhh \dx^2 \leq -\frac{M}{6} \dx^3 - \frac{(\dhh- M \dx)^3}{96M^2} \nonumber\\
    & +\frac{(-\dg+\dhh \dx-\frac{M}{2}\dx^2)^2}{-2(\dhh- M \dx) }\nonumber\\
    \Leftrightarrow &\df\geq \frac{M}{6} \dx^3 +\frac{(\dg-\frac{M}{2}\dx^2)^2}{2(\dhh- M \dx) } + \frac{(\dhh- M \dx)^3}{96M^2}.\label{condjisup}
\end{align}
Conditions \eqref{condijinf} and \eqref{condjisup} are equivalent to a single constraint, independently of the sign of $\dx$:
    \begin{align*}
         \df&\geq -\frac{M}{6} |\dx|^3 + \frac{\Par{\dg +\frac{M}{2} \dx |\dx|}^2}{2\Par{\dhh+M|\dx|}} + \frac{\Par{\dhh + M|\dx|}^3}{96M^2}.  
   \end{align*}
   Similarly, Conditions \eqref{condijsup} and \eqref{condjiinf} can be expressed as:
   \begin{align*}
       \df &  \leq \frac{M}{6} |\dx|^3 - \frac{\Par{\dg -\frac{M}{2} \dx |\dx|}^2}{2\Par{M|\dx|-\dhh}} -\frac{\Par{ M|\dx|-\dhh}^3}{96M^2}.
   \end{align*}
   We have obtained interpolation conditions to be satisfied by the pair $(1,2)$ for $S$ to be $\HLip{M}$-interpolable. Extending the argument to the pair $(2,1)$ allows concluding that a $\HLip{M}$-interpolable set $S=\{(x_i,g_i,h_i,f_i)\}_{i=1,2}$ necessarily satisfies, for all $i,j=1,2$, Conditions \eqref{eq:c2} and \eqref{eq:c3}.

Suppose now $x_2>x_1$ and $\dhh=\pm M\dx$,\ $\dg=\pm \frac{M}{2}\dx^2$. Then, 
\begin{align}
f_2-f_1=\int_{x_1}^{x_2}g_{\min}(x)\mathrm{d}x
\Leftrightarrow \df= \pm \frac{M}{6} \dx^3\label{equal1}
       \end{align}
       
       Finally, suppose $x_1<x_2$ and $\dhh=\pm M \dx \Leftrightarrow \Delta h_{21}=\pm M \Delta x_{21}$, implying $\dg=\mp\frac{M}{2}\dx^2$. By Lemma \ref{prop:extremalgrad} applied to the pair $(2,1)$, it holds that 
       \begin{align}
           T^f_{21}= \pm \frac{M}{6} \Delta x_{21} ^3\underset{\eqref{gijji},\eqref{hijji}}{\Leftrightarrow} \df=\mp \dx^3.\label{equal2}
       \end{align}
      Combining  \eqref{equal1} and \eqref{equal2}, it holds that $\forall i,j=1,2$ if $ \Delta h_{ij}=\pm M| \Delta x_{ij}|$, then necessarily $T^f_{ij}=\pm \frac{M}{6}| \Delta x_{ij}|^3$ (and $T^g_{ij}=\pm \frac{M}{2} \Delta x_{ij}| \Delta x_{ij}|$) for $S$ to be $\HLip{M}$-interpolable.
 \qed \end{proof}
\end{proposition}

\section{Proof of Lemma \ref{lem:full_expr}}\label{app:full_expr}
\begin{proof}[Proof of Lemma \ref{lem:full_expr}]
Condition \eqref{eq:cond2} evaluated at $(j,i)$ is equivalent to Condition \eqref{eq:cond22sup} evaluated at $(i,j)$:
   \begin{align*}
       &T^f_{ji}\geq -\frac{M}{6} |\Delta x_{ji}|^3 + \frac{\Par{T^g_{ji} +\frac{M}{2} \Delta x_{ji} |\Delta x_{ji}|}^2}{2\Par{\Delta h_{ji}+M|\Delta x_{ji}|}} + \frac{\Par{\Delta h_{ji} + M|\Delta x_{ji}|}^3}{96M^2}
       \\
       \Leftrightarrow&
       -T^f_{ij}+T^g_{ij} \Delta x_{ij}-\frac{1}{2} \Delta h_{ij} \Delta x_{ij}^2\geq -\frac{M}{6} | \Delta x_{ij}|^3 - \frac{\Par{-T^g_{ij}+ \Delta h_{ij} \Delta x_{ij} -\frac{M}{2}  \Delta x_{ij} | \Delta x_{ij}|}^2}{2\Par{ \Delta h_{ij}-M| \Delta x_{ij}|}} \\&\hspace{5cm}- \frac{\Par{ \Delta h_{ij} - M| \Delta x_{ij}|}^3}{96M^2} \\
       \Leftrightarrow &T^f_{ij} \leq \frac{M}{6} | \Delta x_{ij}|^3 + \frac{\Par{T^g_{ij} -\frac{M}{2}  \Delta x_{ij} | \Delta x_{ij}|}^2}{2\Par{ \Delta h_{ij}-M| \Delta x_{ij}|}}+T^g_{ij} \Delta x_{ij}-\frac{1}{2} \Delta h_{ij} \Delta x_{ij}^2-T^g_{ij} \Delta x_{ij}\\&\quad -\frac{1}{2} \Delta h_{ij} \Delta x_{ij}^2 + \frac{\Par{ \Delta h_{ij} - M| \Delta x_{ij}|}^3}{96M^2}\\
       \Leftrightarrow &T^f_{ij} \leq \frac{M}{6} | \Delta x_{ij}|^3 -\frac{\Par{T^g_{ij} -\frac{M}{2}  \Delta x_{ij} | \Delta x_{ij}|}^2}{2\Par{M| \Delta x_{ij}|- \Delta h_{ij}}}-\frac{\Par{M| \Delta x_{ij}|- \Delta h_{ij}}^3}{96M^2}.
   \end{align*}
Hence, if $S$ satisfies \eqref{eq:cond2} $\forall i,j=1,2$, then it satisfies \eqref{eq:cond22sup} $\forall i,j=1,2$.    

We now show that satisfaction $\forall i,j=1,2$ of \eqref{eq:cond1} and \eqref{eq:cond2} (hence satisfaction of \eqref{eq:cond22sup}) imply satisfaction of \eqref{eq:cismooth} $\forall i,j=1,2$. Observe that \eqref{eq:cismooth} is satisfied by both pairs $(i,j)$ and $(j,i)$ if and only if
\begin{align}
   & T^g_{ij}- K\geq 0 \text{ and } T^g_{ji}-K\geq 0 \underset{\eqref{gijji}}{\Leftrightarrow} -T^g_{ij}+ \Delta h_{ij} \Delta x_{ij}-K\geq 0,\label{blop1}
\end{align}
where $ K = -\frac{M}{2} \Delta x_{ij}^2 +\frac{1}{4M} (  \Delta h_{ij} +M \Delta x_{ij})^2$. Condition \eqref{blop1} is equivalent to
\begin{align}
    (T^g_{ij}- K)(-T^g_{ij}+ \Delta h_{ij} \Delta x_{ij}-K)\geq 0,\label{blop}
\end{align}where the equivalence follows from the fact that $T^g_{ij}< K$ and $-T^g_{ij}+ \Delta h_{ij} \Delta x_{ij}<K$ is in contradiction with Condition \eqref{eq:cond1}, since it implies:
\begin{align*}
    2K> \Delta h_{ij} \Delta x_{ij} \Leftrightarrow -\frac{M}{2} \Delta x_{ij}^2 +\frac{1}{2M}   \Delta h_{ij} ^2>0.
\end{align*}
Substracting \eqref{eq:cond2} and \eqref{eq:cond22sup} yields:
\begin{align*}
&\frac{M}{3}| \Delta x_{ij}|^3\geq\frac{\Par{T^g_{ij} -\frac{M}{2}  \Delta x_{ij} | \Delta x_{ij}|}^2}{2\Par{M| \Delta x_{ij}|- \Delta h_{ij}}}+\frac{\Par{T^g_{ij}+\frac{M}{2}  \Delta x_{ij} | \Delta x_{ij}|}^2}{2\Par{ \Delta h_{ij}+M| \Delta x_{ij}|}}+ \\&\hspace{3cm}\frac{\Par{M| \Delta x_{ij}|- \Delta h_{ij}}^3+\Par{ \Delta h_{ij} + M| \Delta x_{ij}|}^3}{96M^2}
\\
\Leftrightarrow& \frac{M| \Delta x_{ij}|}{(M| \Delta x_{ij}|)^2- \Delta h_{ij}^2}\left(-(T^g_{ij})^2+T^g_{ij} \Delta h_{ij} \Delta x_{ij}-\frac{M^2 \Delta x_{ij}^2}{16}-\frac{3 \Delta h_{ij}^2 \Delta x_{ij}^2}{8}
+\frac{ \Delta h_{ij}^4}{16M^2}\right)\geq 0\\
\underset{\eqref{eq:cond1}}{\Leftrightarrow}& (T^g_{ij}-\frac{ \Delta h_{ij}^2}{4M}+\frac{M}{4}\dx^2-\frac{ \Delta h_{ij}  \Delta x_{ij}}{2})(-T^g_{ij}-\frac{ \Delta h_{ij}^2}{4M}+\frac{M}{4} \Delta x_{ij}^2+\frac{ \Delta h_{ij}  \Delta x_{ij}}{2})\geq 0
\\
\Leftrightarrow &(T^g_{ij} - K)(-T^g_{ij} + \Delta h_{ij} \Delta x_{ij}-K ) \geq 0,
\end{align*}
and \eqref{eq:cismooth} is satisfied.

Finally, since $ \Delta h_{ij}\geq -M| \Delta x_{ij}|$, \eqref{eq:cond2} implies $T^f_{ij}\geq-\frac{M}{6} | \Delta x_{ij}|^3$, and since $ \Delta h_{ij}\leq M| \Delta x_{ij}|$, \eqref{eq:cond22sup} implies $T^f_{ij}\leq\frac{M}{6} | \Delta x_{ij}|^3$.
 \qed \end{proof}
\begin{remark}
Conditions \eqref{eq:cismooth} and \eqref{eq:cond22sup} can be removed from the final interpolation conditions since they are implied by the two remaining conditions. On the contrary, \eqref{eq:cond1} cannot be removed from the interpolation conditions since it is not implied by \eqref{eq:cond2}. Indeed, consider for example the set $S=\{(0,0,0,0),(\frac14,\frac13,0,0)\}$ satisfying \eqref{eq:cond2} but not \eqref{eq:cond1} for $M=1$.
\end{remark}

\section{Improved univariate descent lemma of CNM is tight}\label{app:improved_descent_lemma_is_tight}
The univariate descent lemma \eqref{eq:improved_descent_lemma} is tight and attained by the function $f(x) = M\frac{ x^3}{6} - \frac{x^2}{2}$ and $x_0 = 0$. The Cubic Regularized Newton iteration \eqref{eq:CNM} on $f$ from $x_0$ is
\begin{align*}
    x_1 & = \mathrm{Arg} \min_x f(x_0) + f'(x_0)(x-x_0) + \frac12 f''(x_0)(x-x_0)^2 + \frac{M}{6} |x-x_0|^3 \\
     & = \mathrm{Arg} \min_x \frac12 (-x^2 + \frac{M}{3} |x|^3)  = \pm \frac{2}{M}
\end{align*}
where both signs are global minimum of the cubic bound. Let $x_1=-\frac{2}{M}$. We have $f(x_0) = 0$, $f(x_1) = -\frac{10}{3M^2}$, $f'(x_1) = \frac{4}{M}$ and therefore $ f(x_0) - f(x_1) = \frac{5M}{12}\Par{\frac{|f'(x_1)|}{M}}^{\frac{3}{2}} = \frac{10}{3M^2}$.

\end{document}